\documentclass[11pt]{article}

    \usepackage{amsthm}
    \usepackage{amsmath}
    \usepackage{amssymb}
    \usepackage[margin=1in]{geometry}
    \usepackage{enumerate}
    \usepackage{hyperref}
    \usepackage{mathrsfs}
    \usepackage{color}
    \usepackage{bm}
    \usepackage{mathtools}
    \usepackage{dsfont}
\title{Sparse power spectrum for passive scalars under Pierrehumbert flow}
\author{Alex Blumenthal\thanks{A.B. was supported by National Science Foundation grant DMS-2237360 and a Sloan Fellowship.
} ~and William Cooperman\footnote{W.C. was supported by National Science Foundation grant DMS-2303355.
}}
\date{\today}

    \newtheorem{theorem}{Theorem}
    \newtheorem{corollary}[theorem]{Corollary}
    \newtheorem{lemma}[theorem]{Lemma}
    \newtheorem{proposition}[theorem]{Proposition}
    
    \newtheorem{letteredtheorem}{Theorem}
    \newtheorem{letteredcorollary}[letteredtheorem]{Corollary}

    \theoremstyle{definition}
    
    \newtheorem{definition}[theorem]{Definition}
    \newtheorem{remark}[theorem]{Remark}
    
    \newtheorem{claim}[theorem]{Claim}
    
    \newtheorem{setting}{Setting}

\newcommand{\C}{\mathbb{C}}

\newcommand{\E}{\mathbb{E}}

\newcommand{\N}{\mathbb{N}}
\renewcommand{\P}{\mathbb{P}}
\newcommand{\R}{\mathbb{R}}

\newcommand{\Z}{\mathbb{Z}}

\newcommand{\Fc}{\mathcal{F}}

\newcommand{\Lc}{\mathcal{L}}

\newcommand{\Gc}{\mathcal{G}}

\newcommand{\Sc}{\mathcal{S}}

\newcommand{\dist}{\operatorname{dist}}

\newcommand{\T}{\mathbb T}
\newcommand{\Cc}{\mathcal C}

\newcommand{\dif}{\,\mathrm{d}}
\newcommand{\LL}{\mathcal{L}}
\newcommand{\supp}{\operatorname{supp}}
\newcommand{\uo}{{\underline{\omega}}}

\begin{document}

\maketitle

\begin{abstract}
We study the long-time Fourier distribution of driven passive scalars transported by an incompressible alternating shear flow (the Pierrehumbert flow) subjected to IID random phase shifts, in the regime of strong shear-strength. With a novel and self-contained proof, we establish (i) a quantitative version of the cumulative Batchelor's law;  (ii) quantitative negative regularity mixing estimates; and (iii) a strong form of `sparsity' -- unevenness in the distribution of Fourier mass in radial sectors in frequency space. Our method uses the strong shear-strength assumption to handle a purely deterministic scalar driving term.

\bigskip
% \noindent \textbf{Keywords:} Passive scalar advection, Batchelor's Law \\

\noindent \textbf{2020 Mathematics Subject Classification:} Primary 35Q49, 76F20; Secondary 37D25, 37A25, 76F25.
% the Pierrehumbert flow, a randomly phase-shifted alternating shear flow on $\T^2$, in the large-shear regime. We first establish quantitative negative-Sobolev mixing estimates, both averaged and quenched, with rates that improve with the shear amplitude. For deterministic scalar forcing---either impulsive in time or continuous in time---we use these estimates to construct unique stationary laws and prove a cumulative form of Batchelor's law: the Fourier mass below wavenumber $R$ grows logarithmically with $R$. We then show that this spectrum is strongly anisotropic. At appropriate temporal phases, all but a vanishing fraction of the cumulative Fourier mass is confined to a narrow cone of wavenumbers. Thus angular sparsity, previously observed in pulsed-diffusion models, also occurs for a spatially smooth continuous-time flow.

% \begin{itemize}
%     \item 
% \end{itemize}

\end{abstract}

\section{Introduction}\label{sec:intro}
    %!TEX root = main.tex

Let $u = u(t, x)$ be a possibly time-dependent incompressible velocity field on the periodic box $\T^d$, and let $\varphi = \varphi(t, x)$ solve, for $\kappa \geq 0$, the advection-diffusion equation
\begin{equation}\label{eq:introDrivenPassiveScalar}
    \partial_t \varphi + u \cdot \nabla \varphi = \kappa \Delta \varphi + S \, , \qquad \varphi(0,\cdot) = \varphi_0(\cdot) \,.
\end{equation}
The driving term $S = S(t, x)$ is taken to be spatially regular and mean-free.
We are concerned with the long-time behavior of solutions to \eqref{eq:introDrivenPassiveScalar} in a subset of the Batchelor regime (see, e.g., \cite{batchelor1959small, bedrossianBatchelorSpectrumPassive2020,shraiman2000scalar,falkovich2001particles}) where $0 \leq \kappa \ll 1$ and the velocity field $u$ is spatially regular.

When the Lagrangian flow of $u$ is sufficiently chaotic -- corresponding, in the fluid context, to shear-straining throughout a large portion of state space -- solutions to \eqref{eq:introDrivenPassiveScalar} exhibit \emph{passive scalar turbulence}, characterized by a hierarchy of small-scale structures \cite{shraiman2000scalar,falkovich2001particles}. Roughly speaking, this hierarchy can be viewed as a kind of non-equilibrium `conveyor belt' of Fourier mass: scalar mass is injected at large scales through the driving term $S(t, x)$, carried from large to small scales by the shear-straining of $u$ through the transport term, and finally dissipated by the $\kappa \Delta$ term at spatial scales $\approx \sqrt{\kappa}$.
In this context, \emph{Batchelor's law} (\cite{batchelor1959small}) is the characteristic power-spectral law
\begin{align}\label{eq:batchelorModewiseIntro}
| \hat \varphi(k)|^2 \approx \frac{1}{|k|^{d}} \, , \qquad 1 \ll |k| \ll \frac{1}{\sqrt{\kappa}} \, ,
\end{align}
where $k \in \Z^d$ and $\hat \varphi(k)$ is the $k$-th spatial Fourier coefficient of $\varphi$.

Versions of \eqref{eq:batchelorModewiseIntro} have been established for a wide class of velocity fields, including solutions to stochastic Navier-Stokes with spatially smooth, white-in-time forcing \cite{bedrossianBatchelorSpectrumPassive2020, cooperman2024exponential}. For instance, when $S(t, x)$ is white in time and $u(t, x)$ comes from the stochastic Navier--Stokes equations at a fixed Reynolds number, a \emph{radial} version of Batchelor's law holds\footnote{with the caveat that \eqref{eq:radialBatchelor1} can fail at a small set of `exceptional' radii $R$; see \cite{cooperman2026fourier} for details.}:
\begin{align}\label{eq:radialBatchelor1}
    \sum_{|k| \in [R, R + h]} | \hat \varphi(k)|^2 \approx \frac{1}{R} \, , \qquad 1 \ll R \ll \frac{1}{\sqrt{\kappa}} \, ,
\end{align}
where $h > 0$ is a fixed annular width \cite{cooperman2026fourier}. 

Equation~\eqref{eq:radialBatchelor1} sums over annuli of wavenumbers of roughly the same amplitude, and the $1/R$ scaling law is uniform across the inertial range of scales $\ll \sqrt{\kappa}^{-1}$. This equation is not incompatible, however, with a deeply uneven concentration of Fourier mass within an individual radial shell -- it is allowed, for instance, for Fourier mass to become confined to a narrow sector of wavenumbers. This point was explored in \cite{blumenthal2026sparsity}, which established that for certain pulsed-diffusion toy models (see, e.g., \cite{feng2019dissipation}) for which shear-straining is roughly aligned across the fluid domain, Fourier mass can be badly depleted from entire sectors of wavenumbers -- a phenomenon referred to in \cite{blumenthal2026sparsity} as \emph{sparsity}. 

However, the pulsed-diffusion examples given in \cite{blumenthal2026sparsity} -- perturbations of Arnold's cat map -- cannot be realized by continuous-time flows. 
The purpose of this manuscript is to show that a system of strong alternating flows on $\T^2$, subjected to random phase shifts (also known as the Pierrehumbert flow \cite{pierrehumbert1994tracer}) exhibits sparsity when the individual shears have large amplitude. In this regime, shear-straining is well aligned across most, but not all, of the fluid domain, confirming the mechanism of \cite{blumenthal2026sparsity} in this more realistic setting. 
% These models are significantly more relevant to physical fluids because they can be constructed from stationary solutions of the incompressible Euler equations.

The main results presented in this paper are:
\begin{itemize}
    \item[(i)] a novel and self-contained proof of a `cumulative' version of Batchelor's law; and
    \item[(ii)] a proof of a relatively dramatic version of angular sparsity: the Fourier mass outside a very narrow cone of modes is disproportionately small.
\end{itemize}
An advantage of our approach is that it allows us to treat a \emph{deterministic} source term $S(t, x)$, in contrast with the vast majority of previous work on Batchelor's law which requires  white-in-time scalar source. The only other work we are aware of treating deterministic source is the very recent paper of \cite{liss2026batchelor} on passive scalar turbulence for strong-amplitude wedge flows (for which the velocity field is not spatially differentiable). To the authors' awareness, the results presented here are the first to treat passive scalar turbulence with a deterministic source for a smooth velocity field $u$.

% side quests:
%     \begin{itemize}
%     \item in setting of (i), Batchelor for deterministic forcing. Relatively straightforward proof using significantly $H^{-s} \to H^{-s}$ decay of a single iterate of the transfer operator (on average).
%     \item (i) Independent and conceptually simple Fourier-based proof of Batchelor's law for high amplitude scalar.
%     \end{itemize}

The setting is laid out in Section \ref{subsec:setting1}. A preliminary statement of the results is given in Section \ref{subsec:results1}. Some context and previous literature are discussed in Section \ref{subsec:discussion1}, together with a plan for the rest of the paper.

\subsection{Setting}\label{subsec:setting1}

\paragraph{Notation.} Throughout, the $d$-dimensional periodic box $\T^d$ is parametrized by $[0,1)^d$. We largely work with $d = 1, 2$. For $x \in \T^2$ write $x = (x^1, x^2)$, while for $k = (k^1, k^2) \in \Z^2$ we write $e_k(x) = e^{2 \pi i k \cdot x}$ for the usual Fourier basis. 
Given a set $S \subseteq \Z^2$, write $\Pi_S$ for the Fourier projection defined so that $\widehat{\Pi_S \psi} = \mathds{1}_S \hat{\psi}$.
Common sets include the ball \[B_R := \{p \in \Z^2 \mid |p| < R\}\] for $R \geq 1$ and cone \[\mathcal{C}_\theta \coloneq \{(k, j) \in \Z^2 : |j| < \theta |k|\} \] for $\theta > 0$. 
For generalized functions $\varphi$ on $\T^2$, we write $\hat \varphi : \Z^2 \to \C$ for the Fourier series, and 
\begin{equation*}
  \|\varphi\|_{H^{s}} \coloneq \left(\sum_{k \in \Z^2} (1 + |k|^2)^s|\hat{\varphi}(k)|^2\right)^{1/2}
\end{equation*}
for the usual inhomogeneous Sobolev norm of regularity $s \in \R$.
For $r\in\R$, write
\[
  H^r_0:=\{g\in H^r(\T^2):\widehat g(0)=0\}
\]
for the closed subspace of mean-zero distributions, equipped with the same
inhomogeneous Sobolev norm and its induced topology.

% {\color{blue} Convention for the whole paper: big “C” is a sufficiently large constant $\geq 1$, while little “c” is a sufficiently small constant $\leq 1$.   } 

\paragraph{Basic assumptions: vector field.}

Throughout, we fix a $\psi : \T \to \R$, treated as a \emph{shear profile} for the resulting Pierrehumbert flow and assume the following mild nondegeneracy assumptions: 
\begin{itemize}
    \item[(ND)] $\psi \in C^3(\T)$ and 
    \[\{ \psi = 0 \} \cap \{ \psi' = 0\} = \emptyset \, , \qquad \{ \psi'' = 0\} \cap \{ \psi''' = 0 \} = \emptyset \,. \]
\end{itemize}
These are easily satisfied by $\psi = \sin(2 \pi \cdot)$, for instance. 

Next, we impose the following on the random phase shifts for our model: 
\begin{itemize}
    \item[(P)] The phase shifts $\uo = (\omega_i^j)_{i,j}$ are independent and identically distributed, with common density $p \in L^\infty(\T)$. 
\end{itemize}
Let us briefly set down some additional notation for the probabilistic setup. Define $(\Omega, \Fc, \P) = (\Omega_0, \Fc_0, \P_0)^\N$ with $\Omega_0 = \T^2$, where $\P_0(\dif \omega^1, \dif \omega^2) = p(\omega^1) p(\omega^2) \dif \omega^1 \dif \omega^2$ is the law of a pair of phase shifts $(\omega^1, \omega^2)$. Let $\sigma : \Omega \circlearrowleft$ be the left shift on $\Omega$, so that $(\sigma \uo)_{j}^i = \omega_{j + 1}^i$. 

Let $L > 0$, regarded as a \emph{shear amplitude}, to be taken sufficiently large in what follows depending in particular on $\psi$ and $p$. 
The vector fields $u(t, x) = u_\uo(t, x)$ we consider are given by 
\begin{align}\label{eq:velField1}
    \begin{aligned}
        u(t, x) = 
        \begin{cases}
            \begin{pmatrix}
                0 \\ L \psi(x^1 + \omega_i^1)
            \end{pmatrix} & t \in [2i-2,2i-1)\\ 
            \begin{pmatrix}
                L \psi(x^2 + \omega_i^2) \\ 0
            \end{pmatrix} & t \in [2i-1, 2i)
        \end{cases}
    \end{aligned}
\end{align}
where $i = 1, 2, \dots$ counts the number of 2-periodic cycles. For $t \geq 0$ the corresponding Lagrangian flow $\Phi^t = \Phi^t_\uo : \T^2 \to \T^2$ is defined through
% for $t \geq 0$, the (random) Lagrangian flow $\Phi^t = \Phi^t_\uo : \T^2 \to \T^2$ satisfies
\[\frac{d}{dt} \Phi^t(x) = u(t, \Phi^t(x)) \]
and satisfies the composition law\footnote{Note that a single cycle takes two time iterates, hence the mismatch between the elapsed physical time  $2m$ and the number of phase shifts $m$.}
\[\Phi^{2n + 2m}_\uo = \Phi^{2n}_{\sigma^m \uo} \circ \Phi^{2m}_\uo \,. \]

For the model and results below we set $\kappa=0$. For $0 \leq r \leq t$ define 
\[\Phi^{r,t}_\uo := \Phi^t_\uo \circ (\Phi^r_{\uo})^{-1}\]
and write $\mathcal{S}^{r, t}= \mathcal{S}^{r, t}_\uo$ for the homogeneous transport propagator for \eqref{eq:introDrivenPassiveScalar}; by the method of characteristics, 
\[\mathcal{S}^{r, t} \varphi = \varphi \circ (\Phi^{r, t}_\uo)^{-1} \, . \]

% {\color{blue}  comment: $\sigma$ shifts the number of two-timestep `cycles', hence the peculiar notation } 

\paragraph{Basic assumptions: driven passive scalar.} We now turn to two specific settings for the driving term $S(t, x)$ in \eqref{eq:introDrivenPassiveScalar}. 
% Two settings of interest: discrete-time driving and continuous-time driving. 
% In both cases, for simplicity we consider the initial value problem \eqref{eq:introDrivenPassiveScalar} with $\varphi_0 \equiv 0$. 

\begin{setting}[Discrete-time, deterministic driving]\label{set:discDet}
We have that 
\begin{align}
    S(t, x) = \sum_{n = 1}^\infty \delta_{2n}(t) f \, , 
\end{align}
where $\delta_{2n}$ is the standard Dirac delta supported at $t = 2n$, and $f \in H^1_0 \setminus \{ 0 \}$ is a fixed (deterministic) mean-zero function.
\end{setting}

% \begin{setting}[Discrete-time, IID driving]\label{set:discRand}
%     We have that 
% \begin{align}
%     S(t, x) = \sum_{n = 1}^\infty \delta_{2n}(t) f_n 
% \end{align}
% where $f_n := \eta_n f$. Here, $(\eta_n)$ is an IID, square-integrable, mean-zero sequence (independent of the phase shifts $(\omega^i_j)$), and where $f$ is a fixed mean-zero $H^1$ function. 
% \end{setting}

\begin{remark}\label{rmk:rightContinuousIntro}
    Relevant to Setting \ref{set:discDet} is the convention, followed throughout the paper, that $\varphi_t$ is right-continuous in time, i.e., \[\lim_{\tau \to t+} \varphi_\tau = \varphi_t\] for all $t \geq 0$. In particular, in Setting \ref{set:discDet} the time-$t$ scalar $\varphi_t$ is given as 
    \begin{align}
        \varphi_t = \mathcal{S}^{0,t} \varphi_0 + \sum_{n = 1}^{\lfloor t/2\rfloor} \mathcal{S}^{2n, t} f \, , 
    \end{align}
    which differs from the left-limit $\varphi_{t-} := \lim_{\tau \to t-} \varphi_\tau$ only for $t \in 2 \Z$. 
\end{remark}

The next setting addresses continuous-in-time driving. 

\begin{setting}[Continuous-time, deterministic driving]\label{set:contDet}
    We have that 
    \begin{align}
        S(t, x) = f(x)
    \end{align}
    for a fixed $f \in H^1_0$ with
\begin{gather}\label{eq:cont-force-assumptions}
  \|f\|_{H^1}+\|f\|_{L^\infty} < \infty \, , \\ 
  \left\|\int_\T f(x^1,x^2)\dif x^1\right\|_{L^2_{x^2}} +
  \left\|\int_\T f(x^1,x^2)\dif x^2\right\|_{L^2_{x^1}} > 0 \,. 
\end{gather}
In particular, in Setting \ref{set:contDet} the time-$t$ scalar $\varphi_t$ is given by 
\begin{equation}
%   \label{eq:scalar-duhamel}
  \varphi_t = \mathcal{S}^{0, t} \varphi_0 + 
%   \begin{cases}
%     \sum_{i=1}^{\lfloor t/2\rfloor} \mathcal{S}^{2i,t} f & \text{Setting \ref{set:discDet}} \\ 
%   \int_0^t \mathcal{S}^{\tau, t} f \dif \tau & \text{Setting \ref{set:contDet}}
    % \end{cases}
  \int_0^t \mathcal{S}^{\tau, t} S(\tau, x) \dif \tau \, . 
%   \sum_{i=1}^{\lfloor t/2\rfloor} \mathcal{S}^{2i,t} f \, .
\end{equation}
\end{setting}

% \begin{setting}[White-in-time driving]
%     {\color{red}  fill in } 
% \end{setting}

% Settings above are simplest to state and most interesting. see remark \ref{rem:generalizations1} for comments on the most generalized settings we can consider with our methods. 

% \subparagraph{Deterministic driving, discrete-time case.} In this case, we let $\{ f_n\}_{n \geq 1}$ be a sequence of mean-zero, $H^1$ functions $\T^2 \to \R$ with the property that for all $n$, 
% \[F^{-1} \leq \|f_n\|_{L^2} \leq \|f_n\|_{H^1} \leq F \]
% where $F > 0$ is a constant. 

% Thus the forcing profile $f_n$ is delivered at the physical time $2n$; with
% our right-limit convention, this impulse is included in $\varphi_{2n}$.

% {\color{blue}  to modify: just set $f_n$ all equal to some fixed $f$. no need to get too complicated in the introduction.   } 

% Define the (discrete-time) forced scalar by
% \begin{equation}\label{eq:forced-scalar}
%   \begin{cases}
%     \varphi_n = \LL_{\alpha_n, \lambda} (f_{n-1} + \varphi_{n-1}) &\quad \text{for $n \geq 1$,}\\
%     \varphi_0 = 0.
%   \end{cases}
% \end{equation}

% Denote by $\mathcal{S}^{r,t}$ the solution map for the homogeneous operator from time $r$ to $t$, that is,
% \begin{equation}\label{eq:unforced-scalar}
%   \begin{cases}
%     \mathcal{S}^{r,t} g = \LL_{\alpha_t, \lambda} \mathcal{S}^{r, t-1} g &\quad \text{for $t \geq r+1$,}\\
%     \mathcal{S}^{r,r} g = g.
%   \end{cases}
% \end{equation}

\paragraph{Driven scalar Markov chain}

% By linearity,
% \begin{equation}
%   \label{eq:scalar-duhamel}
%   \varphi_t = \mathcal{S}^{0, t} \varphi_0 + 
%   \begin{cases}
%     \sum_{i=1}^{\lfloor t/2\rfloor} \mathcal{S}^{2i,t} f & \text{Setting \ref{set:discDet}} \\ 
%   \int_0^t \mathcal{S}^{\tau, t} f \dif \tau & \text{Setting \ref{set:contDet}}
%   \end{cases}
% %   \int_0^t \mathcal{S}^{\tau, t} S(\tau, x) \dif \tau
% %   \sum_{i=1}^{\lfloor t/2\rfloor} \mathcal{S}^{2i,t} f \, .
% \end{equation}

% {\color{green} Fix $r \in [0,2)$, viewed as a \emph{temporal phase}\footnote{While Pierrehumbert flow is not 2-periodic, due to the random phase shifts, it is ``2-periodic in law'', in that for $t, x$ fixed, the probabilistic law of $u(t, x)$ coincides with that for $u(t + 2, x)$.} for the Pierrehumbert flow.   } 
Consider Setting \ref{set:discDet} for the moment; Setting \ref{set:contDet} is handled similarly.
We take initial data $\varphi_0\in H^{-s}_0$. Incompressibility and the
mean-zero forcing preserve the spatial mean, so $\varphi_t\in H^{-s}_0$
for all $t\geq0$.
Observe that for any $n \geq 0$, 
\[\varphi_{2(n+1)} = \mathcal{S}^{2n, 2(n+1)} \varphi_{2n} + f \,. 
\]
In particular, the probabilistic law of $\varphi_{2(n+1)}$ conditioned on $(\varphi_{2k})_{0 \leq k \leq n}$ depends only on $\varphi_{2n}$, hence $(\varphi_{2n})$ constitutes a Markov chain. A similar argument applies to the left-limiting scalars $(\varphi_{2n-}), \varphi_{2n-} = \lim_{t \to 2n-} \varphi_t$. 
In view of $H^{-s}$-decay, it is natural to view each of $(\varphi_{2n}), (\varphi_{2n-})$ as Markov chains on $H^{-s}_0$. The following addresses \emph{stationary measures} of these chains.

\begin{proposition}\label{prop:statStatesPSA}
    Assume either Setting \ref{set:discDet} or \ref{set:contDet}. Then, for all sufficiently small $s > 0$ and all sufficiently large $L > 0$,
    \begin{itemize}
        \item[(a)] The sequences $(\varphi_{2n}), (\varphi_{2n-})$ constitute Markov chains on $H^{-s}_0$;
        \item[(b)] There exist unique stationary measures $\mu, \mu_-$ on $H^{-s}_0$ corresponding to $(\varphi_{2n}), (\varphi_{2n-})$, respectively; and
        \item[(c)] For any fixed $\varphi_0\in H^{-s}_0$ it holds that
        \begin{align}
            \operatorname{Law}(\varphi_{2n}) \to \mu \, , \qquad \operatorname{Law}(\varphi_{2n-}) \to \mu_-  
        \end{align}
        in the weak-$*$ topology of probability measures on $H^{-s}_0$.
    \end{itemize}
    % There exists $s > 0$ with the property that for all $L$ sufficiently large, the chains $(\varphi_{2n})$ and $(\varphi_{2n-})$, regarded as Markov chains on $H^{-s}$, admit unique stationary measures $\mu, \mu_-$ respectively. 
    % \begin{itemize}
    %     \item[(a)] The
    % \end{itemize}
    % {\color{red}There exists $s > 0$ with the property that for any $L$ sufficiently large, there exists a unique stationary measure $\mu$ for the driven scalar chain $(\varphi_{2n})$.
    % }
    % \begin{itemize}
    %     \item[(a)] 
    % \end{itemize}

    % {\color{red} In either of Settings \ref{set:discDet} or \ref{set:contDet}, view the time-$2n$ solutions $(\varphi_{2n})_{n \geq 0}$ as a Markov chain in $n$. Let $L$ be sufficiently large. Then, there is a unique stationary measure for  } 
\end{proposition}
% In fact, can show that at long times, the probabilistic law of $\varphi_{2n}$, averaged over phase shift realizations $\uo$, converges to the stationary measure in a suitable sense as $n \to \infty$. Elaborated on in Section \ref{subsec:wrapup3}. 

Proposition \ref{prop:statStatesPSA} is proved in Section \ref{subsec:wrapup3}. 

\begin{remark}\label{rmk:off-timeChain} \ 
    \begin{itemize}
        \item[(a)] Uniqueness of $\mu, \mu_-$ is likely true for all $L > 0$, not just for $L$ sufficiently large, although we do not prove this here. See \cite[Theorem 2.6]{bedrossian2024negative} for a broader perspective on negative regularity mixing and implications for unique stationary measures. 
        % {\color{blue}  uniqueness probably true for all $L$, following \cite{bedrossian2024negative}; here we are only concerned with $L$ large. proof in Section \ref{subsec:wrapup3}. } 
        
        \item[(b)] Let $r \in (0,2)$, thought of as a ``temporal phase'' for the Pierrehumbert flow. It is natural to consider the sequences $(\varphi_{2n+r})$. However, except for $r = 1$, these sequences need not form Markov chains. Nevertheless, it is possible to show the weak-$*$ convergence
        \begin{align}
            \operatorname{Law}(\varphi_{2n+r}) \to \mu_r 
        \end{align}
        for a (unique) Borel probability $\mu_r$ on $H^{-s}_0$ depending on the phase $r$. We set $\mu_0=\mu_2=\mu$. See Section \ref{subsec:wrapup3} for further details.
        
        % {It can be shown (Section \ref{subsec:wrapup3}) that the odd-time sequences $(\varphi_{2n+1}), (\varphi_{2n+1-}), n \geq 0$ also form Markov chains, and Proposition \ref{prop:statStatesPSA} easily implies unique existence of a stationary measures $\mu_{odd}, \mu_{odd,-}$ at these times.   } 
        
    \end{itemize}
\end{remark}

% {\color{red}  comment: form of forcing allows for a full horizontal-then-vertical shear effect between scalar drivings. this is kinda important or else the scalars can cancel each other out-- see discussion below } 

% Assume also that nonzero elements of the forcing sequence are separated by between $2$ and $F$ steps, that is, assume that if $|n-m| = 1$ then either $f_n = 0$ or $f_m = 0$, and for every $n \geq 0$ there is $k \in \{0, \dots, F-1\}$ with $f_{n+k} \neq 0$.

% \subparagraph{Deterministic driving, continuous-time case.}

% {\color{blue}  put down the setting here. use $f$ constant in time for simplicity. 

% comment: we don't do it here but our proof works if the $f$ varies moderately slowly in time. } 

\subsection{Statement of results}\label{subsec:results1}

\paragraph{(I) Undriven passive scalars.}
We begin with statements on small-scale formation for the passive scalar equation 
\begin{align}\label{eq:PSAnoDrivingIntro}
    \partial_t \varphi + u \cdot \nabla \varphi = 0 \, , \quad \varphi(0, \cdot) = \varphi_0
\end{align}
with no scalar driving. 

\begin{letteredtheorem}\label{thm:introNegRegMix}
    There exist $s, L_0, C > 0$, depending only on the density $p$ and shear profile $\psi$, along with universal constants $\mathfrak{c}, \mathfrak{d} > 0$, with the following properties. 
    Let $L \geq L_0$ and let $\varphi_t := \varphi(t, \cdot)$ solve \eqref{eq:PSAnoDrivingIntro} with fixed initial mean-zero $\varphi_0 \in L^2$. Then, 
    \[\E \| \varphi_t\|_{H^{-s}}^2 \leq C L^{\mathfrak{d} - t \mathfrak{c}} \| \varphi_0 \|_{L^2}^2 \,. \]
\end{letteredtheorem}

{
% our value for $\mathfrak{c}$ is quite poor - something around $\mathfrak{c} \approx .0175$. Likely suboptimal, have not made an effort to improve in this paper.  Value $\mathfrak{d}$ is $< 1$.

Theorem \ref{thm:introNegRegMix} is proved in Section \ref{subsec:quenchedNegRegMix}. See Remark \ref{rmk:explicitConsts2} for explicit constants in Theorem \ref{thm:introNegRegMix}. 
} 

We next record a consequence for \emph{almost-sure} $H^{-s}$ decay, also proved in Section \ref{subsec:quenchedNegRegMix}. 

\begin{letteredcorollary}\label{cor:introNegRegMixQuenched}
    Let $0 < \mathfrak{c}' < \frac12\mathfrak{c}$. There exist $C_{\mathfrak{c}'} > 0$ and a random variable $D = D_{\mathfrak{c}', L} \geq 1$, depending on the phase shifts $(\omega^j_i)_{i,j}$, such that 
    \[\E D \leq C_{\mathfrak{c}'} L^{\mathfrak{d}'},\qquad \mathfrak d'=2s+2\mathfrak c'\]
    with the following property. Let $\varphi_t := \varphi(t, \cdot)$ solve \eqref{eq:PSAnoDrivingIntro} with fixed initial mean-zero $\varphi_0 \in H^1$. Then, 
    \[\| \varphi_t\|_{H^{-s}} \leq D L^{- t \mathfrak{c}'} \| \varphi_0\|_{H^1} \,. \]
\end{letteredcorollary}

{We note that, as in \cite{bedrossianAlmostsureExponentialMixing2019}, one can replace $\| \varphi_0\|_{H^1}$ with $\| \varphi_0 \|_{H^m}, m > 0$ on changing the mixing rate $\mathfrak{c}'$. Similarly, one can arrange for random variable $D$ to admit higher moments at expense of reducing the mixing rate $\mathfrak{c}'$; see Section \ref{subsec:quenchedNegRegMix} for details. } 

\paragraph{(II) Driven passive scalars}
We now turn to a description of statistically stationary solutions to the driven problem \eqref{eq:introDrivenPassiveScalar}. 

% following statements hold in both discrete and continuous-time scalar driving cases. 
% {\color{red}Below, let $\mu$ be as in Proposition \ref{prop:statStatesPSA} in either of settings \ref{set:discDet}, \ref{set:contDet}.}

\begin{letteredtheorem}\label{thmA:batch}
    % {\color{red}  both settings \ref{set:discDet}, \ref{set:contDet} } 
    In either Setting \ref{set:discDet} or \ref{set:contDet},
    there exists $\zeta > 0$ with the following properties. Let $q \geq 1$ and let $L$ be sufficiently large. Let $r \in [0,2)$ be arbitrary, and let $\varphi_{\rm stat, r}$ be a random scalar in $H^{-s}_0$ with distribution $\mu_r$ (Proposition \ref{prop:statStatesPSA}, Remark \ref{rmk:off-timeChain}). Then there exists $C \geq 1$, depending on $p, \psi, f$ and $q$ such that, with $c = C^{-1}$, it holds that
    \begin{align}\label{eq:batchelorIntro}
        c (\log_L R)^q \leq \E \| \Pi_{B_R} \varphi_{\rm stat, r} \|_{L^2}^{2q} \leq C (\log_L R)^q
    \end{align}
    for all $R \geq L^\zeta$.
    Moreover, in Setting \ref{set:discDet}, equation \eqref{eq:batchelorIntro} holds when $\mu$ is replaced by $\mu_-$. 
    In addition, for all $R\geq L^\zeta$,
    \begin{equation}\label{eq:batchelorLowerTailIntro}
      \P\{\|\Pi_{B_R}\varphi_{\rm stat,r}\|_{L^2}^2<c\log_LR\}
      \leq CL^{-c},
    \end{equation}
    also with $\mu_r$ replaced by $\mu_-$ in Setting \ref{set:discDet}.
\end{letteredtheorem}

We continue with a statement on sparsity in Setting \ref{set:discDet}. We include the left-limiting and right-limiting laws at even times, as well as phases bounded away from the start of the horizontal shear. 

\begin{letteredtheorem}\label{thmA:sparsity}
    % only setting \ref{set:discDet}.
    Assume Setting \ref{set:discDet}. 
    Let $\vartheta \in [0,1)$, $q \geq 1$ and $\varepsilon_0 \in (0,1]$. Then, there exist $C, L_0 \geq 1$ such that the following holds for all $L \geq L_0, R \geq 2$. Let $\Theta \coloneq L^\vartheta$.

    Below, $\mu$ is stationary for $(\varphi_{2n})$, and $\mu_r$ is as in Remark \ref{rmk:off-timeChain} with $\mu_2, \mu_0 := \mu$, while $\mu_-$ is stationary for $(\varphi_{2n-})$. 
    \begin{itemize}
        \item[(a)] Let $\varphi_{\rm stat, +}$ be distributed like $\mu_-$ or like $\mu_r$ for $r \in [1 + \varepsilon_0, 2)$. Then, 
        \begin{align}\label{eq:goodConeConf1}
            \E \| \Pi_{\Cc_\Theta \cap B_R} \varphi_{\rm stat, +} \|_{L^2}^{2 q} \leq C L^{- \frac{1 - \vartheta}{8} } \left(\log_L R\right)^q  \,. 
        \end{align}
        \item[(b)] Let $\varphi_{\rm stat,0}$ be distributed like $\mu = \mu_0$. Then, 
        \begin{align}
            \E \| \Pi_{\Cc_\Theta \cap B_R} \varphi_{\rm stat,0} \|_{L^2}^{2 q} \leq C \left(L^{- \frac{1 - \vartheta}{8} } (\log_L R)^q + 1\right) \,. 
        \end{align}
    \end{itemize}

    % and let $\varphi_{\rm stat}$ be a random mean-free scalar distributed like $\mu$. Then, 
    % \begin{align}
    %     \E \| \Pi_{\Cc_\Theta \cap B_R} \varphi \|^{2 q} \leq C \left(L^{- \frac{1 - \vartheta}{8} } \log_L R + 1\right) \,. 
    % \end{align}
\end{letteredtheorem}

% {\color{red}  note: teh following is OK so far in Setting \ref{set:discDet}. havent checked for \ref{set:contDet}. 

% address left-limit separately? in a remark for instance? } 

% consequence: dividing through by $\log_L R$ gives small numbers. 

Theorem \ref{thmA:sparsity} is proved at the end of Section \ref{subsec:contSparsity4}. 
Combining~\eqref{eq:batchelorLowerTailIntro} and Theorem~\ref{thmA:sparsity}, we obtain the following consequence for the proportion of $L^2$ mass of $\varphi_{\rm stat}$, $\varphi_{\rm stat,-}$ in the cone $\Cc_\Theta$.

\begin{letteredcorollary}\label{corA:sparsityDisc}
{Assume the setting of Theorem \ref{thmA:sparsity} and let $\zeta$ be as in Theorem \ref{thmA:batch}. There exists $C \geq 1, c := C^{-1}, L_0 \geq 1$, for which the following holds for any $R \geq L^\zeta$, $L \geq L_0$.

\begin{itemize}
    \item[(a)] Let $\varphi_{\rm stat, +}$ be as in Theorem \ref{thmA:sparsity}(a). With probability $\geq 1 - L^{-c}$, it holds that 
    \begin{align}\label{eq:tightQuenched1}\frac{\| \Pi_{\Cc_\Theta \cap B_R} \varphi_{\rm stat,+}\|_{L^2}^2}{\| \Pi_{B_R} \varphi_{\rm stat, +}\|_{L^2}^2} \leq C L^{- \frac{1 - \vartheta}{16}} \,. \end{align}
    \item[(b)] Let $\varphi_{\rm stat,0}$ be as in Theorem \ref{thmA:sparsity}(b). With probability $\geq 1 - L^{-c}$, it holds that 
    \[\frac{\| \Pi_{\Cc_\Theta \cap B_R} \varphi_{\rm stat,0}\|_{L^2}^2}{\| \Pi_{B_R} \varphi_{\rm stat,0}\|_{L^2}^2} \leq C \left(L^{- \frac{1 - \vartheta}{16}} + \frac{1}{ \log_L R}\right) \,. \]
\end{itemize}

% Let $R \geq L^\zeta$. Then, there is an event $\mathcal{G}$ with proba $\geq 1 - L^{-c}$ such that, conditioned on $\mathcal{G}$, one has
%     \begin{align}
%         \frac{\| \Pi_{\Cc_\Theta \cap B_R} \varphi_{\rm stat,-}\|_{L^2}^2}{\| \Pi_{B_R} \varphi_{\rm stat}\|_{L^2}^2} \leq C L^{- \frac{1 - \vartheta}{8}} \,. 
%     \end{align}

%     \begin{align}
%         \frac{\| \Pi_{\Cc_\Theta \cap B_R} \varphi_{\rm stat}\|_{L^2}^2}{\| \Pi_{B_R} \varphi_{\rm stat}\|_{L^2}^2} \leq C \left(L^{- \frac{1 - \vartheta}{8}} + \frac{1}{ \log_L R}\right) \,. 
%     \end{align}
}\end{letteredcorollary}

\begin{remark}[Temporal phase and cone confinement]\label{rmk:phaseConfinement1} \ 
% Some comments are in order. 
    \begin{itemize}
        \item[(a)] Since the velocity field $u$ alternates between horizontal and vertical shears, versions of the `tight' estimates \eqref{eq:goodConeConf1}, \eqref{eq:tightQuenched1} hold for $r \in [\varepsilon_0, 1]$ with $\mathcal{C}_\Theta$ replaced by the flipped cone $\tilde{\mathcal{C}}_\Theta := \{ (k,j) \in \Z^2 : |k| < \Theta |j| \}$. 
        \item[(b)] In Setting~\ref{set:discDet}, estimates at right-limiting even times (distribution $\mu$) are worse since in this case $\varphi_{\rm stat}$ incorporates a fresh, unadvected scalar whose mass may well lie entirely in $\mathcal{C}_\Theta$. 
        \item[(c)] Since we alternate between confinement to the narrow complements of $\tilde{\mathcal{C}}_\Theta$ and $\mathcal{C}_\Theta$ depending on the phase, the corresponding confinement estimates need not hold during the transitional windows $[0, \varepsilon_0]$ and $[1, 1 + \varepsilon_0]$. 
        {While we do not address these phases in detail in the present manuscript, it is likely that at these times the typical scalar mass is confined to a (possibly much wider) cone \emph{depending on the temporal phase $r$.} } 
    \end{itemize}

    % We note that the estimate is much tighter for the left-limiting law $\mu_-$ or for the ``good'' phases $[1 + \varepsilon_0, 2]$. At these times,  $\varphi_{\rm stat,good}$ has enjoyed at least $\varepsilon_0$-time of the strong vertical shear, driving scalar $L^2$ mass out of $\mathcal{C}_\Theta$. In contrast, the right-limiting law $\mu = \mu_0$ incorporates a `fresh', unadvected scalar term, while $\mu_r, r \in [1, 1 + \varepsilon_0)$ has not yet experienced 

    % since the right-limiting scalar $\varphi_{\rm stat}$ incorporates a `fresh', unadvected scalar term. 
\end{remark}

% {\color{red}  add almost-sure statements?  } 

% \begin{letteredcorollary}
%     Batchelor's law in the IID forcing {\color{red}  bother to include?  } 
% \end{letteredcorollary}

% \begin{itemize}
%     \item not much additional work to handle IID scalar driving. handled in a variety of papers -- BBPS4 and BK, and opted not to here for brevity and to focus on what is new. 
% \end{itemize}

% {\color{red}  to do: set up some of the ``statistically stationary scalar for driven process is unique'' from BFPS. Set up that if you have $T$-independent estimates, those pass to the stat stationary scalar. note that for every $r \in [0,2)$, you have a law $\mu^{\rm stat}_r$ on $H^{-}$ with respect to which you have stationarity for phase $r$ process. moment estimates are a little funny - you have much better estimates for $\mu^{\rm stat}_{0-}$ which does not account for the `fresh' scalar driving. phase parameter $r$ a pecularity of the fact that pierrehumbert is 2-periodic in law. relation among the $\mu^{\rm stat}_r$'s. } 

% various generalizations in setting possible -- see {\color{red}  ??? } 

Next we turn to sparsity for Setting \ref{set:contDet}, which is somewhat less subtle to state due to the continuity of $t \mapsto \varphi_t$. 

\begin{letteredtheorem}\label{thmA:contSparsity1}
    % {\color{red}statement for setting \ref{set:contDet}}
    Assume Setting \ref{set:contDet}. Let $\vartheta \in [0,1), \varepsilon_0 \in (0,1)$ and $q \geq 1$. Then, there exist $C, L_0 \geq 1$ such that the following holds for all $L \geq L_0, R \geq 2$. 
    Fix $r \in [1+\varepsilon_0, 2]$ and let $\varphi_{\rm stat, r}$ be distributed like $\mu_r$ (see Remark \ref{rmk:off-timeChain}). Then,
        \begin{equation*}
            \E \| \Pi_{\Cc_\Theta \cap B_R} \varphi_{\rm stat, r} \|_{L^2}^{2 q} \leq C L^{- \frac{1 - \vartheta}{8} } \left(\log_L R\right)^q  \,,
        \end{equation*}
        where $\Theta \coloneq L^\vartheta$.
%         \begin{align}
%             \E \| \Pi_{\mathcal{C}_\Theta \cap B_R} \varphi_{\rm stat, r}\|^{2q} \leq C L^{- \frac{1 - \vartheta}{8}} \left(\log_L R\right)^q \,. 
%         \end{align}
%     \item[(b)] Fix $r \in [1,2]$ and let $\varphi_{\rm stat, r}$ be distributed like $\mu_r$. Then, 
%         \begin{align}
%             \E \| \Pi_{\mathcal{C}_\Theta \cap B_R} \varphi_{\rm stat, r}\|^{2q} \leq C L^{- \frac{1 - \vartheta}{8}} \left(\left(\log_L R\right)^q + 1\right) \,. 
%         \end{align}
% \end{itemize}
\end{letteredtheorem}

% \begin{remark} \ 
%     \begin{itemize}
%         \item[(a)] Analogous statements for $r \in [\varepsilon_0, 1]$ (resp. $r \in [1, 2]$) hold with $\mathcal{C}_\Theta$ replaced by the flipped cone $\tilde{\mathcal{C}}_\Theta := \{ (k,j) \in \Z^2 : |k| < \Theta |j| \}$. 
%         \item[(b)] As with Setting \ref{set:discDet}, not all phases enjoy the same quality of cone confinement. Since we alternate between confinement to $\tilde{\mathcal{C}}_\Theta$ and $\mathcal{C}_\Theta$ depending on the phase, time continuity suggests a breakdown of confinement during the brief transitional windows $[0, \varepsilon_0]$ and $[1, 1 + \varepsilon_0]$ between the two.  
%     \end{itemize}
% \end{remark}

% By time continuity, phases in $[1, 1 + \varepsilon_0]$ or $[0, \varepsilon_0]$ are brief `transitional' windows during which confinement of scalar mass to the cone $\mathcal{C}_\Theta$ (resp. $\tilde{\mathcal{C}}_\Theta$) falters. 

% We close with the analogue of Corollary \ref{corA:sparsityDisc} in Setting \ref{set:contDet}.
\begin{letteredcorollary}\label{corA:contSparsityRatio1}
    Assume the setting of Theorem \ref{thmA:contSparsity1} and let $\zeta$ be as in Theorem \ref{thmA:batch}. There exists $C \geq 1, c := C^{-1}, L_0 \geq 1$, for which the following holds for any $R \geq L^\zeta$, $L \geq L_0$. Let $r \in [1 + \varepsilon_0, 2]$ and let $\varphi_{\rm stat, r}$ be distributed like $\mu_r$. Then, with probability $\geq 1 - L^{-c}$,
    \begin{equation*}\frac{\| \Pi_{\Cc_\Theta \cap B_R} \varphi_{\rm stat,r}\|_{L^2}^2}{\| \Pi_{B_R} \varphi_{\rm stat, r}\|_{L^2}^2} \leq C L^{- \frac{1 - \vartheta}{8}} \,, \end{equation*}
    where $\Theta \coloneq L^\vartheta$.
% \begin{itemize}
%     \item[(a)] Let $r \in [1 + \varepsilon_0, 2]$. With probability $\geq 1 - L^{-c}$, it holds that 
%     \[\frac{\| \Pi_{\Cc_\Theta \cap B_R} \varphi_{\rm stat,r}\|_{L^2}^2}{\| \Pi_{B_R} \varphi_{\rm stat, r}\|_{L^2}^2} \leq C L^{- \frac{1 - \vartheta}{8}} \,. \]
%     \item[(b)] Let $r \in [1,2]$. With probability $\geq 1 - L^{-c}$, it holds that 
%     \[\frac{\| \Pi_{\Cc_\Theta \cap B_R} \varphi_{\rm stat, r}\|_{L^2}^2}{\| \Pi_{B_R} \varphi_{\rm stat, r}\|_{L^2}^2} \leq C \left(L^{- \frac{1 - \vartheta}{8}} + \frac{1}{ \log_L R}\right) \,. \]
% \end{itemize}
\end{letteredcorollary}
The comments of Remark \ref{rmk:phaseConfinement1}(a) and (c) apply to Setting \ref{set:contDet} as well.

\begin{remark}
    In this manuscript, we have opted for the sake of length and clarity of focus to omit treating diffusivity in \eqref{eq:introDrivenPassiveScalar}, instead focusing on `ideal', undamped stationary laws \cite{bedrossianBatchelorSpectrumPassive2020}. While unaddressed, the authors are reasonably confident the details remain much the same when diffusivity is accounted for. Similarly, we have made the decision to omit treating random source terms, noting that random sources are technically easier to work with and are treated elsewhere in the literature, e.g., in \cite{bedrossianBatchelorSpectrumPassive2020,blumenthal2026sparsity}. 
    % \begin{itemize}
    %     \item Skipping diffusivity. Comment: not so different with or without. for sake of simplicity we have omitted it but likely it could be achieved. 
        % \item skipping random source. handled in other papers. 
    % \end{itemize}
\end{remark}

\subsection{Discussion and Literature Review}\label{subsec:discussion1}

Batchelor's law was originally proposed in his 1959 paper \cite{batchelor1959small}; see also the early works \cite{gibson1968fine,kraichnan1968small}. We emphasize that we work in a special subset of Batchelor-regime passive scalar turbulence. In particular, we consider a spatially regular velocity field at a fixed Reynolds number and do not address, for example, the Obukhov--Corrsin correction \cite{obukhov1949structure, corrsin1951spectrum, rowan2025obukhov}. Passive scalar turbulence provides a toy model of small-scale formation in hydrodynamic turbulence \cite{frisch1995turbulence}. For surveys, see \cite{shraiman2000scalar, warhaft2000passive, dimotakis2005turbulent, falkovich2001particles}.

\paragraph{Mixing properties of incompressible flows.}
A cumulative version of Batchelor's law (as in Theorem \ref{thmA:batch} above) was first established for solutions to \eqref{eq:introDrivenPassiveScalar} for a broad class of randomly-driven velocity fields, including solutions to stochastic Navier--Stokes, in \cite{bedrossianBatchelorSpectrumPassive2020}. This result was the culmination of the works \cite{bedrossianLagrangianChaosScalar2018,bedrossianAlmostsureExponentialMixing2019,bedrossianAlmostsureEnhancedDissipation2021}, which among other things established positivity of the Lyapunov exponent (see \eqref{eq:lyapExponentIntro} below) of the Lagrangian flow and almost-sure $H^{-s}$ decay (Corollary \ref{cor:introNegRegMixQuenched}). The works \cite{cooperman2024exponential, nersesyan2024chaotic} establish similar results under weaker assumptions on the noise applied to the Navier--Stokes component. 
In summary, and combined with the results of \cite{cooperman2026fourier}, the state of the art is that for white-in-time scalar driving, cumulative and radial versions of Batchelor's law as in \eqref{eq:radialBatchelor1} hold in a wide variety of settings. 

All these versions of Batchelor's law rely on establishing almost-sure $H^{-s}$ decay of solutions of the initial value problem for
\begin{align}\label{eq:advecDiff1}
    \partial_t \varphi + u \cdot \nabla \varphi = \kappa \Delta \varphi
\end{align}
at an exponential rate independent of $\kappa$ \cite{bedrossianAlmostsureEnhancedDissipation2021}.  This is a strong form of so-called \emph{enhanced dissipation}.  For alternative approaches to enhanced dissipation for randomly-driven fluids, see, e.g.,  \cite{cooperman2025harris, NavarroSeis2026}; for more dynamics-based approaches for certain classes of deterministic flows, see \cite{tao2025optimal, elgindi2025optimal}; and finally see \cite{zelati2024mixing} for an overview of the subject. 

An even stronger mixing property is $H^{-s}$ decay on average: for the undriven problem \eqref{eq:advecDiff1} with $\kappa = 0$,
\begin{align}\label{eq:HminSdec1}
    \E \| \varphi_t\|_{H^{-s}}^2 \lesssim e^{- \gamma t}\|\varphi_0\|_{H^{-s}}^2
\end{align}
for $\gamma > 0$. This property was established for a broad class of passive scalar problems in \cite{bedrossian2024negative}, using tools from microlocal analysis together with some of the twisted Feynman--Kac semigroup technology from \cite{bedrossianAlmostsureExponentialMixing2019}. Among other consequences, \eqref{eq:HminSdec1} implies almost-sure $H^{-s}$ decay for the problem with $\kappa = 0$~\cite[Corollary 1.13]{bedrossian2024negative}.
While the proofs are substantially different, there is a sense in which the present manuscript (particularly Proposition \ref{thm:neg-reg-mix}) documents a ``hands-on'', quantitative approach to the primary conclusions of \cite{bedrossian2024negative}.

A feature of our approach in this paper is a novel, self-contained, and direct proof of a version of \eqref{eq:HminSdec1}, using a weighted negative-regularity norm that penalizes a wide cone of frequencies around the $x^1$ axis (Proposition \ref{thm:neg-reg-mix} below). In particular, we avoid using any of the aforementioned technology -- Lyapunov exponents and twisted semigroups. In addition to establishing almost-sure exponential mixing and Batchelor's law, this approach yields uniqueness of the stationary measure for the driven scalar chain (Proposition \ref{prop:statStatesPSA}), as in \cite{bedrossian2024negative}.

Another interesting byproduct of our study is to exhibit $H^{-s}$ decay at a strong, \emph{quantitative} exponential rate improving with the shear amplitude $L$. While in principle it is possible (as shown in \cite{son2025quantitative, liu2026quantitative}) to extract effective $H^{-s}$ decay rates following roughly the path laid out in previous works such as \cite{blumenthal2023exponential}, the rates extracted slow down with the shear amplitude, contrary to the intuition that stronger alternating shearing should improve mixing (cf. \cite{blumenthal2018lyapunov}). The Fourier-based method presented in this paper is well-suited to extracting mixing rates confirming this intuition. 

\paragraph{Nonuniform hyperbolicity and the role of noise.}
Shear-straining corresponds to dynamical hyperbolicity of the Lagrangian flow: stretching and compression of phase space. Hyperbolicity is measured by the Lyapunov exponent
\begin{align}\label{eq:lyapExponentIntro}
    \lambda(x) = \limsup_{t \to \infty} \frac{1}{t} \log | D \Phi^{0,t}(x)| \, .
\end{align}
The condition $\lambda > 0$ implies an exponential rate of shear-straining,
which leads to blowup in $H^1$ for typical scalars supported over hyperbolic points, a first step toward establishing $H^{-1}$ decay, which is a priori a stronger property.
One difficulty is that, in physical systems, hyperbolicity typically holds only in some parts of state space and fails in others, where $\lambda(x) = 0$.
Such behavior appears in fluids, for example as coherent vortices.
Although shearing within a vortex does produce some small-scale formation (see, e.g., \cite{bedrossian2017enhanced}), it does so at a qualitatively slower rate than hyperbolic flows. 

Even at large shear amplitude, vortices and the hyperbolic set are hopelessly intermingled. This mechanism \cite{duarte1994plenty} is known for the Chirikov standard map, a related model \cite{chirikov1979universal, chirikov2008chirikov}, which itself serves as a useful toy model of Lagrangian flow -- for more on this comparison, see, e.g., \cite{crisanti2012products}.  
Proving that $\lambda(x) > 0$ on a set of positive area -- let alone $H^{-s}$ decay -- for models with `mixed' phase space, such as the standard map, is a notoriously difficult open problem in smooth ergodic theory \cite{crovisier2019problem}. 

In ours and in many recent works on $H^{-s}$ decay of passive scalars, the role of noise is to smooth out the long-time statistics of tangent directions. A folklore theorem following Furstenberg \cite{furstenberg1963noncommuting} states that the standard map with absolutely continuous phase shifts has positive Lyapunov exponents. In the presence of noise, almost-sure $H^{-1}$ decay for the Pierrehumbert flow was established in \cite{blumenthal2023exponential, cooperman2023exponential}. Even a small amount of noise `unlocks' the full strength of the exponent, resulting in $\lambda \approx \log L$ \cite{blumenthal2017lyapunov} and in rapid mixing of the associated Markov semigroup on passive particles \cite{blumenthal2018lyapunov}. 

For additional treatments of almost-sure statistical properties of random dynamical systems, see, e.g., \cite{dewitt2024expanding, dewitt2025conservative}. 

\paragraph{Challenges handled in this manuscript.}
Even when the Lyapunov exponent is large, for any fixed phase shift there is always a portion of state space on which hyperbolicity is weak, an issue which must be addressed when proving mixing properties for systems of this type; see, for example, the dynamical works \cite{blumenthal2020diffusion, blumenthal2021statistical, blumenthal2018lyapunov}. This does not pose a problem for the Lyapunov exponent or for mixing, but it could imperil \emph{quantitative} $H^{-s}$ decay or cone confinement as in Theorem \ref{thmA:sparsity}, since Fourier mass could be `twisted' into what should be the sparse cone of modes roughly parallel to the $x^1$ axis. 
Our resolution is to work with the full mean-free $L^2$ scalar and show that, with very high probability, one period of the Pierrehumbert flow carries the vast majority of its $L^2$ mass out of $\mathcal C_\Theta$ (Lemma \ref{lem:cone-decay}). The approach is Eulerian, in contrast to much of the mixing literature (though not all; see e.g.~\cite{NavarroSeis2026}) which proves scalar mixing by tracking (Lagrangian) trajectories of the two-point motion.

% \paragraph{Deterministic scalar driving.}
Treating a deterministic source term is another point of interest and technical difficulty. To our knowledge, it has otherwise been handled only in the very recent work \cite{liss2026batchelor} for a deterministic wedge flow, using the theory of anisotropic Banach spaces, very different from the approach in the present manuscript. The tradeoff is that \cite{liss2026batchelor} works with nonsmooth wedge flows, for which one must contend with the possible accumulation of discontinuities in the gradient of the Lagrangian flow -- a zero-Lebesgue-area version of the nonuniform hyperbolicity exhibited by smooth models such as Pierrehumbert flow with smooth shear profiles. In turn, \cite{liss2026batchelor} can treat deterministic flows, whereas our paper relies critically on randomized phase shifts. 

In both papers, treating a deterministic source term requires very strong short-time mixing. Roughly speaking,  the scalar source is rapidly conveyed to sufficiently small scales to avoid catastrophic interference -- through either cancellation or reinforcement -- with fresh, not-yet-advected scalar source. It remains open to study Batchelor's law in any setting where the deterministic scalar forcing varies in time at a frequency commensurate with the strength of the velocity field.

\subsection*{Plan for the paper}

Section \ref{sec:negRegMix} contains the proofs of Theorem \ref{thm:introNegRegMix} and Corollary \ref{cor:introNegRegMixQuenched}, and develops our novel technique, a hands-on negative regularity mixing estimate using an anisotropic negative regularity norm penalizing `roughly horizontal' wavenumbers. Section \ref{sec:batchLaw} addresses Batchelor's law (Theorem \ref{thmA:batch}), and finally Section \ref{sec:sparsity} addresses sparsity in Batchelor's law (Theorems \ref{thmA:sparsity}, \ref{thmA:contSparsity1} and Corollaries \ref{corA:sparsityDisc}, \ref{corA:contSparsityRatio1}). 

\subsection*{Acknowledgements}

The authors would like to recognize Manh Khang Huynh for helpful conversations and discussion at the beginning of the work in this manuscript. 

\subsection*{Acknowledgement of AI Use}

Large Language Models\footnote{ChatGPT 5.4 Pro and ChatGPT 5.6 Sol} (LLMs) were used in the preparation of this manuscript in several places, namely in pinpointing use of (i) the van der Corput estimate in the proof of Lemma \ref{lem:cone-decay}; and (ii) the $L^p$ version of the maximal ergodic theorem in the proof of Claim \ref{cla:boundWNMN3}. Beyond this, LLMs were used at times for writing rough first drafts, finding minor mathematical errors, and cleaning up and polishing exposition.

% {\color{red}  

% Section 2: proofs of Theorem \ref{thm:introNegRegMix} and Corollary \ref{cor:introNegRegMixQuenched}. tool introduced: mixing estimates using an anisotropic negative regularity penalizing `roughly horizontal' wavenumbers

% section 3: batchelor's law in all settings

% Section 4: sparsity in all settings

% } 

\section{Negative Regularity Mixing}\label{sec:negRegMix}
    %!TEX root = main.tex

\newcommand{\llambda}{L}

The purpose of this section is to document quantitative negative-regularity decay of solutions of the transport equation under the Pierrehumbert model laid out previously. 

\paragraph{Notation.} 
% We begin by setting down some notation and conventions to be used throughout the paper. 
We let $L_0 \geq 1$, to be specified sufficiently large as we go along, while throughout we assume $L \geq L_0$. We reserve the value $C \geq 1$ for a constant (possibly different line-to-line) depending on various data such as the shear profile $\psi$ and the density $p$, as well as on auxiliary parameters, but \emph{never} on the value of $L \in [L_0, \infty)$ and never on the phase shift sequence $\uo = (\omega_i^j)$. The shear profile $\psi$ and the phase distribution $p$ are as in Assumptions (ND) and (P) from Section \ref{subsec:setting1}.

\paragraph{Even-time realization.} Throughout the paper, but particularly in this section, we will use the following discrete-time realization of the Lagrangian flow $\Phi^{2n}$ at even-integer times. 
For $\omega \in \T$, we define $T_{\omega, L} : \T^2 \to \T^2$ to be the mapping 
\[ x \mapsto T_{\omega, L}(x) := (x^2 + L \psi(x^1 + \omega), x^1) \, . \]
The mapping $T_{\omega, L}$ can be thought of as a composition: first,  applying a time-1 vertical shear with profile $\psi$ and phase shift $\omega$, and second, flipping the $x^1, x^2$ axes. For mean-zero $\varphi : \T^2 \to \R$ we define
\begin{gather*}\Lc_{\omega, L} \varphi := \varphi \circ (T_{\omega, L})^{-1} \, , \\ 
\Lc_{\omega, L} \varphi(x^1, x^2) = \varphi(x^2, x^1 - L \psi(x^2 + \omega))
\end{gather*}
to be the associated \emph{transfer operator}, a discrete-time analogue of the propagators $\mathcal{S}^{r,t}$ for the original linear transport equation.

% The following is straightforward to check. 
\begin{lemma}
  For any $n \geq 1$ and phase shift sequence $\uo = (\omega^i_j) \in \Omega$, it holds that 
  \begin{gather}
    \Phi^{2n}_\uo = T_{\omega_n^2, L} \circ T_{\omega_n^1, L} \circ \dots \circ T_{\omega_1^2, L} \circ T_{\omega_1^1, L} \,, \quad \text{and} \\ 
    \mathcal{S}^{0,2n}_\uo = \Lc_{\omega_n^2, L} \circ \Lc_{\omega_n^1,L} \circ \dots \circ \Lc_{\omega_1^2,L} \circ \Lc_{\omega_1^1, L} \, .
  \end{gather}
\end{lemma}

\begin{proof}
  Let $\Phi^2$ denote the time-2 Lagrangian flow with phase shift $(\omega^1, \omega^2)$, and observe that $\Phi^2$ corresponds to a composition of the vertical shear step
  \[\begin{pmatrix} x^1 \\ x^2 \end{pmatrix} \mapsto \begin{pmatrix}
                x^1 \\ x^2 + L \psi(x^1 + \omega^1)
            \end{pmatrix} \, , \]
  corresponding to the time interval $[0,1)$, with the horizontal shear step 
  \[\begin{pmatrix}
    x^1 \\ x^2 
  \end{pmatrix} \mapsto \begin{pmatrix}
                x^1 + L \psi(x^2 + \omega^2) \\ x^2
            \end{pmatrix}
  \]
  corresponding to the time interval $[1,2)$. 
  The proof now amounts to the algebraic step of checking that $\Phi^2 = T_{\omega^2, L} \circ T_{\omega^1, L}$. 
\end{proof}

% \[T^n_{\uo, L} = T_{\omega_n, L} \circ \dots \circ T_{\omega_1, L} \, . \]
% Write $\Lc = \Lc_{\omega, L}$ or $\Lc_{\omega}$, for short, for the \emph{transfer operator} associated with $T = T_{\omega, L}$, which sends
% \[\Lc \psi := \psi \circ T^{-1} \,,  \] 
% noting {\color{red}  connection between compositions of $\Lc$ and solution to transport } 
% we define the sequence of iterated transfer operators
% \[\Lc^n_{\uo, L} \varphi := \varphi \circ (T^n_{\uo, L})^{-1} = \Lc_{\omega_n, L} \circ \dots \circ \Lc_{\omega_1, L} \varphi \,. \]

\paragraph{Plan for Section \ref{sec:negRegMix}} We begin in Section \ref{subsec:prelim2} with some preliminaries on the phase-averaged distribution of Fourier mass after an iterate $\Lc_{\omega, L}$ of the discrete-time propagator defined above. Section \ref{subsc:negRegMix2} concerns the main result of this section, Proposition \ref{thm:neg-reg-mix}, on a version of $H^{-s}$ decay-on-average for the discrete-time maps $T_{\omega, L}$ in a suitably-defined weighted negative regularity (Definition \ref{defn:negAnisotropicNorm2}). Finally, in Section \ref{subsec:quenchedNegRegMix} we prove Theorem \ref{thm:introNegRegMix} and Corollary \ref{cor:introNegRegMixQuenched}. 

\subsection{Preliminaries}\label{subsec:prelim2}

The following Lemma is the primary objective of Section \ref{subsec:prelim2}. 

\begin{lemma}[Main Lemma]\label{lem:cone-decay}
%   Let $p \in L^\infty(\T)$ be a bounded probability density.
  Let $s \in [0,1/2]$. Then, there exists $C = C(\psi) > 0$ such that for any $\theta \geq 1$ and $\varphi \in H^{-s}(\T^2)$ with $\int \varphi = 0$,
  \begin{equation}\label{eq:cone-decay}
    \int_\T \|\Pi_{\mathcal{C}_\theta} \LL_{\omega, \llambda} \varphi\|_{H^{-s}}^2 p(\omega) \dif \omega \leq C\|p\|_{L^\infty} \theta^{\frac{1-2s}{4-2s}}\llambda^{2s - \frac{1-2s}{4-2s}}\|\varphi\|_{H^{-s}}^2.
  \end{equation}
\end{lemma}

The moral of Lemma \ref{lem:cone-decay} is as follows: for a given mean-zero $L^2$ function $\varphi$, a single iterate of $T_\omega$ transmits, with high probability, the vast majority of the Fourier mass of $\varphi$ out of the horizontal cone $\Cc_\theta$. 

% Moral: for a given $L^2$ function $\psi$, a single iterate of $T_\omega$ transfers, with high probability, the bulk of Fourier mass out of the cone $\mathcal{C}_\theta$. 

The proof of Lemma \ref{lem:cone-decay} will use the following two auxiliary results. The first is a standard stationary phase argument for oscillatory integrals -- see, e.g., \cite[Chapter VIII, Proposition 2]{stein1993harmonic}.

\begin{lemma}[Local van der Corput estimate]\label{lem:osc-int}
  Let $I = [a,b] \subset \R$ with $a<b$, let $\rho \in W^{1,1}(I)$, and let $\xi \in C^2(I)$. Let 
  \[
    \mu := \inf_{y \in I} |\xi''(y)| \, , \qquad 
    \nu := \inf_{y \in I} |\xi'(y)| \, . 
  \]
  % Then
  % \begin{equation}\label{eq:vdc-second}
  %   \left|\int_I \rho(y)e^{i\xi(y)}\,\dif y\right|
  %   \leq C\left(1+|I|^{-1}\right)\|\rho\|_{W^{1,1}(I)}\mu^{-1/2},
  % \end{equation}
  % where $C > 0$ is a universal constant. If, in addition,
  % \[
  %   \nu := \inf_{y \in I} |\xi'(y)| > 0,
  % \]
  % then
  % \begin{equation}\label{eq:vdc-first}
  %   \left|\int_I \rho(y)e^{i\xi(y)}\,\dif y\right|
  %   \leq C\left(1+|I|^{-1}\right)\|\rho\|_{W^{1,1}(I)}\nu^{-1}.
  % \end{equation}
  % Consequently,
  Assume that $\mu > 0$. Then, 
  \begin{equation}\label{eq:vdc-min}
    \left|\int_I \rho(y)e^{i\xi(y)}\,dy\right|
    \leq C\left(1+|I|^{-1}\right)\|\rho\|_{W^{1,1}(I)}\min\{\mu^{-1/2},\nu^{-1}\}\, \, .
  \end{equation}
  % whenever both $\mu,\nu > 0$.
\end{lemma}
Above, we follow the convention that $0^{-1} = \infty$.

% \cite[Chapter VIII, Proposition 2]{stein-harmonic-analysis}

% \begin{proof}[Proof sketch] {\color{blue}  bother including? } 
%     The estimate is the standard one-dimensional van der Corput argument. Since $|\xi''|\geq \mu$, $\xi'$ is monotone. For unit amplitude, split at the unique critical point, or at the endpoint where $|\xi'|$ is smallest, on the scale $\mu^{-1/2}$ and integrate by parts off this set; the term $\int |\xi''|/|\xi'|^2$ is the variation of $1/\Phi'$, giving $\sup_{J\subset I}|\int_J e^{i\xi}|\lesssim \mu^{-1/2}$. If $|\xi'|\geq \nu$, the same integration by parts gives the $\nu^{-1}$ bound. A final integration by parts in the amplitude gives the factor $\|\rho\|_\infty+\|\rho'\|_1$, and the displayed $W^{1,1}$ form follows from $\|\rho\|_\infty\lesssim (1+|I|^{-1})\|\rho\|_{W^{1,1}(I)}$.
% \end{proof}

The next lemma will be used for controlling the effect on Fourier mass of masking by a cutoff function. 
For $\chi : \T \to \R$ and $\omega \in \T$ we define $\chi_\omega : \T^2 \to \R$ by
\[\chi_\omega(x) = \chi(x^1 + \omega) \, .\] 

\begin{lemma}\label{lem:neg-ortho}
  Let $\varphi \in H^{-s}(\T^2)$ for $0 \leq s \leq \frac{1}{2}$.
  Assume $\chi \in C^\infty(\T)$ with $\|\chi\|_{L^2}^2 \leq \varepsilon$ and $\|\chi\|_{H^{1/2}} \leq 1$. 
  Then
  \begin{equation*} 
    \int_\T \|\chi_\omega \varphi\|_{H^{-s}}^2 \dif \omega \leq C \varepsilon^{1-2s}\|\varphi\|_{H^{-s}}^2,
  \end{equation*}
% %   where, for each $\omega \in \T$, $\chi_\omega$ is regarded as the smooth function on $\T^2$ depending only on the first coordinate,
%   \begin{equation*}
%     \chi_\omega(x) := \chi(x^1 + \omega),
%   \end{equation*}
%   and $\chi_\omega\varphi$ denotes multiplication of the distribution $\varphi$ by this smooth function. 
  % Here $\dif \omega$ is normalized Haar measure on $\T$, the Sobolev norms are inhomogeneous, and 
  Here, $C$ is independent of $\chi$, $\varphi$, and $\varepsilon$.
\end{lemma}
\begin{proof}
  By density, it suffices to prove the estimate when $\varphi$ is a trigonometric polynomial.
  %  the general case then follows by completing in $H^{-s}$, since the estimate below is uniform.
  For $(j,k) \in \Z^2$ we have
  \begin{align*}
    \widehat{\chi_\omega\varphi}(j, k) & = \sum_{\ell \in \Z} \widehat{\chi_\omega}(j-\ell, 0) \widehat{\varphi}(\ell, k) \\ 
    & = \sum_{\ell \in \Z} 
       e^{2\pi i (j-\ell)\omega}\widehat{\chi}(j-\ell) \widehat{\varphi}(\ell, k) \, ,
  \end{align*}
  % where $\hat{\chi_\omega}(m, 0) = e^{2\pi i m\omega}\hat{\chi}(m)$. 
  % In particular, the product 
  % $\hat{\chi_\omega}(j-\ell,0)\overline{\hat{\chi_\omega}(j-\ell',0)}$ contains the character $e^{2\pi i(\ell'-\ell)\omega}$, and
  hence
  \begin{align*}
    \int_\T |\widehat{\chi_\omega\varphi}(j, k)|^2 \dif \omega &= \sum_{\ell, \ell' \in \Z} \left(\int_\T \widehat{\chi_\omega}(j-\ell, 0) \overline{\widehat{\chi_\omega}(j-\ell', 0)} \dif \omega\right) \widehat{\varphi}(\ell, k) \overline{\widehat{\varphi}(\ell', k)}\\
    &= \sum_{\ell, \ell' \in \Z} \left(\int_\T e^{2\pi i (\ell' - \ell) \omega} \dif \omega\right) \widehat{\chi_0}(j-\ell, 0) \overline{\widehat{\chi_0}(j-\ell', 0)} \widehat{\varphi}(\ell, k) \overline{\widehat{\varphi}(\ell', k)}\\
    &= \sum_{\ell \in \Z}  |\widehat{\chi}(j-\ell)|^2  |\widehat{\varphi}(\ell, k)|^2,
  \end{align*}
  by orthogonality of the characters in the $\omega$ variable.

  Summing over $(j, k)$, we have
  \begin{align*}
    \int_\T \|\chi_\omega \varphi \|_{H^{-s}}^2 \dif \omega &= \sum_{j,k,\ell \in \Z} (1 + j^2 + k^2)^{-s}|\widehat{\chi}(j-\ell)|^2  |\widehat{\varphi}(\ell, k)|^2\\
    &= \sum_{\ell,k \in \Z} \frac{|\widehat{\varphi}(\ell, k)|^2}{(1 + \ell^2 + k^2)^s} \sum_{j \in \Z} \frac{(1 + \ell^2 + k^2)^s}{(1 + j^2 + k^2)^s} |\widehat{\chi}(j-\ell)|^2\\
    &\leq C \|\varphi\|_{H^{-s}}^2 \sup_{\ell ,k \in \Z} \sum_{j \in \Z} \left(1 + (j-\ell)^{2}\right)^s |\widehat {\chi}(j-\ell)|^2 \\
    &\leq C \|\varphi\|_{H^{-s}}^2 \|\chi\|_{H^{s}}^2\\
    &\leq C \|\varphi\|_{H^{-s}}^2 \varepsilon^{1-2s}.
  \end{align*}
  Here, from the second to the third line we used the inequality
  \begin{equation*}
    1 + \ell^2 + k^2 \leq C (1 + j^2 + k^2)(1 + (j-\ell)^2) \, , 
  \end{equation*}
  % and hence
  % \begin{equation*}
  %   \frac{(1 + \ell^2 + k^2)^s}{(1 + j^2 + k^2)^s} \leq C_s(1 + (j-\ell)^2)^s.
  % \end{equation*}
  while from the fourth to fifth lines we interpolated between $H^{1/2}$ and $L^2$.
  %  and used our assumptions on $\| \chi\|_{L^2}, \| \chi\|_{H^{1/2}}$. 
  % Finally, interpolation in the inhomogeneous Sobolev scale gives
  % \begin{equation*}
  %   \|\chi\|_{H^s}^2 \leq \|\chi\|_{L^2}^{2(1-2s)} \|\chi\|_{H^{1/2}}^{4s} \leq \varepsilon^{1-2s}. 
  % \end{equation*}
\end{proof}

\subsubsection*{Proof of Lemma \ref{lem:cone-decay}}

First, we reduce to $p \equiv 1$ by bounding \[\int_\T \| \Pi_{\mathcal{C}_\theta} \mathcal{L}_{\omega, L} \varphi\|_{H^{-s}}^2 p(\omega) \dif \omega \leq \| p\|_{L^\infty} \int_\T \| \Pi_{\mathcal{C}_\theta} \mathcal{L}_{\omega, L} \varphi\|_{H^{-s}}^2 \dif \omega \, . \]  
  % Indeed, by H\"older's inequality, the $p \equiv 1$ case implies the general case.

  Next, let $\varepsilon \in (0,1/2)$, to be taken small as we go along.
  In particular, let $4\varepsilon$ be smaller than the minimum distance between two zeros of $\psi''$.
  Fix an even $\chi \in C^\infty(\R)$ satisfying $0 \leq \chi \leq 1$, $\chi(r) = 0$ for $|r| \leq \frac{1}{2}$, and $\chi(r) = 1$ for $|r| \geq 1$, and set
  \begin{equation*}
    \rho^\varepsilon(y) \coloneq \prod_{\psi''(z) = 0}\chi\left(\frac{\dist_\T(y, z)}{\varepsilon}\right) \, .
  \end{equation*}
  % where we have abused notation somewhat and viewed $\psi : \R \to \R$ by periodic extension. 
  As in Lemma~\ref{lem:neg-ortho}, for $\omega \in \T$ we write
    $\rho^\varepsilon_\omega(x) \coloneq \rho^\varepsilon(x^1+\omega)$
  and view $\rho^\varepsilon_\omega$ as a function on $\T^2$ depending only on the first coordinate.

  Define the restricted operator by
  \begin{equation*}
    \LL_{\omega,\llambda}^\varepsilon\varphi \coloneq \LL_{\omega, \llambda} \left(\rho^\varepsilon_\omega \varphi \right)\,,
  \end{equation*}
  which masks $\varphi$ away from the zeros of $\psi''$.
  We will treat the restricted and unrestricted contributions separately. 

  \paragraph*{Step 1: Restricted contribution.} We begin by estimating $\int \| \Pi_{\Cc_\theta} \Lc_{\omega, L}^\varepsilon \varphi\|^2_{H^{-s}} \dif \omega$. 
  By definition, $\LL_{\omega, \llambda} \varphi(x^1, x^2) = \varphi(x^2, x^1 - \llambda\psi(x^2 + \omega))$. Making this change of variables in Fourier space,
  \begin{align*}
    \LL_{\omega, \llambda}^\varepsilon \varphi(x^1, x^2) &= \sum_{j,k \in \Z} e^{2 \pi i (j x^2 + k(x^1 - \llambda\psi(x^2 + \omega)))}\rho^\varepsilon(x^2+\omega)\widehat{\varphi}(j, k)\\
                         &= \sum_{j, k, \ell \in \Z} e^{2 \pi i ((j+\ell)x^2 + k x^1)}a_{\omega, k\llambda}^\varepsilon(\ell)\widehat{\varphi}(j, k)\\
                         &= \sum_{j,k,\ell \in \Z} e^{2 \pi i (j x^2 + k x^1)}a_{\omega, k\llambda}^\varepsilon(\ell) \widehat{\varphi}(j-\ell, k),
  \end{align*}
  where for $ \eta \in \R$ we set
  \begin{align*} 
    a_{\omega, \eta}^\varepsilon(\ell) & = \int e^{- 2 \pi i \ell y } \rho^{\varepsilon}(y + \omega) e^{- 2 \pi i \eta \psi(y + \omega)} \dif y \\ 
    & = e^{2 \pi i \ell \omega} \left(\rho^\varepsilon e^{- 2 \pi i \eta \psi(\cdot)}\right)^{\wedge}(\ell) 
  \end{align*}
  so that 
  \begin{equation*}
  \rho^\varepsilon(y+\omega)e^{-2 \pi i \eta \psi(y + \omega)} =: \sum_{\ell \in \Z} a_{\omega, \eta}^\varepsilon(\ell)e^{2 \pi i \ell y} \, . 
  \end{equation*}
  For the individual Fourier coefficient at $(k,j)$, we have
  \begin{equation*}
    \widehat{\LL_{\omega, \llambda}^\varepsilon \varphi}(k, j) = \sum_{\ell \in \Z} a_{\omega, k\llambda}^\varepsilon(\ell)\widehat{\varphi}(j-\ell, k).
  \end{equation*}
  Squaring and averaging in $\omega \in \T$, and noting that off-diagonal terms vanish due to orthogonality of the characters $\omega \mapsto e^{2 \pi i \ell \omega}$, we obtain
  \[\int_\T |\widehat{\LL^\varepsilon_{\omega, \llambda} \varphi}(k, j)|^2 \dif \omega = \sum_{\ell \in \Z} | \widehat \varphi(j-\ell, k)|^2 \int_\T |a^\varepsilon_{\omega, kL}(\ell)|^2 \dif \omega \,.\]
  % \begin{equation*}
  %   \int_\T |\widehat{\LL^\varepsilon_{\omega, \llambda} \varphi}(k, j)|^2 \dif \omega = \sum_{\ell, \ell' \in \Z} \hat{\varphi}(j-\ell, k)\overline{\hat{\varphi}(j-\ell', k)}\int_\T a^\varepsilon_{\omega, k\llambda}(\ell)\overline{a^\varepsilon_{\omega, k\llambda}(\ell')} \dif \omega.
  % \end{equation*}
  % We now observe that the off-diagonal terms of the sum vanish: for $\ell \neq \ell'$ we compute 
  % \begin{align}
  %   \int_\T a^\varepsilon_{\omega, k\llambda}(\ell)\overline{a^\varepsilon_{\omega, k\llambda}(\ell')} \dif \omega &= \int_\T \int_\T \int_\T \rho^\varepsilon(y+\omega)\rho^\varepsilon(y'+\omega)e^{- 2 \pi i \left[(\ell y - \ell' y') + k\llambda (\psi(y + \omega) - \psi(y' + \omega))\right]} \dif y \dif y' \dif \omega\label{eq:off-diag-cancel}\\
  %                                                        &= \int_\T \int_\T \int_\T \rho^\varepsilon(y)\rho^\varepsilon(y')e^{- 2 \pi i \left[(\ell y - \ell' y') + (\ell'-\ell)\omega + k\llambda (\psi(y) - \psi(y'))\right]} \dif y \dif y' \dif \omega\nonumber\\
  %                                                        &= \int_\T e^{-2\pi i (\ell'-\ell)\omega} \dif \omega \int_\T \int_\T \rho^\varepsilon(y)\rho^\varepsilon(y')e^{- 2 \pi i \left[(\ell y - \ell' y') + k\llambda (\psi(y) - \psi(y'))\right]} \dif y \dif y'\nonumber\\
  %                                                        &= 0 \, . 
  % \end{align}

  Incorporating the $H^{-s}$ Sobolev weight at $(j,k)$, 
  %  now the average of $\| \Lc^\varepsilon_{\omega, L} \varphi\|_{H^{-s}}^2$, we have
  \begin{align}
    \int_\T \frac{|\widehat{\LL^\varepsilon_{\omega, \llambda} \varphi}(k, j)|^2}{(1 + j^2 + k^2)^s} \dif \omega &= \sum_{\ell \in \Z} \frac{|\widehat{\varphi}(j-\ell, k)|^2}{(1 + j^2 + k^2)^s}\int_\T |a^\varepsilon_{\omega, k\llambda}(\ell)|^2 \dif \omega \nonumber\\
    &= \sum_{\ell \in \Z} \frac{|\widehat{\varphi}(j-\ell, k)|^2}{(1 + (j-\ell)^2 + k^2)^s} \\ 
    & \qquad \cdot \underbrace{\frac{(1 + (j-\ell)^2 + k^2)^s}{(1 + j^2 + k^2)^s} 
     \int_\T 
    %  \left|\int_\T \rho^\varepsilon(y+\omega)e^{-2 \pi i \left[ \ell y + k\llambda \psi(y + \omega)\right]} \dif y\right|^2 
    |a^\varepsilon_{\omega, k L}(\ell)|^2
     \dif \omega}_{=: \Gamma_{j, k, \ell}} \nonumber\\
    &\leq \left(\sup_{\ell \in \Z} \Gamma_{j, k, \ell}\right)\sum_{\ell \in \Z} \frac{|\widehat{\varphi}(j-\ell, k)|^2}{(1 + (j-\ell)^2 + k^2)^s} \nonumber\\
    &= \left(\sup_{\ell \in \Z} \Gamma_{j, k, \ell}\right) \sum_{\ell \in \Z} \frac{|\widehat{\varphi}(\ell, k)|^2}{(1 + \ell^2 + k^2)^s}  \, . \label{eq:phi-linfty-bound}
  \end{align}
%   where we write $\Gamma_{j,k,\ell}$ to denote the integral above.

  We now estimate 
  $\sup_{\ell} \Gamma_{j, k, \ell}$. 
  Fix $\beta > 0$, to be taken sufficiently large in the following estimates. We consider two cases. Below, $(k,j) \in \Cc_\theta$ is fixed (hence $k \neq 0$), and $\ell \in \Z$ is arbitrary.   

	  \subparagraph{Case a.} Suppose that $(j-\ell)^2 \leq \beta(j^2 + k^2)$.
	  Recalling that $\supp(\rho^\epsilon)$ is distance at least $\varepsilon/2$ from $\{ \psi'' = 0\}$, and recalling our nondegeneracy assumption $\{\psi'' = 0\} \cap \{\psi''' = 0\} = \emptyset$, it holds that $|\psi''(\cdot+\omega)| \geq c\varepsilon$ on the support of $\rho^\varepsilon_\omega$ for some $c > 0$ depending only on $\psi$. 
	  
      We seek now to apply Lemma~\ref{lem:osc-int} to 
      \begin{align}
        |a^{\varepsilon}_{\omega, k L}(\ell)| = 
        \left|\int e^{- 2 \pi i \ell y} \rho^\varepsilon(y) e^{- 2 \pi i k L \psi(y)} \dif y\right| \, , 
      \end{align}
      with phase \[\xi(y) :=  - 2 \pi \left( \ell y + k\llambda \psi(y) \right)\] and amplitude \[\rho(y) \mapsto \rho^\epsilon(y ) \, . \] 
      For this, split the support of $\rho^\varepsilon(\cdot)$ into intervals $I$ left after removing the $\varepsilon/2$-neighborhoods of the zeros of $\psi''(\cdot)$; there are  finitely-many such $I$, each of which has length  bounded from below in terms of $\psi$. On each such interval $I$, apply Lemma~\ref{lem:osc-int}; since 
	  \begin{equation*}
	    \inf_I |\xi''| \geq 2 \pi c |k|\llambda\varepsilon
	  \end{equation*}
	  and $\|\rho^\varepsilon(\cdot)\|_{W^{1,1}(\T)} \leq C\varepsilon^{-1}$, the second-derivative estimate gives
		  \begin{equation*}
		    \left|\int_\T \rho^\varepsilon(y)e^{i\xi(y)}\dif y\right|
		    \leq C\varepsilon^{-3/2}(|k|\llambda)^{-1/2} \, .
		    % \leq C\varepsilon^{-2}(|k|\llambda)^{-1/2}.
		  \end{equation*}
      
      For the multiplier in the definition of $\Gamma_{j, k, \ell}$ we bound
          \[\frac{(1 + (j-\ell)^2 + k^2)^s}{(1 + j^2 + k^2)^s} \leq \frac{(1 + (\beta + 1) (j^2 + k^2) )^s}{(1 + j^2 + k^2)^s} \lesssim \beta^s\]
          if $\beta$ is sufficiently large. In all, we obtain
	  \begin{equation*}
	    \Gamma_{j,k,\ell} \leq C \varepsilon^{-3} \beta^s (|k|\llambda)^{-1}.
	  \end{equation*}

  \subparagraph{Case b.} Suppose that $(j-\ell)^2 > \beta(j^2 + k^2)$. Note in this case that $\ell \neq 0$ since $(k, j) \in \mathcal{C}_\theta$. 
  Then $|\xi'| \gtrsim  |\ell| - |k| \llambda \|\psi\|_{C^1}$, while from our assumption we have $2 \ell^2 > (\beta - 2) j^2 + \beta k^2 \geq \beta k^2$ so long as $\beta \geq 2$. Plugging these in gives the lower bound
	  \begin{equation}\label{eq:lowerBoundXi}
	    |\xi'| \gtrsim |\ell| \left(1 - 2\beta^{-1/2}\llambda\|\psi\|_{C^1}\right) \gtrsim |\ell| \, , 
      % = |\ell|\left(1 - C \beta^{-1/2} \llambda \|\psi\|_{C^1}\right).
	  \end{equation}
    where we have specified
    \[\beta =C  \llambda^2 \, ,\] with $C$ chosen appropriately depending on $\psi$.
    
      Applying Lemma \ref{lem:osc-int} as before, separately along the intervals $I$ comprising the support of $\rho^\epsilon(y )$ and this time using the first-derivative lower bound \eqref{eq:lowerBoundXi}, we obtain
	  \begin{equation}\label{eq:caseb}
	    \left|\int_\T \rho^\varepsilon(y)e^{i\xi(y)}\dif y\right|
	    \leq C\|\rho^\varepsilon\|_{W^{1,1}(I)}
	    |\ell|^{-1} \leq C \varepsilon^{-1} |\ell|^{-1} \,. 
	  \end{equation}
    %  so that the RHS of \eqref{eq:caseb} is bounded 
    %     $\leq C \varepsilon^{-1} |\ell|^{-1}$. 

        % Choose $\beta = C \llambda^2\left(1 + \|\psi\|_{C^1}\right)^2$, with $C$ set so that the RHS of \eqref{eq:caseb} is bounded 
        % $\leq C \varepsilon^{-1} |\ell|^{-1}$. 
        Incorporating the Sobolev weight, we have shown thus far that
        \[
        \Gamma_{j,k,\ell} \leq C \frac{(1 + (j-\ell)^2 + k^2)^s}{(1 + j^2 + k^2)^s}|\ell|^{-2} \varepsilon^{-2} \,.
        \]
        % we seek to bound
	      %   \begin{equation*}
	      %     \frac{(1 + (j-\ell)^2 + k^2)^s}{(1 + j^2 + k^2)^s}|\ell|^{-2}.
	      %   \end{equation*}
	        In Case b, $|j|+|k| \leq C\beta^{-1/2}|j-\ell|$ implies $|j-\ell|+|k| \leq C|\ell|$. Therefore
	        \begin{equation*}
	          \frac{(1 + (j-\ell)^2 + k^2)^s}{(1 + j^2 + k^2)^s}|\ell|^{-2s}
	          \leq C \, , 
	        \end{equation*}
	        from which it follows that
	        \begin{equation*}
	          \Gamma_{j,k,\ell} \leq C\varepsilon^{-2}|\ell|^{2s-2} \,. 
	        \end{equation*}
	        Since Case b also gives $|\ell| \geq c \beta^{1/2}|k|$ and $2s-2<0$, we obtain
	        \begin{equation*}
	          \Gamma_{j,k,\ell}
	          \leq C\varepsilon^{-2}\beta^{s-1}|k|^{2s-2}
	          \leq C\varepsilon^{-2}\llambda^{2s-2}|k|^{2s-2} \, .
	          % \leq C\varepsilon^{-2}\llambda^{2s-2}|k|^{-1}.
	        \end{equation*}

% 	    \subparagraph{Consolidating Cases a, b.}  
%         In either case,         
%   \begin{equation*}
%     \Gamma_{j,k,\ell} \leq C\varepsilon^{-3} \llambda^{2s-1} |k|^{-1}.
%   \end{equation*}
% 	  Indeed, Case a gives $\Gamma_{j,k,\ell} \leq C\varepsilon^{-3}\llambda^{-1}|k|^{-1}$, which is bounded by the right-hand side above since $\varepsilon < 1$, $\llambda \geq 1$, and $s \geq 0$. In Case b, this choice of $\beta$ makes the last factor in the preceding display uniformly bounded, and the Case b assumption implies $|\ell|^2 \geq c\beta |k|^2 \geq c\llambda^2|k|^2$. Thus, for $k \neq 0$,
% 	  \begin{equation*}
% 	    \varepsilon^{-2}|\ell|^{-2}
% 	    \leq C\varepsilon^{-2}\llambda^{-2}|k|^{-2}
% 	    \leq C\varepsilon^{-4}\llambda^{2s-1}|k|^{-1},
% 	  \end{equation*}
% 	  which is the required common bound for the modes used in the cone sum below.

\subparagraph{Completing Step 1.}

With the above choice of $\beta$, Case a gives
\begin{equation*}
  \Gamma_{j,k,\ell} \leq C\varepsilon^{-3}\llambda^{2s-1}|k|^{-1}.
\end{equation*}
Case b gives the stronger bound
\begin{equation*}
  \Gamma_{j,k,\ell} \leq C\varepsilon^{-2}\llambda^{2s-2}|k|^{-1}.
\end{equation*}
Since $\varepsilon < 1$, $\llambda \geq 1$, and $s \geq 0$, both cases imply the common estimate
\begin{equation}\label{eq:gamma-common}
  \sup_{\ell \in \Z}\Gamma_{j,k,\ell} \leq C\varepsilon^{-3}\llambda^{2s-1}|k|^{-1}.
\end{equation}
% In particular, we may also use the weaker bound
% \begin{equation*}
%   \sup_{\ell \in \Z}\Gamma_{j,k,\ell} \leq C\varepsilon^{-4}\llambda^{2s-1}|k|^{-1}.
% \end{equation*}
We now sum over $(k,j) \in \mathcal{C}_\theta$, recalling $|j| < \theta |k|$ along this cone. 
  \begin{align}
    \int_\T \|\Pi_{\mathcal{C}_\theta} \LL^\varepsilon_{\omega, \llambda} \varphi\|_{H^{-s}}^2 \dif \omega  &= \int_\T \sum_{(k,j) \in \mathcal{C}_\theta} \frac{|\widehat{\LL^\varepsilon_{\omega, \llambda} \varphi}(k, j)|^2}{(1 + j^2 + k^2)^s} \dif \omega \nonumber\\
    &\leq C \sum_{(k, j) \in \mathcal{C}_\theta} \varepsilon^{-3}\llambda^{2s-1}|k|^{-1}\sum_{\ell \in \Z} \frac{|\widehat{\varphi}(\ell, k)|^2}{(1 + \ell^2 + k^2)^s} \nonumber\\
    &\leq C \theta \varepsilon^{-3} \llambda^{2s-1}\|\varphi\|_{H^{-s}}^2.\label{eq:phi-l1-bound}
	  \end{align}
      In the last line, we have noted that for each fixed $k$ we are summing over $O(\theta |k|)$ values of $j$, exactly cancelling the $|k|^{-1}$.

	%   The last inequality does not lose an infinite factor in the $j$ summation. Indeed, for each fixed $k \neq 0$,
	%   \begin{equation*}
	%     \#\{j \in \Z : (k,j) \in \mathcal{C}_\theta\}
	%     = \#\{j \in \Z : |j| < \theta |k|\}
	%     \leq C\theta |k|,
	%   \end{equation*}
	%   since $\theta \geq 1$ and $|k| \geq 1$. Therefore, writing
	%   \begin{equation*}
	%     A_k \coloneq \sum_{\ell \in \Z}\frac{|\hat{\varphi}(\ell,k)|^2}{(1+\ell^2+k^2)^s},
	%   \end{equation*}
	%   we have
	%   \begin{equation*}
	%     \sum_{(k,j)\in\mathcal{C}_\theta}|k|^{-1}A_k
	%     = \sum_{k\neq 0}|k|^{-1}A_k\#\{j : (k,j)\in\mathcal{C}_\theta\}
	%     \leq C\theta\sum_{k\in\Z}A_k
	%     \leq C\theta\|\varphi\|_{H^{-s}}^2.
	%   \end{equation*}
	%   We note also that $(k, j) \in \mathcal{C}_\theta$ implies $k \neq 0$, so the $|k|^{-1}$ factor above is finite.

% {{\color{red}  QUIT HERE } }

  \paragraph{Step 2.} We turn now to bounding $\E \| \Lc_{\omega, L} \varphi - \Lc^\varepsilon_{\omega, L} \varphi \|_{H^{-s}}^2$, averaging over $\omega$ uniformly distributed in $\T$. 
  % We prove, for any given $\varphi \in H^{-s}$, that the restricted operator $\LL_{\omega, \llambda}^\varepsilon$ and the original operator $\LL_{\omega, \llambda}$ act similarly on $\varphi$, for typical choices of $\omega$.

  First, observe that
	  \begin{equation}\label{eq:phi-no-grow}
	    \|\LL_{\omega, \llambda} f \|_{H^{-s}} \leq C\llambda^{s}\|f\|_{H^{-s}} \, ; 
	  \end{equation}
    in the $s = 0$ case this is immediate from incompressibility, while if $s > 0$ it follows 
	  by interpolation between $L^2(\T^2)$ and $H^{-1}(\T^2)$, using that $\LL_{\omega, L}$ is an $L^2$ isometry and the cheap upper bound $\| \LL_{\omega, L} \|_{H^{-1}} \lesssim L$. 
	  This implies 
  \begin{equation}\label{eq:phi-dif-no-grow}
    \int_\T \|\LL_{\omega, \llambda} \varphi - \LL_{\omega, \llambda}^\varepsilon \varphi\|_{H^{-s}}^2 \dif \omega \leq C\llambda^{2s}\int_\T \|\varphi - \rho^\varepsilon_\omega \varphi\|_{H^{-s}}^2 \dif \omega.
  \end{equation}
  By Lemma~\ref{lem:neg-ortho} with $\chi = 1 - \rho^\varepsilon, \chi_\omega = 1 - \rho^\varepsilon_\omega$ and using that $\|1-\rho^\varepsilon_\omega\|_{L^2}^2 \leq C\varepsilon$ and $\|1-\rho^\varepsilon_\omega\|_{H^{1/2}} \leq C$, we have
  \begin{equation}\label{eq:apply-neg-ortho}
    \int_\T \|\varphi - \rho^\varepsilon_\omega\varphi\|_{H^{-s}}^2 \dif \omega \leq C\varepsilon^{1-2s}\|\varphi\|_{H^{-s}}^2.
  \end{equation}

  Combining~\eqref{eq:phi-dif-no-grow} and~\eqref{eq:apply-neg-ortho} with~\eqref{eq:phi-l1-bound},
  \begin{equation*}
    \int_\T \|\Pi_{\mathcal{C}_\theta} \LL_{\omega, \llambda} \varphi\|_{H^{-s}}^2 \dif \omega \leq C\left(\llambda^{2s}\varepsilon^{1-2s} + \theta \varepsilon^{-3}\llambda^{2s-1}\right)\|\varphi\|_{H^{-s}}^2.
  \end{equation*}
  If $\llambda \geq 16 \theta$ then we choose $\varepsilon \approx \left(\frac{\theta}{\llambda}\right)^{1/(4-2s)}$. This balances the two terms since
  \begin{equation*}
    \llambda^{2s}\varepsilon^{1-2s}
    \approx
    \theta \varepsilon^{-3}\llambda^{2s-1}
    \approx
    \theta^{\frac{1-2s}{4-2s}}\llambda^{2s-\frac{1-2s}{4-2s}}.
  \end{equation*}
  If $L < 16 \theta$, then \eqref{eq:cone-decay} is immediate from the deterministic bound \eqref{eq:phi-no-grow}.
% \end{proof}

We have completed the proof of Lemma \ref{lem:cone-decay}.

\subsection{Negative regularity mixing}\label{subsc:negRegMix2}

The following weighted negative-regularity norm is a crucial tool in our approach to showing negative-regularity mixing. 

\begin{definition}\label{defn:negAnisotropicNorm2}
  For $\varphi \in H^{-s}(\T^2)$, $s\in\R$ and $\gamma>0$, define
\begin{equation*}
  \|\varphi\|_{H^{-s}_{\theta, \gamma}}^2 \coloneq \gamma \|\Pi_{\mathcal{C}_\theta} \varphi\|_{H^{-s}}^2 + \|\Pi_{\Z^2 \setminus \mathcal{C}_\theta} \varphi\|_{H^{-s}}^2 \, . 
\end{equation*}  
\end{definition}

The following states that, for suitably large $\gamma$ and $\theta$, the $H^{-s}_{\theta, \gamma}$ norm of a scalar contracts dramatically in a single iterate. 
% We also abbreviate
% \begin{equation*}
%   \|\varphi\|_{H^{-s}_\llambda} \coloneq \|\varphi\|_{H^{-s}_{\llambda^{3/4}, \llambda^{4s}}}
% \end{equation*}
% as a convenient norm that satisfies the next theorem with $\varepsilon = 1/2$.

\begin{proposition}[Negative-regularity mixing]\label{thm:neg-reg-mix}
  % Let $p \in L^\infty(\T)$ be a bounded probability density.
  There is a constant $C = C(\psi, p) > 0$ such that for any $\varepsilon \in (0, 1)$ there is $s_0 = s_0(\varepsilon) > 0$ such that for any $s \in [0,s_0]$ and any $\varphi \in H^{-s}(\T^2)$ with mean zero,
  \begin{equation}\label{eq:negRegMix22}
    \int \|\LL_{\omega, \llambda} \varphi\|_{H^{-s}_{\theta, \gamma}}^2 p(\omega) \dif \omega \leq C  \llambda^{(\varepsilon-2)s}\|\varphi\|_{H^{-s}_{\theta, \gamma}}^2 \, , 
  \end{equation}
  where 
  \begin{align}\label{eq:setGammaTheta}
    \text{$\gamma \coloneq \llambda^{4s}$ and $\theta \coloneq \llambda^{1-\varepsilon/2}$.}
  \end{align}
\end{proposition}

\begin{proof}[Proof of Proposition \ref{thm:neg-reg-mix}]
  As in the beginning of the proof of Lemma \ref{lem:cone-decay}, it suffices to prove the result for $p \equiv 1$. If $s = 0$ the bound is trivial from incompressibility; from here on, let $s \in (0,1/2)$ be fixed, with a tighter upper bound on $s$ to be specified later on. By density, it suffices to consider the case when $\varphi$ is a trigonometric polynomial.
  %  In the coming proof, we emphasize that $C$ constants only ever depend on $p$ and $\psi$, and not on $\theta$. 

  First, we record the following deterministic bound for the mass originating in $\mathcal{C}_\theta$ (equation \eqref{eq:phi-no-grow} with $\varphi \mapsto \Pi_{\Cc_\theta} \varphi$): 
  \begin{align}\label{eq:startInsideCone}
    \|\LL_{\omega, \llambda} \Pi_{\mathcal{C}_\theta} \varphi\|_{H^{-s}}^2 \leq C\llambda^{2s} \|\Pi_{\mathcal{C}_\theta} \varphi\|_{H^{-s}}^2 \,. 
  \end{align}
  Next, we use Lemma~\ref{lem:cone-decay} to bound the mass that ends inside $\mathcal{C}_\theta$ by
  \begin{equation*}
    \int_\T \|\Pi_{\mathcal{C}_\theta} \LL_{\omega, \llambda} \varphi\|_{H^{-s}}^2 \dif \omega \leq C \theta^{\frac{1-2s}{4-2s}}\llambda^{(2s - \frac{1-2s}{4-2s})}\|\varphi\|_{H^{-s}}^2 = C L^{\alpha_s} \| \varphi \|_{H^{-s}}^2
  \end{equation*}
  on applying the replacement {$\theta = L^{1 - \varepsilon/2}$} and setting \[\alpha_s := 2 s - \frac{\varepsilon(1 - 2 s)}{8 - 4 s} \, . \]

  It remains to estimate $\|\Pi_{\Z^2 \setminus \mathcal{C}_\theta} \LL_{\omega, \llambda} \Pi_{\Z^2 \setminus \mathcal{C}_\theta} \varphi\|_{H^{-s}}^2$, capturing Fourier mass starting and ending away from $\Cc_\theta$. 
  
  \begin{claim}
    For fixed $k \in \Z$, we have the Parseval-type identity
    % Next, we observe from the form of $T_\omega$ that for each fixed $k \in \Z$,
  \begin{equation}\label{eq:parsevalType}
    \sum_{j \in \Z} |\widehat{\varphi}(j, k)|^2 = \sum_{j \in \Z} |\widehat{\LL_{\omega, \llambda}\varphi}(k, j)|^2 \, .
  \end{equation}
  \end{claim}
  \begin{proof}[Proof of Claim]
    Expanding the Fourier series of $\varphi$ in the given expression $\Lc_{\omega, L} \varphi(x^1, x^2) = \varphi(x^2, x^1 - L \psi(x^2 + \omega))$ yields
    \begin{align}
      \Lc_{\omega, L} \varphi & = \sum_k e^{2 \pi i k (x^1 - L \psi(x^2 + \omega))} \sum_j \widehat \varphi(j, k) e^{2 \pi i j x^2} \, , 
    \end{align}
    from which we obtain that for fixed $k$, 
    \begin{gather}
      \sum_j \widehat{\Lc_{\omega, L} \varphi}(k, j) e^{2 \pi i j x^2}= f_k(x^2) \, , \\ 
      f_k(x^2) := e^{- 2 \pi i k L \psi(x^2 + \omega)} \sum_j \widehat \varphi(j, k) e^{2 \pi i j x^2} \,. 
    \end{gather}
    The desired identity follows from Parseval's theorem applied to $y \mapsto f_k(y)$. 
  \end{proof}
  
  % To start, we note that for fixed $k \in \Z$, we have the Parseval-type identity
  %   % Next, we observe from the form of $T_\omega$ that for each fixed $k \in \Z$,
  % \begin{equation}\label{eq:parsevalType}
  %   \sum_{j \in \Z} |\hat{\varphi}(j, k)|^2 = \sum_{j \in \Z} |\widehat{\LL_{\omega, \llambda}\varphi}(k, j)|^2 \, .
  % \end{equation}
  % Equation \eqref{eq:parsevalType} follows from the form of $T_{\omega, L}$ and the observation that the Fourier coefficients of $\mathcal L_{\omega,L}\varphi$ with first
  % frequency $k$ are precisely the one-dimensional Fourier coefficients, in the $x^2$ variable, of
  % \begin{equation*}
  %   x^2 \mapsto e^{-2\pi i kL\psi(x^2+\omega)} f_k(x^2)  \, , \qquad f_k(y) := \sum_{j \in \Z} \hat \varphi(j, k) e^{2 \pi i j y }\, .
  % \end{equation*}
  
  Write $f = \Pi_{\Z^2 \setminus \mathcal{C}_\theta} \varphi$ for short. Note that for modes $(k, j) \in \Z^2 \setminus \mathcal{C}_\theta$, one has $|j| \geq \theta |k|$, hence $1 + j^2 + k^2 \geq 1 + (1 + \theta^2) k^2$ and so for each fixed $k \in \Z$, 
  \begin{align}\label{eq:singleKWeightedEst} \sum_{\substack{j\in\Z\\ (k,j)\notin\mathcal C_\theta}}
  \frac{\left|\widehat{\mathcal L_{\omega,L}f}(k,j)\right|^2}
       {(1+j^2+k^2)^s} & \leq
  \left(1+(1+\theta^2)k^2\right)^{-s}
  \sum_{j\in\Z}\left|\widehat{\mathcal L_{\omega,L}f}(k,j)\right|^2 \nonumber \\ 
& = \left(1+(1+\theta^2)k^2\right)^{-s}
  \sum_{j \in \Z} |\widehat{f}(j, k)|^2 \, . 
\end{align}

Recall that the Fourier support of $f$ is contained in $\Z^2 \setminus \mathcal{C}_\theta$; for $(j,k)$ in this support, we have the bound $|j| \leq \theta^{-1} |k|$.
Thus, at $k = 0$ one has $\sum_{j \in \Z} |\widehat{f}(j, k)|^2 = 0$ since $f$ is mean-free. For $k \neq 0$, $1 + j^2 + k^2 \leq 1 + (1 + \theta^{-2}) k^2$, hence the RHS of \eqref{eq:singleKWeightedEst} is bounded
\begin{align*}
  &\leq \left(\frac{1 + (1 + \theta^{-2}) k^2}{1 + (1 + \theta^2) k^2}\right)^s \sum_{j \in \Z} \frac{|\widehat{f}(j, k)|^2 }{(1 + j^2 + k^2)^s}  \\ 
    &\lesssim \left(\frac{1 + \theta^{-2}}{1 + \theta^2}\right)^s \sum_{j\in\Z}\frac{|\widehat{f}(j, k)|^2 }{(1 + j^2 + k^2)^s}
\end{align*}
for all $k \in \Z$. Summing over $k \in \Z$ gives
\begin{equation*}
    \|\Pi_{\Z^2 \setminus \mathcal{C}_\theta} \LL_{\omega, \llambda} \Pi_{\Z^2 \setminus \mathcal{C}_\theta} \varphi\|_{H^{-s}}^2 \leq C \left(\frac{1 + \theta^{-2}}{1 + \theta^2}\right)^s \|\Pi_{\Z^2 \setminus \mathcal{C}_\theta} \varphi\|_{H^{-s}}^2 \, .
  \end{equation*}

  Combining the above with \eqref{eq:startInsideCone}, we conclude the following (deterministic) estimate for all mass that ends outside $\mathcal{C}_\theta$: 
  \begin{align*}
    \|\Pi_{\Z^2 \setminus \mathcal{C}_\theta} \LL_{\omega, \llambda} \varphi\|_{H^{-s}}^2 &\lesssim \left(\|\Pi_{\Z^2 \setminus \mathcal{C}_\theta} \LL_{\omega, \llambda} \Pi_{\mathcal{C}_\theta} \varphi\|_{H^{-s}}^2 + \|\Pi_{\Z^2 \setminus \mathcal{C}_\theta} \LL_{\omega, \llambda} \Pi_{\Z^2 \setminus \mathcal{C}_\theta} \varphi\|_{H^{-s}}^2\right) \\
    &\lesssim  \llambda^{2s}\|\Pi_{\mathcal{C}_\theta} \varphi\|_{H^{-s}}^2 + \left(\frac{1 + \theta^{-2}}{1 + \theta^2}\right)^s\|\Pi_{\Z^2 \setminus \mathcal{C}_\theta} \varphi\|_{H^{-s}}^2 \,.
  \end{align*}

  % Grouping terms and choosing $\theta$ such that $\theta^2 = \llambda^{2-\varepsilon}$ we find
  Plugging in the above estimates and using that $\left(\frac{1 + \theta^{-2}}{1 + \theta^2}\right)^s \lesssim L^{s (\varepsilon - 2)}$, we obtain
  \begin{align*}
  \int_\T
  \left\|\mathcal L_{\omega,L}\varphi
  \right\|_{H^{-s}_{\theta,\gamma}}^2\,d\omega
  &\lesssim
  \gamma L^{\alpha_s}
  \left(
    \left\|\Pi_{\mathcal C_\theta}\varphi\right\|_{H^{-s}}^2
    +
    \left\|\Pi_{\Z^2\setminus\mathcal C_\theta}\varphi\right\|_{H^{-s}}^2
  \right) \\
  &\qquad
  +
    L^{2s}\left\|\Pi_{\mathcal C_\theta}\varphi\right\|_{H^{-s}}^2 
  +L^{(\varepsilon-2)s}
  \left\|\Pi_{\Z^2\setminus\mathcal C_\theta}\varphi\right\|_{H^{-s}}^2 \, . 
\end{align*}
  Grouping terms on the RHS gives the expression
  \begin{gather*}
     \left(\gamma L^{\alpha_s} + L^{2s}\right) \| \Pi_{\mathcal{C}_\theta} \varphi\|_{H^{-s}}^2 + (\gamma L^{\alpha_s} + L^{(\varepsilon - 2) s}) \| \Pi_{\Z^2 \setminus \mathcal{C}_\theta} \varphi\|_{H^{-s}}^2 \\ 
     \leq \max\{ L^{\alpha_s} + \gamma^{-1} L^{2s} , \gamma L^{\alpha_s} + L^{(\varepsilon - 2) s}\} \| \varphi \|^2_{H^{-s}_{\theta, \gamma}} \,. 
  \end{gather*}
	Plug in $\gamma = L^{4s}$. We now specify $s$ sufficiently small so that 
	\[4 s + \alpha_s \leq (\varepsilon - 2) s \, , \]
	which is possible since the above LHS evaluates as $- \varepsilon / 8$ at $s = 0$, and conclude.
\end{proof}

Later on, we will iterate \eqref{eq:negRegMix22}. Since the $H^{-s}_{\theta, \gamma}$ norm used above is badly comparable to $H^{-s}$, the following variant of Proposition \ref{thm:neg-reg-mix} will be useful for `rectifying' an initial $\varphi \in L^2$. 

\begin{lemma}\label{lem:intialBootstrap2}
  Assume the setting and notation of Proposition \ref{thm:neg-reg-mix}. Then, for any mean-zero $\varphi \in L^2$, one has 
  \begin{align}\label{eq:initialBootstrap2}
    \int \| \mathcal{L}_{\omega, L} \varphi \|_{H^{-s}_{\theta, \gamma}}^2 p(\omega) \dif \omega \leq C \| \varphi\|_{L^2}^2 \,. 
  \end{align}
\end{lemma}
\begin{proof}
    % We conclude with a summary of the $L^2$ upper bound in \eqref{eq:negRegMix22}.
  To start, we have the easy bound
  \[\| \Pi_{\Z^2 \setminus \mathcal{C}_\theta} \mathcal{L}_{\omega, L } \varphi \|_{H^{-s}}^2 \leq \| \varphi \|_{L^2}^2\]
  by volume-preservation of $T_{\omega, L}$. 
  Next, for the $\mathcal{C}_\theta$ contribution of $\mathcal{L}_{\omega, L} \varphi$ we use Lemma \ref{lem:cone-decay}: 
  \[\int \| \Pi_{\mathcal{C}_\theta} \mathcal{L}_{\omega, L} \varphi \|_{H^{-s}}^2 \dif \omega \leq C L^{\alpha_s} \| \varphi\|_{H^{-s}}^2 \leq C L^{\alpha_s} \| \varphi\|_{L^2}^2 \, .\]
  We have already arranged for $s$ to be taken sufficiently small so that $\gamma L^{\alpha_s} \leq L^{s (\varepsilon - 2)} \lesssim 1$, hence the weighted $H^{-s}_{\theta, \gamma}$-norm is bounded $\lesssim \| \varphi\|_{L^2}^2$. This completes the proof. 
\end{proof}

\subsection{Annealed and quenched exponential mixing}\label{subsec:quenchedNegRegMix}

% Proofs of Theorem \ref{thm:introNegRegMix} and Corollary \ref{cor:introNegRegMixQuenched}

We turn now to time-iterated annealed and quenched negative-regularity estimates. The forms we present below easily imply Theorem \ref{thm:introNegRegMix} and Corollary \ref{cor:introNegRegMixQuenched}, but are  more detailed to keep track of exponents in a way we will use in the remainder of the paper. 

Throughout Section \ref{subsec:quenchedNegRegMix}, $(\varphi_t)$ refers to the \emph{undriven} passively transported scalar from 
\begin{align}\label{eq:PSAundriven2}
  \partial_t \varphi_t + u \cdot \nabla \varphi_t = 0 \,, \qquad \varphi(0, \cdot) = \varphi_0
\end{align}
for some fixed mean-free $\varphi_0$ (possibly a negative-regularity distribution), where $u = u_\uo(t, x)$ is as in \eqref{eq:velField1}, and where $\psi$ and $\uo$ are as in assumptions (ND) and (P), respectively. 

% We first record exponent-exposed versions of the annealed and quenched
% % negative-regularity estimates.  These are the forms used later in the
% sample-wise Batchelor argument; the statements in the introduction follow by
% specializing parameters.

% {\color{red}Shift convention to fold into notation later: for the full
% passive-scalar phase sequence $\uo=(\omega_i^\ell)_{i\geq1,\ell=1,2}$, write
% $\theta\uo=(\omega_{i+1}^\ell)_{i\geq1,\ell=1,2}$.  Thus
% $\Sc^{2j,2j+2m}_{\uo}=\Sc^{0,2m}_{\theta^j\uo}$.  This shift notation is
% unrelated to the cone/aniso parameter also denoted $\theta$.

% (done in section 1 -- will use $\sigma$ for this notation)
% }

We begin with Theorem \ref{thm:introNegRegMix}, restated below. 

\begin{theorem}[Decay on average, unforced passive scalar]\label{thm:neg-reg-flexible}
  There exist universal $s, \mathfrak{c}, \mathfrak{d} > 0$ and constants $C = C(\psi, p), L_0 = L_0(\psi, p) \geq 1$ such that the following hold for all $L \geq L_0$, mean-free $\varphi_0 \in H^{-s}$ and $t \geq 0$: 
  \begin{align}\label{eq:timeTest2}
    \left(\E \| \varphi_t \|_{H^{-s}}^2\right)^{1/2} \leq C L^{\mathfrak{d} - t \mathfrak{c}} \| \varphi_0\|_{H^{-s}}
  \end{align}
\end{theorem}

\begin{remark}\label{rmk:explicitConsts2}
  We briefly note that it is straightforward to check from the coming proof that one can take the explicit values 
  \begin{align}
    s=\frac1{60},\qquad
      \mathfrak c=\frac1{120},\qquad
      \mathfrak d=4s+2\mathfrak c=\frac1{12} \,. 
  \end{align}
  This fact is not used elsewhere. 
\end{remark}

% {\color{blue}  to add: can take $\kappa$ like 0.008 or so from this proof. likely highly nonoptimal. } 

% {\color{blue}  to do: replace $\kappa$ with something else. will need kappa for other purposes later on, like diffusivity } 

\begin{proof}[Proof of Theorem \ref{thm:neg-reg-flexible}]
  Fix $\varepsilon = \frac{99}{100}, \delta = \frac{1}{100}$ and let $s$ be sufficiently small as in Proposition \ref{thm:neg-reg-mix}. 
  % \[
  %   \varepsilon=\frac{99}{100},\qquad
  %   \delta=\frac1{100},\qquad
  %   s=\frac1{60}.
  % \]
  % These choices are admissible in Proposition~\ref{thm:neg-reg-mix}.
  % Indeed, the only additional smallness condition imposed on $s$ at the end
  % of its proof is
  % \[
  %   4s+\alpha_s\leq(\varepsilon-2)s,
  %   \qquad
  %   \alpha_s=2s-\frac{\varepsilon(1-2s)}{8-4s}.
  % \]
  % For the choices above,
  % \[
  %   4s+\alpha_s-(\varepsilon-2)s
  %   =-\frac{2711}{714000}<0.
  % \]
  Set
  \[
    \theta=L^{1-\varepsilon/2},\qquad
    \gamma=L^{4s},\qquad
    \mathfrak c:=\frac{(2-\varepsilon-\delta)s}{2} \,. 
    % =\frac1{120},
  \]
  % and abbreviate
  % $\|\cdot\|_A=\|\cdot\|_{H^{-s}_{\theta,\gamma}}$.

  Let $C_0=C_0(\psi,p)$ be the constant in
  Proposition~\ref{thm:neg-reg-mix}.  After increasing
  $L_0=L_0(\psi,p)$, we may assume that
  $C_0L^{-\delta s}\leq1$ for every $L\geq L_0$.  Proposition
  \ref{thm:neg-reg-mix} then gives
  \begin{equation}\label{eq:one-step-explicit-intro}
    \int_\T\|\Lc_{\omega,L}h\|_{H^{-s}_{\theta, \gamma}}^2p(\omega)\dif\omega
    \leq
    L^{(\varepsilon-2+\delta)s}\|h\|_{H^{-s}_{\theta, \gamma}}^2
    =L^{-2\mathfrak c}\|h\|_{H^{-s}_{\theta, \gamma}}^2.
  \end{equation}
  Iterating, and using that $\mathcal{S}^{0,2n} = \Lc_{\omega_n^2, L} \circ \Lc_{\omega_n^1, L} \circ \dots \circ \Lc_{\omega_1^2, L} \circ \Lc_{\omega_1^1, L}$, $\mathcal{S}^{0,2n} \varphi_0 = \varphi_{2n}$, we obtain
  % To iterate this estimate, flatten the phase sequence by writing
  % $\zeta_{2j-1}=\omega_j^1$ and $\zeta_{2j}=\omega_j^2$, and set
  % \[
  %   X_m:=\Lc_{\zeta_m,L}\circ\cdots\circ\Lc_{\zeta_1,L}\varphi_0,
  %   \qquad
  %   \mathcal F_m:=\sigma(\zeta_1,\ldots,\zeta_m).
  % \]
  % Since $X_{m-1}$ is $\mathcal F_{m-1}$-measurable and $\zeta_m$ is
  % independent of $\mathcal F_{m-1}$, applying
  % \eqref{eq:one-step-explicit-intro} conditionally gives
  % \[
  %   \E\!\left(\left.\|X_m\|_A^2\,\right|\mathcal F_{m-1}\right)
  %   \leq L^{-2\mathfrak c}\|X_{m-1}\|_A^2.
  % \]
  % Repeated conditioning, together with the even-time realization of the
  % propagator as a composition of $2n$ transfer operators, yields
  \[
    \E\|\varphi_{2n}\|_{H^{-s}_{\theta, \gamma}}^2
    \leq L^{-4n\mathfrak c}\|\varphi_0\|_{H^{-s}_{\theta, \gamma}}^2 \, , 
  \]
  % Since $\gamma=L^{4s}\geq1$, the weighted and unweighted norms satisfy
  % \[
  %   \|h\|_{H^{-s}}
  %   \leq\|h\|_A
  %   \leq L^{2s}\|h\|_{H^{-s}}.
  % \]
  hence
  \begin{equation}\label{eq:even-time-explicit-intro}
    \left(\E\|\varphi_{2n}\|_{H^{-s}}^2\right)^{1/2}
    \leq
    L^{2s-2n\mathfrak c}\|\varphi_0\|_{H^{-s}}
  \end{equation}
  on using that $\|h\|_{H^{-s}} \leq\|h\|_{H^{-s}_{\theta, \gamma}} \leq L^{2s}\|h\|_{H^{-s}}$.  
  It remains to pass from even to arbitrary times, which we do with the deterministic estimate
  \begin{align} \label{eq:deterministFracTime2}
    \|\Sc^{2n,t}h\|_{H^{-s}}
    \leq C L^{2s}\|h\|_{H^{-s}} \, ,
\end{align}
  which holds for $t \in [2n, 2n+2)$. The proof of Theorem \ref{thm:neg-reg-flexible} is complete on taking $\mathfrak{d} = 4s + 2 \mathfrak{c}$. 

  Since we will use it shortly, we briefly record the following variant of \eqref{eq:timeTest2}, giving a tighter bound at integer times at the cost of initializing \eqref{eq:PSAundriven2} at a mean-free $\varphi_0 \in L^2$: 
  \begin{align}\label{eq:bootstrapL2Start2}
    \left(\E \| \varphi_{2n}\|_{H^{-s}}^2\right)^{1/2} \leq C L^{- (2n-1)\mathfrak{c}} \| \varphi_0\|_{L^2} \,. 
  \end{align}
  Equation \eqref{eq:bootstrapL2Start2} follows from the same iteration as before, except that when estimating $\| \mathcal{L}_{\omega, L} \varphi_0\|_{H^{-s}_{\theta, \gamma}}$ one uses the $L^2$ upper bound from \eqref{eq:initialBootstrap2} in Lemma \ref{lem:intialBootstrap2}, at the cost of one less iterate of $T_{\omega, L}$.
  This also proves Theorem \ref{thm:introNegRegMix} with the same
  $\mathfrak c,\mathfrak d$: for $n\geq1$ and $2n\leq t<2n+2$,
  \eqref{eq:deterministFracTime2} and \eqref{eq:bootstrapL2Start2} give
  \[
    \E\|\varphi_t\|_{H^{-s}}^2
    \leq CL^{4s-2(2n-1)\mathfrak c}\|\varphi_0\|_2^2
    \leq CL^{\mathfrak d-t\mathfrak c}\|\varphi_0\|_2^2.
  \]
  For $0\leq t<2$, use $L^2$ isometry. \qedhere

  % \begin{align*}
  %   \left(\E\|\varphi_t\|_{H^{-s}}^2\right)^{1/2}
  %   &\leq
  %   C L^{2s}
  %   \left(\E\|\varphi_{2n}\|_{H^{-s}}^2\right)^{1/2}\\
  %   &\leq
  %   C L^{4s-2n\mathfrak c}\|\varphi_0\|_{H^{-s}}\\
  %   &\leq
  %   C L^{4s+2\mathfrak c-t\mathfrak c}
  %   \|\varphi_0\|_{H^{-s}}.
  % \end{align*}
  % Thus the theorem holds with the explicit universal values
  % \[
  %   \boxed{
  %     s=\frac1{60},\qquad
  %     \mathfrak c=\frac1{120},\qquad
  %     \mathfrak d=4s+2\mathfrak c=\frac1{12}.}
  % \]
\end{proof}

% \begin{theorem}[annealed mixing, exponent-exposed]%\label{thm:neg-reg-flexible}
%   Fix $\varepsilon\in(0,1)$ and $0<\delta<2-\varepsilon$.  Let
%   $s_0=s_0(\varepsilon)$ be as in Proposition \ref{thm:neg-reg-mix}.  For
%   $0<s\leq s_0$, set
%   \[
%     \theta=L^{1-\varepsilon/2},\qquad
%     \gamma=L^{4s},\qquad
%     \kappa=\frac{(2-\varepsilon-\delta)s}{2}.
%   \]
%   There are constants $C=C(\psi,p,\delta)>0$ and
%   $L_0=L_0(\psi,p,\delta,s)>0$ such that for all $L\geq L_0$, all mean-zero
%   $\varphi_0\in L^2$, and all integers $n\geq1$,
%   \begin{equation}\label{eq:flexible-L2-annealed}
%     \left(\E\|\Sc^{0,2n}\varphi_0\|_{H^{-s}_{\theta,\gamma}}^2\right)^{1/2}
%     \leq
%     C L^{-(2n-1)\kappa}\|\varphi_0\|_{L^2}.
%   \end{equation}
%   In particular, for every $k\in\Z^2\setminus\{0\}$,
%   \begin{equation}\label{eq:flexible-fourier-mode}
%     \E\|\Sc^{0,2n}e_k\|_{H^{-s}}
%     \leq
%     C L^{-(2n-1)\kappa}.
%   \end{equation}
%   Also, for all mean-zero $\varphi_0\in H^{-s}$,
%   \begin{equation}\label{eq:flexible-Hminus-annealed}
%     \E\|\Sc^{0,2n}\varphi_0\|_{H^{-s}}
%     \leq
%     C L^{2s}L^{-2n\kappa}\|\varphi_0\|_{H^{-s}}.
%   \end{equation}
%   For $t=2n+q$ with $n\in\N$ and $q\in[0,2)$, one has the corresponding
%   continuous-time bound
%   \begin{equation}\label{eq:flexible-Hminus-continuous}
%     \E\|\Sc^{0,t}\varphi_0\|_{H^{-s}}
%     \leq
%     C L^{4s+2\kappa}L^{-\kappa t}\|\varphi_0\|_{H^{-s}}.
%   \end{equation}
% \end{theorem}

We turn now to quenched mixing estimates. The following is a version of the developments in \cite[Section 7]{bedrossianAlmostsureExponentialMixing2019}; see also the more general framework presented in \cite{dewitt2025equivalence}.

\begin{corollary}\label{cor:quenchedMix2}
  Assume the setting and notation of Theorem \ref{thm:neg-reg-flexible}. Let $q \geq 1$ and fix 
  \[{0 < \mathfrak{c}' < \frac{\mathfrak{c}}{6q} }\,, \quad \mathfrak{d}' = 2s+2\mathfrak c' \, . \]
  Then, there exist $C = C(\psi, p, q, \mathfrak{c}')$ and a random variable $D = D_{\mathfrak{c}'} : \Omega \to [1,\infty)$ with
  \[\| D\|_{L^q(\Omega)} \leq C \]
  such that for all mean-zero $\varphi_0 \in H^1$ and $t \geq 0$,
  \begin{align}
    \| \varphi_t\|_{H^{-s}} \leq D L^{\mathfrak{d}'  - \mathfrak{c}' t} \| \varphi_0\|_{H^1} \, . 
  \end{align}
\end{corollary}

% {\color{blue}  mention BBPSII, dolgopyat-dewitt short paper... } 

% \cite{dewitt2025equivalence}

\begin{proof}
  Recall that for $k \in \Z^2$, $e_k(x) = e^{2 \pi i k \cdot x}$ is the standard Fourier basis element. We begin by noting that for $m \geq 1$, equation \eqref{eq:bootstrapL2Start2} with $\varphi_0 \mapsto e_k, k \neq 0$ implies
  \begin{align}
    \E \| \mathcal{S}^{0,2m} e_k \|_{H^{-s}} \leq \left(\E \| \mathcal{S}^{0,2m} e_k\|_{H^{-s}}^2\right)^{1/2} \leq C L^{-(2m-1) \mathfrak{c} } \, .
  \end{align}
  Combined with the fact that $ 0\leq \|\Sc^{0,2m}_{\uo}e_k\|_{H^{-s}}\leq \|e_k\|_{L^2}=1$, it follows that 
    \begin{align}\label{eq:modeByModeLqEst2}
    \left\|\|\Sc^{0,2m}_{\uo}e_k\|_{H^{-s}}\right\|_{L^q(\Omega)}
    \leq
    \left(\E\|\Sc^{0,2m}_{\uo}e_k\|_{H^{-s}}\right)^{1/q}
    \leq
    C L^{-(2m-1) \mathfrak{c} / q} \,. 
\end{align}

  Define $a = 3 \mathfrak{c'}$, noting $0 < a < \mathfrak{c}/2q$, and 
  \begin{equation}\label{eq:two-step-random-constant}
    D(\uo)
    :=
    1+
    \sup_{m\geq1}
    L^{2ma}
    \sum_{k\in\Z^2\setminus\{0\}}
    (1+|k|^2)^{-3/2}
    \|\Sc^{0,2m}_{\uo}e_k\|_{H^{-s}}.
  \end{equation}
  From \eqref{eq:modeByModeLqEst2} it follows that 
  \begin{align*}
    \left\|D-1\right\|_{L^q(\Omega)}
    &\leq
    \sum_{m\geq1}
    L^{2ma}
    \sum_{k\in\Z^2\setminus\{0\}}
    (1+|k|^2)^{-3/2}
    \left\|\|\Sc^{0,2m}_{\uo}e_k\|_{H^{-s}}\right\|_{L^q(\Omega)} \\ 
    & \leq C \sum_{m\geq1}
    L^{2ma}
    \sum_{k\in\Z^2\setminus\{0\}}
    (1+|k|^2)^{-3/2} L^{-(2m-1)\mathfrak{c} / q} \\ 
    & \leq C L^{ \mathfrak{c} / q} \sum_{m\geq1}
    L^{2m (a - \mathfrak{c} / q)} \\ 
    & \leq C L^{2 a - \mathfrak{c}/q} \, ,
    \end{align*}
  hence $\| D \|_{L^q(\Omega)} \leq C$ uniformly in $L$. 

  Next, we observe that for mean-free $g \in H^3$, $| \widehat g(k)| \leq (1 + |k|^2)^{-3/2} \| g \|_{H^3}$ for all $k$, hence
  \[\| \mathcal{S}^{0,2m} g\|_{H^{-s}} \leq D L^{-2ma} \| g \|_{H^3} \,. \]
  A standard argument (e.g., by mollification) implies now that 
  \[\| \mathcal{S}^{0,2m} g \|_{H^{-s}} \leq C D L^{-2m(a/3)} \| g \|_{H^1} \, . \]
  Absorbing the above constant into the definition of $D$ and recalling $\mathfrak{c}' = a/3$ completes the proof at even times $t = 2m, m \geq 1$. For $t\in[2m,2m+2)$, \eqref{eq:deterministFracTime2} gives the factor
  $L^{2s-2m\mathfrak c'}\leq L^{2s+2\mathfrak c'-t\mathfrak c'}$.
  The interval $0\leq t<2$ follows from $L^2$ isometry, proving the stated choice of $\mathfrak d'$.

\end{proof}

\begin{remark}
  Since it is used elsewhere, we record the following somewhat stronger estimate at even integer times: for mean-free $g \in H^1$ and $m \geq 1$, it holds that 
\begin{align}\label{eq:betterQuenched2}
  \| \mathcal{S}^{0,2m}_\uo g \|_{H^{-s}} \leq D(\uo) L^{-2m \mathfrak{c}'} \| g \|_{H^1} \,. 
\end{align}
\end{remark}

\begin{proof}[Proof of Corollary \ref{cor:introNegRegMixQuenched}]
To retain the larger rate range $0<\mathfrak c'<\mathfrak c/2$ in the
introduction, use a Fourier truncation directly in $H^1$. Put
$M_m:=\lceil L^{m\mathfrak c}\rceil$. Cauchy--Schwarz on the modes
$0<|k|\leq M_m$ and the $L^2$ isometry on the remaining modes give
\[
  A_m:=\|\mathcal S^{0,2m}\|_{H^1_0\to H^{-s}}
  \leq\left(\sum_{0<|k|\leq M_m}
       \frac{\|\mathcal S^{0,2m}e_k\|_{H^{-s}}^2}{1+|k|^2}\right)^{1/2}
       +M_m^{-1}.
\]
By \eqref{eq:bootstrapL2Start2} and Jensen's inequality,
\[
  \E A_m\leq C(1+m\log L)^{1/2}L^{-(2m-1)\mathfrak c}
        +L^{-m\mathfrak c}.
\]
Consequently $D_0:=1+\sum_{m\geq1}L^{2m\mathfrak c'}A_m$ satisfies
$\E D_0\leq C_{\mathfrak c'}$ uniformly in $L\geq L_0$, since
$2\mathfrak c'<\mathfrak c$. The partial-cycle estimate
\eqref{eq:deterministFracTime2}, together with $L^2$ isometry for $t<2$,
proves the corollary with $D=CD_0L^{2s+2\mathfrak c'}$.
\end{proof}

\section{Batchelor's Law}\label{sec:batchLaw}
    %!TEX root = main.tex

We turn now to Batchelor's law for the driven passive scalar equation
\begin{align}\label{eq:PSA3}
  \partial_t \varphi_t + u \cdot \nabla \varphi_t = S \, , \quad \varphi_0 = 0 \,, 
\end{align}
where $S$ is a discrete-time forcing as in Setting~\ref{set:discDet}. 
% We focus mainly on the deterministic driving
% ().
%  deferring random / IID driving
% (Setting~\ref{set:discRand}) to {\color{red} later.}
Since it costs us little to do so, we will work in this section with the following sequential-type generalization of Setting \ref{set:discDet}: 

\begin{setting}\label{set:generalDiscDet}
  We have that 
  \begin{align}
    S(t, x) = \sum_{n = 1}^\infty \delta_{2n}(t) f_n \, , 
  \end{align}
  where $(f_n)_{n \geq 1}$ is a sequence of mean-free, $H^1$ functions satisfying 
  \begin{align}
    F^{-1} \leq \| f_n\|_{L^2} \leq \| f_n\|_{H^1} \leq F
  \end{align}
  for some fixed $F \geq 1$. 
\end{setting}

In Section \ref{subsec:detailedBatch3} we present a detailed version of the cumulative upper and lower bounds on $\| \Pi_{B_R} \varphi\|_{L^2}$ in Setting \ref{set:generalDiscDet}. 
We present in Section \ref{subsec:batchCont} a similar version in the continuous time setting (Setting \ref{set:contDet}). Section \ref{subsec:wrapup3} bridges between the statements of Section \ref{sec:batchLaw} and those in the introduction.

% Analogous sparsity statements addressing Fourier mass in the cone $\mathcal{C}_\theta$ are given in Section \ref{subsec:sparsityDisc3}. Finally, in Section \ref{subsec:wrapup3} we bridge these statements with those in Section \ref{sec:intro} ({\color{red}Theorems \ref{??}}), as well as briefly addressing changes necessary in the case of IID scalar driving (Setting \ref{set:discRand}). 

% {\color{red}  alteration to teh plan: will now be presenting batchelor's law here and sparsity section 4 } 

\subsection{Batchelor's law, discrete-time deterministic scalar driving}\label{subsec:detailedBatch3}

\begin{theorem}\label{thm:detailedBatch3}
  Let $q \geq 1$. There exist constants $C,L_0\geq1$, $c:=C^{-1}$,
  and $\zeta\geq1$, depending on $\psi,p,q,F$, with the following properties
  for all $L\geq L_0$.
  \begin{itemize}
    \item[(a)] There is a sequence of random variables $\mathfrak{U}_N : \Omega \to [1,\infty), N \geq 1$  with 
    \[\| \mathfrak{U}_N \|_{L^q(\Omega)} \leq C\]
    such that for any $R \geq L$ and for any $N \geq 1$, $t \in [2N, 2N+2)$, it holds that
    \begin{align}
      \| \Pi_{B_R} \varphi_{t} \|_{L^2}^2 \leq \mathfrak{U}_N(\uo) \log_L R \,. 
    \end{align}
    \item[(b)] There is a sequence of events $\mathcal{G}_N \subset \Omega$ satisfying 
    \[\P \left(\left(\mathcal{G}_N\right)^c\right) \leq C L^{-c}\] 
    such that for all $R \geq L^\zeta, N \geq C \log_L R, \uo \in \Gc_N$ and $t \in [2N, 2N+2)$, one has the lower bound 
    \[\| \Pi_{B_R} \varphi_{t}\|_{L^2}^2 \geq c \log_L R \,. \]
  \end{itemize}
\end{theorem}

\begin{remark} \ 
\begin{itemize}
  \item[(i)]  The events $\mathcal{G}_N$ and random variables $\mathfrak{U}_N$ depend on
    $N$, but satisfy estimates uniform in $N$.
    Indeed, for fixed $N$, the
    time-$2N$ scalar from \eqref{eq:PSA3} in
    Setting~\ref{set:generalDiscDet} satisfies
    \begin{align}
      \varphi_{2N} = \sum_{n=1}^{N}\mathcal{S}^{2n,2N}f_n \, .
    \end{align}
    The distribution of Fourier mass of $\varphi_{2N}$ therefore depends on
    the behavior of the Lagrangian flow on the \emph{terminal} time segments
    $[2n,2N]$, hence on the phase shifts $\omega_j^i$ for $n\leq j\leq N$,
    as reflected in the shifted mixing constants used below.
    \item[(ii)] We emphasize that the upper bound (Theorem \ref{thm:detailedBatch3} (a)) holds for all times $N$. The lower bound of part (b) only holds for $N \gtrsim \log_L R$, which is necessary to allow sufficient time to populate spatial scales from $1$ to $R^{-1}$.
    \item[(iii)] The upper bound requires only $R\geq L$, a natural starting point of the inertial range (since the velocity has amplitude of order $L$). For the lower bound, two successive shears can transport even fresh forcing to frequencies of order $L^2$ before the next piece of forcing is injected, leaving $o(1)$ expected Fourier mass in $B_L$ as $L\to\infty$; thus we need to strengthen the assumption on $R$. Rather than attempting to pin down the precise sharp assumption on $R$, we assume for the lower bound that $R \geq L^\zeta$ for some constant $\zeta \geq 1$.
\end{itemize}
  % This is why the random variables and events depend on the terminal index
  % $N$.
\end{remark}

\subsubsection*{Proof of Theorem \ref{thm:detailedBatch3}}

% \begin{proof}[Proof of Theorem \ref{thm:detailedBatch3}]

  Let $s$ and $\mathfrak c$ be as in Theorem~\ref{thm:neg-reg-flexible}. Define 
  \[
    \mathfrak c_*:=\frac{\mathfrak c}{24q},
    \qquad
    \alpha:=2\mathfrak c_* \,. 
  \]
  Apply Corollary \ref{cor:quenchedMix2} with $\mathfrak{c}' = \mathfrak{c}_*$ and under the replacement $q \mapsto 2q$, yielding $D \in L^{2q}(\Omega)$ with $\| D\|_{L^{2q}(\Omega)} \leq \bar C$ such that for mean-free $g \in H^1$, 
  \begin{equation}
    \|\Sc^{0,2m}_{\uo}g\|_{H^{-s}}
    \leq
    D(\uo)L^{-\alpha m}\|g\|_{H^1} 
  \end{equation}
  (see estimate~\eqref{eq:betterQuenched2}). 
  Here and throughout the rest of the proof, we reserve the notation
  \[\bar C = \bar C(\psi, p, q, F, s, \mathfrak{c}) \, , \quad \bar c = \bar C^{-1} \, , \]
  the precise value of which may change line-to-line, 
  while the constants $C, c = C^{-1}$ as in the statement of Theorem \ref{thm:detailedBatch3} will be specified at the end of the proof. We first prove both bounds for $R \geq L^\zeta$ and $L \geq L_0$, with $\zeta\geq1$ and $L_0$ to be specified, then extend the upper bound to $R\geq L$ by monotonicity.
  % $\bar C$ will denote a sufficiently large constant depending only on $\psi, p, q, F$ and $s, \mathfrak{c}$, while the constants
  
  % \begin{remark}

  %   {\color{blue}  blah blah avoid circularity } 
    
  % \end{remark}
  
  Define 
\begin{align}\label{eq:shiftedDconst3}    D_j(\uo):=D(\sigma^j\uo),\qquad j\geq0 \, , \end{align}
  where we recall that $\sigma : \Omega \circlearrowleft$ is the leftward shift on phase shift sequences $\uo = (\omega^j_i)$, i.e., $\sigma \uo = (\omega^j_{i + 1})$. We observe that 
  \begin{equation}\label{eq:detailedBatch3-shifted-mixing}
    \|\Sc^{2j,2r}_{\uo}g\|_{H^{-s}}
    \leq
    D_j(\uo)L^{-\alpha(r-j)}\|g\|_{H^1} \, , 
  \end{equation}
  and that since $\P$ is $\sigma$-invariant ($\P = \P \circ \sigma^{-1}$), it holds that 
  \begin{equation}\label{eq:unifMomentEst3}
    \sup_{j\geq0}\|D_j\|_{L^{2q}(\Omega)}\leq \bar C \, .
  \end{equation}
  In particular, for $w<r$,
  \begin{equation}\label{eq:detailedBatch3-correlation}
    \big|\langle \Sc^{2w,2r}f_w,f_r\rangle\big|
    \leq
    \bar C D_wL^{-\alpha(r-w)}.
  \end{equation}

  \paragraph{Preliminaries: Random multiplicative constants.}
  Define
  \begin{align}\label{eq:defnRandMultConsts3}
    W_N:=\sum_{j=1}^{N}D_jL^{-\mathfrak c_*(N-j)}, 
    \qquad 
    M_N:=\max_{1\leq k\leq N}
      \frac1k\sum_{j=N-k+1}^{N}D_j.
  \end{align}
  \begin{claim}\label{cla:boundWNMN3}
    It holds that 
    \begin{align}
      \sup_N \| W_N\|_{L^{2q}(\Omega)}, \quad \sup_N \| M_N\|_{L^{2q}(\Omega)} \leq \bar C \,. 
    \end{align}
  \end{claim}
  \begin{proof}[Proof of Claim]
    Minkowski's inequality gives
  \[
    \sup_{N\geq1}\|W_N\|_{L^{2q}(\Omega)}
    \leq
    \sup_{j\geq0}\|D_j\|_{L^{2q}(\Omega)}
    \sum_{m\geq0}L^{-\mathfrak c_*m}
    \leq \bar C \,. 
  \]
    For the uniform $L^{2q}$ estimate for $M_N$, we use an argument based on the following version of the maximal ergodic theorem, recalled below. 
    \begin{theorem}[{\cite[Theorem 6.3]{krengel2011ergodic}}]\label{thm:LpMaxErgThm} Let $(X, \mathcal{F}, \mu)$ be a probability space and let $T : X \to X$ be a measure-preserving transformation, i.e., $\mu \circ T^{-1} = \mu$. Let $f \in L^p(\mu), p > 1$, and let 
      \[f^* = \sup_{N \geq 1} \frac{1}{N} \sum_{i = 0}^{N-1} f \circ T^i \,. \]
      Then, 
      \[\| f^*\|_{L^p(X)} \leq \frac{p}{p-1} \| f\|_{L^p(X)}\]
      % Then, for any $\lambda > 0$ it holds that 
      % \[\mu \{ f^* > \lambda\} \leq \frac{1}{\lambda} \int_{f^* > \lambda} f d \mu \,. \]
    \end{theorem}
    We apply the maximal ergodic theorem to $X = \hat \Omega := \Omega_0^\Z$, the space of two-sided phase shift sequences $\hat \uo = (\omega_i^j)_{j = 1,2, i \in \Z}$. Let $\hat \sigma : \hat \Omega \circlearrowleft$ denote the leftward shift on $\hat \uo$'s, equipped with the natural probability $\hat \P = \P_0^\Z$. Viewing $D_j$'s and $M_N$'s as functions on $\hat \Omega$, we see that 
    \[M_N \circ \hat \sigma^{-N} = \max_{1 \leq k \leq N} \frac{1}{k} \sum_{j=0}^{k-1} D \circ \hat \sigma^{-j} \, . \]
    Applying Theorem \ref{thm:LpMaxErgThm} with $p = 2q > 1$ and $f = D, T = \hat \sigma^{-1}$, we see that 
    \begin{align*}
      \| M_N\|_{L^{2q}(\Omega)} & = \| M_N \circ \hat \sigma^{-N} \|_{L^{2q}(\hat \Omega)} \\ 
      & \leq \| f^*\|_{L^{2q}(\hat \Omega)} \leq \bar C \| D\|_{L^{2q}(\Omega)} \, , 
    \end{align*}
    hence $\| M_N\|_{L^{2q}(\Omega)} \leq \bar C$ as desired. 
  \end{proof}
  % To control $M_N$, pass to the natural two-sided extension of the product
  % phase space.  For fixed $N$, the averages in $M_N$ are backward ergodic
  % averages of $D\circ\sigma^N$.  The maximal ergodic inequality for the
  % inverse shift, used with exponent $2q>1$, therefore yields
  % \begin{equation}\label{eq:detailedBatch3-maximal}
  %   \sup_{N\geq1}\|M_N\|_{L^{2q}(\Omega)}\leq C.
  % \end{equation}
  % Notice that this also covers $q=1$, since the maximal inequality is being
  % used in $L^2$, not at its $L^1$ endpoint.

  \paragraph{Upper bound, $t = 2N$.}
  Let $N\geq1$ and $\ell_R:=\log_LR\geq1$. We will specify $K_0\geq1$ below and set
  \[
    K:=\lceil K_0\ell_R\rceil,
    \qquad
    A:=(N-K)_+=\max\{N-K,0\} \, .
  \]
  Thus $N-A=\min\{N,K\}\leq K\leq(K_0+1)\ell_R$.
  
  We decompose 
  \[
    \varphi_{2N}=\varphi^{\rm old}+\varphi^{\rm rec} 
  \]
  into `old' and `recent' contributions, 
  where
  \[
    \varphi^{\rm old}
    :=\sum_{j=1}^{A}\Sc^{2j,2N}f_j,
    \qquad
    \varphi^{\rm rec}
    :=\sum_{j=A+1}^{N}\Sc^{2j,2N}f_j.
  \]
  If $N\leq K$, then $A=0$ and the old contribution is empty.
  Since $N-j\geq K$ for $j\leq A$ and
  $\alpha=2\mathfrak c_*$,
  \[
    \sum_{j=1}^{A}D_jL^{-\alpha(N-j)}
    \leq
    L^{-\mathfrak c_*K}
    \sum_{j=1}^{A}D_jL^{-\mathfrak c_*(N-j)}
    \leq L^{-\mathfrak c_*K}W_N.
  \]
  Equation \eqref{eq:detailedBatch3-shifted-mixing}
  therefore gives
  \begin{align}
    \|\Pi_{B_R}\varphi^{\rm old}\|_{L^2}
    &\leq \bar C R^s\|\varphi^{\rm old}\|_{H^{-s}} \nonumber\\
    &\leq \bar C R^sL^{-\mathfrak c_*K}W_N \\ 
    \label{eq:detailedBatch3-upper-old}
    & \leq \bar C R^{s - \mathfrak{c}_* K_0} W_N
  \end{align}
  % {\color{green}Choose $K_0$ so that $\mathfrak{c}_* K_0 \geq s$; this bounds the RHS of \eqref{eq:detailedBatch3-upper-old} from above $\leq C W_N$, which we bound again from above $\leq C W_N \log_L R$ for $R \geq L^\zeta$, taking $\zeta \geq \frac{1}{10}$ and enlarging $L_0$. }

For the recent contribution $\varphi^{\rm rec}$ we use the lossy bound $\| \Pi_{B_R} \varphi^{\rm rec}\|_{L^2}^2 \leq \| \varphi^{\rm rec}\|_{L^2}^2$, which in turn is bounded
\[ \leq \sum_{w = A+1}^N \| \mathcal{S}^{2w,2N} f_w\|_{L^2}^2 + 2\sum_{A + 1 \leq w < r \leq N}\left| \langle \mathcal{S}^{2w, 2N} f_w , \mathcal{S}^{2r, 2N} f_r \rangle\right| \, . \]
The diagonal terms are bounded $\leq \bar C(N-A)\leq\bar C K$, while for the off-diagonal terms,
  \eqref{eq:detailedBatch3-correlation} implies
  \[
    \big|\langle \Sc^{2w,2N}f_w,\Sc^{2r,2N}f_r\rangle\big|
    =
    \big|\langle \Sc^{2w,2r}f_w,f_r\rangle\big|
    \leq \bar C D_wL^{-\alpha(r-w)}.
  \]
  Collecting, 
  \begin{align}
    \|\Pi_{B_R} \varphi^{\rm rec}\|_{L^2}^2
    &\leq
    \bar C(N-A)+\bar C\sum_{A+1\leq w<r\leq N}
      D_wL^{-\alpha(r-w)} \nonumber\\
    &\leq
    \bar C(N-A)+\bar CL^{-\alpha}\sum_{w=A+1}^{N}D_w \nonumber\\
    &\leq \bar C(N-A)(1+L^{-\alpha}M_N)
    \leq \bar C K (1 + L^{-\alpha}  M_N) \,.
    % \leq C (1 + M_N) \ell_R 
    \label{eq:detailedBatch3-upper-recent}
  \end{align}

  Combining the old and recent contributions gives
  \begin{align*}
    \|\Pi_{B_R}\varphi_{2N}\|_{L^2}^2 & 
    \leq 2 \left(\| \Pi_{B_R} \varphi^{\rm old}\|_{L^2}^2 + \| \Pi_{B_R} \varphi^{\rm rec}\|_{L^2}^2 \right) \\ 
    & \leq \bar C R^{2(s - \mathfrak{c}_* K_0)} W_N^2 + \bar C K_0 \ell_R (1 + L^{-\alpha} M_N)  \\ 
    & \leq \bar C \max\{ R^{2 (s - \mathfrak{c}_* K_0)} , K_0 \ell_R\} {(1 + W_N^2 + M_N)}
  \end{align*}
  At this point, we fix
  \[K_0 = \frac{1 + s}{\mathfrak{c}_*} \,  ,\]
  so our upper bound can be written in the form
  $\mathfrak{U}_N \ell_R$
  with 
  \[\mathfrak{U}_N = \bar C K_0 (1 + W_N^2 + M_N) \, .\]
  Note that \[\|{\mathfrak{U}}_N\|_{L^{q}(\Omega)} \leq \bar C\] by Claim \ref{cla:boundWNMN3}. 
  % where $\tilde{\mathfrak{U}}_N = \tilde{\mathfrak{U}}_N(\uo)$ will be enlarged to the desired function $\mathfrak{U}_N$ by a multiplicative constant to be specified  later in the proof. For now we note that \[\|\tilde{\mathfrak{U}}_N\|_{L^{q}(\Omega)} \leq \bar C\] by Claim \ref{cla:boundWNMN3}. 

  % \[\|\Pi_{B_R}\varphi_{2N}\|_{L^2}^2
  %   \leq \mathfrak U_N(\uo)\ell_R \]
  % where 
  %   \[
  %   \mathfrak U_N
  %   :=C\left(
  %     1+W_N^2+M_N
  %   \right) \,, 
  % \]
  % $C$ sufficiently large. That $\| \mathfrak{U}_N\|_{L^q}$ is bounded uniformly in $N$ follows from Claim \ref{cla:boundWNMN3}. 

  % After increasing $\zeta$, we have $\ell_R\geq1$ and
  % $K\leq C\ell_R$.  It follows from
  % \eqref{eq:detailedBatch3-upper-old}--\eqref{eq:detailedBatch3-upper-recent}
  % that
  % \[
  %   \|\Pi_{B_R}\varphi_{2N}\|_{L^2}^2
  %   \leq \mathfrak U_N(\uo)\ell_R,
  % \]
  % where the $R$-independent random variable
  % \[
  %   \mathfrak U_N
  %   :=C\left(
  %     1+W_N^2+M_N
  %   \right)
  % \]
  % satisfies
  % \[
  %   \sup_{N\geq1}\|\mathfrak U_N\|_{L^q(\Omega)}\leq C
  % \]
  % by \eqref{eq:detailedBatch3-maximal} and
  % the bound on $W_N$.  This proves the upper bound at even times.

  \paragraph{Lower bound.}
  Only for the lower bound, assume $N\geq C_N\ell_R$, where $C_N:=K_0+2$. Then $K\leq N$ and $A=N-K$. At the end, choose the theorem's constant $C\geq C_N$.
  In what follows, we will specify a sufficiently small constant
  $0 < k_0 \leq K_0$ and set
  \[
    k:=\lfloor k_0\ell_R\rfloor,
    \qquad
    B:=N-k \, . 
  \]
  % For now, we enforce $\zeta \geq 1$, so that moving forward $R \geq L$ and $\ell_R
  Later on, we will arrange so that $\ell_R$ is large, so that $k \geq 1$, but further details will wait till we specify $L_0$ and $\zeta$ further down the argument.
  % Increasing $\zeta$, we may assume
  % { \[
  %   1\leq k<K\leq N,
  %   \qquad
  %   k\simeq\ell_R,
  %   \qquad
  %   \frac{K}{k}\leq C.
  % \]  } 

  We write 
  \[
    \varphi_{2N}
    =\varphi^{\rm old}+\varphi^{\rm mid}+\varphi^{\rm new},
  \]
  where we have further decomposed $\varphi^{\rm rec}$ into 
  \[
    \varphi^{\rm mid}:=\sum_{j=A+1}^{B}\Sc^{2j,2N}f_j \, , 
    \qquad
    \varphi^{\rm new}:=\sum_{j=B+1}^{N}\Sc^{2j,2N}f_j \, .
  \]
  Expanding the square gives
  \begin{align}
    \| \Pi_{B_R} \varphi_{2N} \|_{L^2}^2 \geq &  \| \Pi_{B_R} \varphi^{\rm new}\|_{L^2}^2 \\  & - 2 \left|\langle \Pi_{B_R} \varphi^{\rm new}, \varphi^{\rm mid}\rangle \right| - 2 \left|\langle \Pi_{B_R} \varphi^{\rm old}, \varphi^{\rm new} + \varphi^{\rm mid} \rangle\right| \, , \label{eq:errorTermsDiscBatch3}
  \end{align}
  where we have discarded the nonnegative terms
  $\|\Pi_{B_R}\varphi^{\rm mid}\|_{L^2}^2$ and
  $\|\Pi_{B_R}\varphi^{\rm old}\|_{L^2}^2$ and bounded the cross terms by their absolute values.
  We will bound the three terms on the right-hand side separately on the event
  \[
    \mathcal G_N
    :=\{W_N+M_N\leq L^{\mathfrak{c}_*}\}.
  \]
  Markov's inequality, the preceding bound on $W_N$, and
  Claim \ref{cla:boundWNMN3} yield
  \begin{equation}\label{eq:detailedBatch3-good-probability}
    \P(\mathcal G_N^c)\leq \bar C L^{-2q\mathfrak{c}_*} = \bar C L^{- \mathfrak{c}/12} \,.
  \end{equation}

  We start with the pairing of $\varphi^{\rm new}$ with itself; with $g_j:=\Sc^{2j,2N}f_j$ this expands as
  \[\| \Pi_{B_R} \varphi^{\rm new}\|_{L^2}^2 = \sum_{j = B + 1}^N \| \Pi_{B_R} g_j\|_{L^2}^2 + 2 \sum_{B + 1 \leq w < r \leq N} \langle g_w, \Pi_{B_R} g_r\rangle \, .\]
  Increase $L_0$ so that the $H^1$ operator norm of each full or partial
  shear is at most $C_\psi L\leq L^2$. Thus the $2(N-j)$ shears
  between times $2j$ and $2N$ give, for $B<j\leq N$,
  \begin{align*}
    \|\Pi_{\Z^2\setminus B_R}g_j\|_{L^2}
    & \leq \bar C R^{-1}\| g_j\|_{H^1} \\ 
    &\leq \bar C R^{-1}L^{4(N-j)} \\ 
    & \leq \bar C R^{-1}L^{4k} 
    \leq \bar C R^{4 k_0 - 1} \, , 
  \end{align*}
  hence 
  \begin{align*}
    \|\Pi_{B_R}g_j\|_{L^2}
    \geq \bar{c} - \bar C R^{4 k_0 - 1} \, . 
  \end{align*}
  % Shrinking $k_0$, it follows that $\| \Pi_{B_R} g_j\|_{L^2}^2 \geq c$ for each $B < j \leq N$. 
  %  There is then a constant
  % $\delta_{\rm tail}>0$ such that, on setting
  % $\delta_0:=\min\{\delta_{\rm old},\delta_{\rm tail}\}$,
  % \begin{equation}\label{eq:detailedBatch3-tail-choice}
  %   R^{-1}L^{Ck}K\leq L^{-\delta_0}.
  % \end{equation}
  % In particular, each diagonal term
  % $\|\Pi_{B_R}g_j\|_{L^2}^2$ is bounded below by a fixed positive constant.
  For the off-diagonal pieces, 
  \begin{align}\label{eq:offDiagonal3}
    \begin{aligned}
      |\langle g_w,\Pi_{B_R}g_r\rangle|
      &\leq
      |\langle\Sc^{2w,2r}f_w,f_r\rangle|
      +\bar C R^{4 k_0-1}\\
      &\leq
      \bar C D_wL^{-\alpha(r-w)}+\bar C R^{4 k_0-1}.
    \end{aligned}
  \end{align}
  On $\mathcal G_N$, geometric summation and the definition of $M_N$ give
  \begin{align*}
    \sum_{B<w<r\leq N}|\langle\Pi_{B_R}g_w,g_r\rangle|
    % & \leq \sum_{B<w<r\leq N} |\langle g_w,g_r\rangle| + \| \Pi_{B_R} g_w\|_{L^2} \| g_r\|_{L^2} \\ 
    &\leq
    \bar C L^{-\alpha}\left(\sum_{w=B+1}^{N}D_w\right)
      +\bar Ck^2R^{4 k_0-1}\\
    &\leq
    \bar C k \left(L^{-\mathfrak{c}_*}  + k R^{4 k_0-1}\right)  \,. 
    % L^{- (\alpha}}
    % C\left(L^{\beta-\alpha}+L^{-\delta_0}\right)k.
  \end{align*}
  % {\color{green} Enforcing $\zeta \geq 2 + \frac{1}{10}$, $L_0$ large and $k_0$ small, the above is $\ll \log_L R$, hence $\|\Pi_{B_R}\varphi^{\rm new}\|_{L^2}^2 \geq c \log_L R$ along $\mathcal{G}_N$, as desired.  } 
  In total, 
  \begin{align*}
    \| \Pi_{B_R} \varphi^{\rm new}\|_{L^2}^2 \geq k (\bar c - \bar C k R^{4k_0 -1} - \bar C L^{- \mathfrak{c}_*}) \, . 
  \end{align*}
  % \begin{equation}\label{eq:detailedBatch3-new-lower}
  %   \|\Pi_{B_R}\varphi^{\rm new}\|_{L^2}^2
  %   \geq
  %   \left(c-CL^{\beta-\alpha}-CL^{-\delta_0}\right)k.
  % \end{equation}

  We address now the middle and old pieces (RHS of equation \eqref{eq:errorTermsDiscBatch3}). Starting with the pairing of $\varphi^{\rm mid}$ and $\varphi^{\rm new}$, for
  $A<w\leq B<r\leq N$, equation \eqref{eq:offDiagonal3} gives 
  \[
    |\langle\Pi_{B_R}g_r,g_w\rangle|
    \leq
    \bar C D_wL^{-\alpha(r-w)}+\bar C  R^{4k_0 -1} \, .
  \]
  Hence, on $\mathcal G_N$,
  \begin{align*}
    |\langle\Pi_{B_R}\varphi^{\rm new},\varphi^{\rm mid}\rangle|
    &\leq
    \bar C L^{-\alpha}\sum_{w=A+1}^{B}D_w
      +\bar CkK R^{4 k_0 -1}\\
    &\leq
    \bar CK L^{-\mathfrak{c}_*}+ \bar CK k R^{4k_0 -1}  \\ 
    & \leq \bar C \ell_R \big(L^{- \mathfrak{c}_*} + k R^{4k_0 -1}\big)
  \end{align*}
  % This contribution is $\ll \log_L R$ on taking $\zeta$ sufficiently large. 
  % Here we used
  % $\sum_{w=A+1}^{B}D_w\leq KM_N$ and $K/k\leq C$.

  For the pairing of $\varphi^{\rm old}$ with $\varphi^{\rm rec} = \varphi^{\rm mid} + \varphi^{\rm new}$, 
  Equation \eqref{eq:detailedBatch3-upper-old} along $\mathcal{G}_N$ gives
  \begin{equation}\label{eq:detailedBatch3-old-good}
    \|\Pi_{B_R}\varphi^{\rm old}\|_{L^2}
    \leq \bar C R^{s - \mathfrak{c}_* K_0}L^{\mathfrak{c}_* } \leq \bar C R^{-1} L^{\mathfrak{c}_*}
  \end{equation}
  while for $\varphi^{\rm rec} = \varphi^{\rm mid} + \varphi^{\rm new}$ we apply the crude bound
  \[
    \|\varphi^{\rm rec}\|_{L^2}\leq \bar CK \leq \bar C \ell_R
  \]
  from our choice of $K_0$. 
  Combining, 
  \[
    |\langle\Pi_{B_R}\varphi^{\rm old},
      \varphi^{\rm mid}+\varphi^{\rm new}\rangle|
    \leq \bar C  R^{-1} L^{\mathfrak{c}_*} \ell_R \,. 
  \]
  % When $K_0$ is large enough, the above is $\ll \log_L R$ for $R \geq L^\zeta$, $\zeta \geq 2 + \frac{1}{10}$. 
  
  Piecing back together the components in our lower bound of $\| \Pi_{B_R} \varphi_{2N}\|_{L^2}^2$ gives, along $\mathcal{G}_N$, 
  \begin{align*}
    \| \Pi_{B_R} \varphi_{2N} \|_{L^2}^2 & \geq k (\bar c - \bar C k R^{4k_0 -1} - \bar C L^{- \mathfrak{c}_*}) \\ 
    & - \bar C \ell_R\big( R^{-1} L^{\mathfrak{c}_*} +  L^{-\mathfrak{c}_*} +  k R^{4k_0 -1} \big) \, . 
  \end{align*}
  Fix $k_0=1/8$ and $\zeta\geq16$, so $4k_0-1=-1/2$ and
  $k\geq\ell_R/16$. Keep the product $\ell_RR^{-1/2}$ intact: for $L\geq2$,
  \[
    \sup_{R\geq L^\zeta}(\log_LR)R^{-1/2}
    =\zeta L^{-\zeta/2}.
  \]
  Consequently, on $\mathcal G_N$,
  \[
    \|\Pi_{B_R}\varphi_{2N}\|_{L^2}^2
    \geq \ell_R\bigl(\bar c-\bar C\zeta L^{-\zeta/2}
      -\bar C L^{-\mathfrak c_*}-\bar C L^{\mathfrak c_*-\zeta}\bigr).
  \]
  Taking $L_0$ sufficiently large establishes the lower bound uniformly in $R$.
  
  % completes the desired lower bound along $\mathcal{G}_{N}$. 
  
  % It follows that
  % \begin{align*}
  %   \|\Pi_{B_R}\varphi_{2N}\|_{L^2}^2
  %   &\geq
  %   \|\Pi_{B_R}\varphi^{\rm new}\|_{L^2}^2
  %   -2|\langle\Pi_{B_R}\varphi^{\rm new},\varphi^{\rm mid}\rangle|\\
  %   &\quad
  %   -2|\langle\Pi_{B_R}\varphi^{\rm old},
  %     \varphi^{\rm mid}+\varphi^{\rm new}\rangle|\\
  %   &\geq
  %   \left(c-CL^{\beta-\alpha}-CL^{-\delta_0}\right)k.
  % \end{align*}
  % Since $\beta<\alpha$ and $k\simeq\ell_R$, increasing $L_0$ proves
  % \[
  %   \|\Pi_{B_R}\varphi_{2N}\|_{L^2}^2\geq c\log_LR
  %   \qquad\text{on }\mathcal G_N.
  % \]

% {\color{red}  QUIT HERE } 

  \paragraph{Passage to $t\in[2N,2N+2)$.}

  Recall that our right-limit convention for $\varphi_t$ implies that no additional scalar driving occurs for $t \in [2N,2N+2)$. For any $N\geq1$, use the upper-bound decomposition to write
  \[
    \varphi_t
    =\underbrace{\mathcal S^{2N,t}\varphi^{\rm old}}_{=:\varphi^{\rm old}_t}
    +\underbrace{\mathcal S^{2N,t}\varphi^{\rm rec}}_{=:\varphi^{\rm rec}_t}.
  \]

  \subparagraph{Upper bound.} 
  The upper bound on the old contribution changes to \[\|\Pi_{B_R}\varphi^{\rm old}_t\|_{L^2}
    \leq
    \bar C R^sL^{2s}\|\varphi^{\rm old}\|_{H^{-s}}
    \leq \bar C L^{2s}R^{-1}W_N
    \leq \bar C W_N \, ,
  \]
  where we take $\zeta\geq\max\{1,2s\}$. The upper bound on the recent contribution does not change at all, since for this term we merely bound
    \[\| \Pi_{B_R} \varphi^{\rm rec}_t\|_{L^2} \leq \| \varphi^{\rm rec}_t\|_{L^2} = \| \varphi^{\rm rec}\|_{L^2}\]
    using that $\mathcal{S}^{2N, t}$ is an isometry. Thus the upper bound holds for every $N\geq1$, uniformly over $t\in[2N,2N+2)$, after enlarging the constant in $\mathfrak U_N$ if necessary.

    To extend the upper bound to $L\leq R\leq L^\zeta$, use
    \[
      \|\Pi_{B_R}\varphi_t\|_{L^2}^2
      \leq \|\Pi_{B_{L^\zeta}}\varphi_t\|_{L^2}^2
      \leq \zeta\mathfrak U_N
      \leq \zeta\mathfrak U_N\log_LR.
    \]
    Replacing $\mathfrak U_N$ by $\zeta\mathfrak U_N$ proves part (a) for all $R\geq L$, with the same uniform moment control.
    % \item $\| \pi_{B_R} \varphi^{\rm old}_t\|_{L^2}$ changes.  Have $ \|\Pi_{B_R}\varphi^{\rm old}_t\|_{L^2}
    % \leq
    % C R^sL^{2s}\|\varphi^{\rm old}\|_{H^{-s}}.$ absorb multiplicative factors by taking $K_0$ larger. 

    \subparagraph{Lower bound.}
    Return to $R\geq L^\zeta$ and $N\geq C_N\ell_R$, and write
    $\varphi^{\rm rec}_t=\varphi^{\rm mid}_t+\varphi^{\rm new}_t$, where
    $\varphi^{\rm mid}_t=\mathcal S^{2N,t}\varphi^{\rm mid}$ and
    $\varphi^{\rm new}_t=\mathcal S^{2N,t}\varphi^{\rm new}$.
    % $\| \Pi_{B_R} \varphi^{\rm rec}\|_{L^2}$ does not change at all, since we just bound $\leq \| \varphi^{\rm rec}_t\|_{L^2} = \| \varphi^{\rm rec}\|_{L^2}$. 

    Set $g_j(t)=\mathcal S^{2j,t}f_j$. For $B<j\leq N$, we have
    $N-j\leq k-1$, and the final evolution adds at most two shears. Hence
    \begin{align}\label{eq:smallScaleContTime3}
      \|\Pi_{\Z^2\setminus B_R}g_j(t)\|_{L^2}
      \leq\bar C R^{-1}L^{4k}\leq\bar C R^{-1/2}.
    \end{align}
    The diagonal, off-diagonal, and middle/new estimates therefore hold
    with the same bounds as above. The old/recent interaction is at most
    $\bar C\ell_R L^{2s+\mathfrak c_*}R^{-1}$, which is uniformly small
    relative to $\ell_R$ after taking $\zeta>2s+\mathfrak c_*$ and $L_0$ large.
    Thus the same random variables $\mathfrak U_N$ and events $\mathcal G_N$
    work uniformly over $t\in[2N,2N+2)$.

\subsection{Batchelor's law, continuous-time deterministic scalar driving}\label{subsec:batchCont}

We now turn to scalar driving that is constant in time (Setting \ref{set:contDet}):
\[
  S(t,x) \equiv f(x) \qquad \text{ for all } \quad t \geq 0 \,. 
\]
The velocity field $u = u_\uo$ and phase sequences $\uo \in \Omega$ are the same as in the previous
sections. The source $f$ is real-valued and mean-zero. By the assumptions in
Setting \ref{set:contDet}, we may fix $F\geq 1$ such that
\begin{equation}\label{eq:cont-force-assumptions3}
  \|f\|_{H^1}+\|f\|_{L^\infty}\leq F,
  \qquad
  \left\|\int_\T f(x^1,x^2)\dif x^1\right\|_{L^2_{x^2}} + \left\|\int_\T f(x^1,x^2)\dif x^2\right\|_{L^2_{x^1}}\geq F^{-1}.
\end{equation}
The scalar solves
\begin{equation}\label{eq:cont-scalar3}
  \partial_t\varphi+u_\uo\cdot\nabla\varphi=f,
  \qquad
  \varphi_0=0.
\end{equation}
By linearity and incompressibility,
\begin{equation}\label{eq:cont-duhamel3}
  \varphi_T
  =
  \int_0^T\Sc^{r,T}_{\uo}f\,\dif r.
\end{equation}

\begin{theorem}\label{thm:contBatchDetailed3}
  Let $q \geq 1$. Then, there exist constants $C, L_0 \geq 1, c = C^{-1}$ and $\zeta > 0$, depending only on $\psi, p, q, F$, with the following properties for all $L \geq L_0$. 

  \begin{itemize}
    \item[(a)] There is a sequence of random variables
    $\mathfrak{U}^{\rm cont}_N:\Omega\to[1,\infty)$ satisfying
    \[
      \sup_{N\geq1}
      \|\mathfrak{U}^{\rm cont}_N\|_{L^q(\Omega)}
      \leq C
    \]
    such that for any $R \geq L^\zeta$ and
    $N \geq C \log_L R$, $t \in [2N, 2N + 2)$, one has
    \begin{align}
      \| \Pi_{B_R} \varphi_t \|_{L^2}^2
      \leq \mathfrak{U}^{\rm cont}_N(\uo) \log_L R \,.
    \end{align}
    \item[(b)] There is a sequence of events
    $\mathcal{G}^{\rm cont}_N\subset\Omega$ satisfying
    \[
      \P\big((\mathcal{G}^{\rm cont}_N)^c\big)
      \leq C L^{-c}
    \]
    such that, with $N,R$ and $t$ as in item (a), and for
    $\uo \in \mathcal{G}^{\rm cont}_N$, one has the lower bound
    \begin{align}
      \| \Pi_{B_R} \varphi_t \|_{L^2}^2 \geq c \log_L R \,. 
    \end{align} 
  \end{itemize}
\end{theorem}

With one notable exception, the proof of Theorem \ref{thm:contBatchDetailed3} largely mirrors that of the discrete-time version Theorem \ref{thm:detailedBatch3}. For brevity, in the coming proof, we emphasize the changes and leave further details to the reader. 

% {\color{red}  since many of the details are from the proof of theorem \ref{thm:detailedBatch3} we just give a summary of what changes below.  } 

\paragraph{Setup.} Let $s,\mathfrak c$ be as in
Theorem~\ref{thm:neg-reg-flexible}.  As in the proof of
Theorem~\ref{thm:detailedBatch3}, apply
Corollary~\ref{cor:quenchedMix2} with moment order $2q$ and denote the
resulting shifted constants by $D_j := D \circ \sigma^j$. Let $\mathfrak{c}_* = \mathfrak{c} / 24q, \alpha = 2 \mathfrak{c}_*$ be as in the setup in the proof of Theorem \ref{thm:detailedBatch3}, and note that equation \eqref{eq:detailedBatch3-shifted-mixing} holds in our setting. We continue the practice of writing $\bar C \geq 1, \bar c = \bar C^{-1}$ for constants depending only on $\psi, p, q, F, s, \mathfrak{c}$, while leaving specification of $C, c^{-1}$ as in the statement of Theorem \ref{thm:contBatchDetailed3} to the end.

% Thus, after possibly decreasing the
% decay exponent during the proof,
% \[
%   \|\Sc^{2j,2r}h\|_{H^{-s}}
%   \leq
%   D_jL^{-\alpha_0(r-j)}\|h\|_{H^1},
%   \qquad j<r,
% \]
% where $\sup_j\|D_j\|_{L^{2q}(\Omega)}\leq\bar C$.  We continue to use the
% constant convention $\bar C,\bar c=\bar C^{-1}$ from the discrete proof.

\subsubsection*{Preliminary Lemma}

The main point is to reduce the continuous source to a sequence of `effective' discrete-time sources. For $j \geq 1$ define 
\begin{align}
  \mathfrak{f}_j = \int_{2(j-1)}^{2j} \mathcal{S}^{\tau, 2j} f \dif \tau \,. 
\end{align}
Then, Duhamel's formula implies the exact identity
\begin{align}\label{eq:contTimeEvens3}
  \varphi_{2N} = \sum_{j = 1}^{N} \mathcal{S}^{2j,2N} \mathfrak{f}_j \, , 
\end{align}
reminiscent of the setup for Theorem \ref{thm:detailedBatch3} (Setting \ref{set:generalDiscDet}). The following checks that the hypotheses of our setting imply that the $(\mathfrak{f}_j)$ are more-or-less compatible with the proof presented there, i.e., (i) the $\mathfrak{f}_j$ do not `collapse' to zero and (ii) are controlled in $H^1$. 

Below, for $j \geq 1$ we decompose 
\begin{equation}\label{eq:cont-effective-source3}
  \mathfrak f_j
  :=
  \Sc^{2j-1,2j}u_{j-1}^{\rm v}+u_{j-1}^{\rm h},
\end{equation}
where 
  \begin{align}
    u_j^{\rm v}
    :=
    \int_{2j}^{2j+1}\Sc^{r,2j+1}f\,\dif r,
    \qquad 
    u_j^{\rm h}
    :=
    \int_{2j+1}^{2j+2}\Sc^{r,2j+2}f\,\dif r \,. 
    \label{eq:cont-unit-sources3}
    % \mathfrak f_j
    % &:=
    % \Sc^{2j-1,2j}u_{j-1}^{\rm v}+u_{j-1}^{\rm h},
    % && j\geq1.
    % \label{eq:cont-effective-source3}
  \end{align}

% We reduce the continuous source to a sequence of effective sources delivered
% at even times.  The only point not already contained in the proof of
% Theorem~\ref{thm:detailedBatch3} is the following lemma.  Besides giving a
% uniform lower bound for an effective source, the lemma rules out cancellation
% between sources accumulated during neighboring unit intervals.

\begin{lemma}[Effective sources for constant-in-time driving]
\label{lem:cont-effective-sources3}
  There are constants $\bar C,\bar c,\alpha_0>0$, and a sequence of random variables
  $\widetilde D_j:\Omega\to[1,\infty)$ such that
  \[
    \sup_{j\geq1}\|\widetilde D_j\|_{L^{2q}(\Omega)}
    \leq \bar C,
  \]
  with the following properties.  
  \begin{itemize}
    \item[(a)] For every $j \geq 1$, 
      \begin{equation}\label{eq:cont-effective-source-size3}
    \bar c
    \leq \|\mathfrak f_j\|_{L^2}
    \leq \bar C,
    \qquad
    \|\mathfrak f_j\|_{H^1}\leq \bar C L^2.
  \end{equation}
    \item[(b)] For integers $1\leq w<r$,
  \begin{equation}\label{eq:cont-effective-correlation3}
    \left|
      \left\langle
        \Sc^{2w,2r}\mathfrak f_w,\mathfrak f_r
      \right\rangle
    \right|
    \leq
    \bar C\widetilde D_wL^{-\alpha_0(r-w)} \, . 
  \end{equation}
  \end{itemize}

  % Then
  % \begin{equation}\label{eq:cont-effective-source-size3}
  %   \bar c
  %   \leq \|\mathfrak f_j\|_{L^2}
  %   \leq \bar C,
  %   \qquad
  %   \|\mathfrak f_j\|_{H^1}\leq \bar C L^2.
  % \end{equation}

\end{lemma}

\begin{remark}
  We note that in part (a), the main difference is that the effective time-$2j$ scalar driving $\mathfrak{f}_j$ can be large in $H^1$. The purpose of part (b) is to ensure that this does not interfere too badly with the correlations between different length-2 periods. 
\end{remark}

% {\color{blue}  part (a) estimate: main difference now is that the effective time-$2j$ force can be large in $H^1$. in part (b) we check that this situation does not jeopardize cross correlations between different epochs. } 

\begin{proof}[Proof of Lemma \ref{lem:cont-effective-sources3}]
  For part (a), it suffices to write our estimates for $\mathfrak{f} = \mathfrak{f}_1, u^{\rm v} = u^{\rm v}_0, u^{\rm h} = u^{\rm h}_0$ ($j = 1$). 
  % Write $(\omega^1, \omega^2)$ for the phase shifts in the time interval $[0,2)$. 
  The upper $L^2$ and $H^1$ bounds are evident; the bulk of the work is in checking the $L^2$ lower bound. To this end, 
  introduce the coordinate averages
  \[
    \bar f^{\rm v}(x^1)
    :=\int_\T f(x^1,x^2)\,\dif x^2,
    \qquad
    \bar f^{\rm h}(x^2)
    :=\int_\T f(x^1,x^2)\,\dif x^1 \, .
    % \qquad
    % g:=\bar f^{\rm v}+\bar f^{\rm h}.
  \]
  Since $f$ is mean-free, $\bar f^{\rm v}$ and $\bar f^{\rm h}$ are
  orthogonal in $L^2(\T^2)$, as are $\mathcal{S}^{1,2}\bar f^{\rm v}$ and $\bar f^{\rm h}$ (noting $\bar f^{\rm h} \circ \Phi^{1,2} = \bar f^{\rm h}$). 
  Set
  \[g := \mathcal{S}^{1,2} \bar f^{\rm v} + \bar f^{\rm h} \, .\]
  By orthogonality, the $L^2$ isometry, and \eqref{eq:cont-force-assumptions3},
  \begin{equation}\label{eq:cont-coordinate-average-lower3}
    \begin{aligned}
      \|g\|_{L^2}^2
      &=\|\bar f^{\rm v}\|_{L^2}^2+\|\bar f^{\rm h}\|_{L^2}^2 \\
      &\geq \frac12\left(\|\bar f^{\rm v}\|_{L^2}+\|\bar f^{\rm h}\|_{L^2}\right)^2
      \geq \frac1{2F^2}.
    \end{aligned}
  \end{equation}
  Our lower bound on $\mathfrak{f}$ in $L^2$ will follow on checking that $\mathfrak{f}$ and $g$ are close in $L^2$, which in turn is implied by estimating
  $\| u^{\rm v} - \bar f^{\rm v}\|_{L^2}$ and $\| u^{\rm h} - \bar f^{\rm h}\|_{L^2}$. 
  In the coming proof, we focus on the `vertical' estimate $\| u^{\rm v} - \bar f^{\rm v}\|_{L^2}$; the horizontal estimate follows similarly and is omitted.

  % We first record the effect of averaging the source over one shear
  % interval.  On a vertical interval, write
  Write
  \[
    f(x^1,x^2)
    =
    \bar f^{\rm v}(x^1)
    +
    \sum_{\ell\neq0}
      f_\ell(x^1)e^{2\pi i\ell x^2}.
  \]
  Write $\omega = \omega^1_1$ for the phase shift on the interval $[0,1)$ and set
  \[
    A(z):=\int_0^1e^{-2\pi i\tau z}\,\dif\tau \,. 
  \]
  Up to an immaterial choice of sign,
  \[
    u^{\rm v}-\bar f^{\rm v}
    =
    \sum_{\ell\neq0}
      f_\ell(x^1)e^{2\pi i\ell x^2}
      A\big(\ell L\psi(x^1+\omega)\big).
  \]
  We have $|A(z)|\leq \min\{1,C|z|^{-1}\}$.  Assumption (ND) implies the
  uniform sublevel estimate
  \[
    \big|\{x\in\T:|\psi(x+\omega)|\leq\eta\}\big|
    \leq C_\psi\eta,
    \qquad 0<\eta<1.
  \]
  Splitting into this sublevel set and its complement, and using
  $H^1(\T)\hookrightarrow L^\infty(\T)$, gives
  \[
    \int_\T |f_\ell(x^1)|^2
      \left|A\big(\ell L\psi(x^1+\omega)\big)\right|^2\,\dif x^1
    \leq
    C_\psi
    \left(
      \eta+\frac1{|\ell|^2L^2\eta^2}
    \right)
    \|f_\ell\|_{H^1(\T)}^2.
  \]
  Taking $\eta=(|\ell|L)^{-2/3}$ and summing in $\ell\neq0$ yields
  \begin{equation}\label{eq:cont-one-shear-average3}
    \|u^{\rm v}-\bar f^{\rm v}\|_{L^2}
    % +
    % \|u_j^{\rm h}-\bar f^{\rm h}\|_{L^2}
    \leq \bar C L^{-1/3}.
  \end{equation}
  The horizontal estimate follows by exchanging the two coordinates.
  Since $\mathfrak f=\mathcal S^{1,2}u^{\rm v}+u^{\rm h}$, the $L^2$
  isometry gives
  \[
    \|\mathfrak f-g\|_{L^2}
    \leq \|u^{\rm v}-\bar f^{\rm v}\|_{L^2}
       +\|u^{\rm h}-\bar f^{\rm h}\|_{L^2}
    \leq \bar C L^{-1/3}.
  \]
  Increase $L_0$ so that $\bar C L_0^{-1/3}\leq(2\sqrt2 F)^{-1}$.
  Then \eqref{eq:cont-coordinate-average-lower3} implies
  \[
    \|\mathfrak f\|_{L^2}
    \geq \|g\|_{L^2}-\|\mathfrak f-g\|_{L^2}
    \geq \frac1{2\sqrt2 F}.
  \]
  This bound is uniform in the phase shifts. Applying the same argument
  after a time shift proves part (a) for every $j\geq1$.

  % {\color{red}  QUIT HERE }
  For part (b), let
  \[
    \mathcal H_j:=\Sc^{2j+1,2j+2},
    \qquad
    \widehat g:=\bar f^{\rm v}+\bar f^{\rm h},
  \]
  noting $\mathcal{H}_j$ corresponds to the \emph{horizontal} shear, fixing $x^2$. 
  Define the approximant to $\mathfrak f_j$ by
  \[
    g_j
    :=
    \Sc^{2j-1,2j}\bar f^{\rm v}+\bar f^{\rm h}
    =
    \mathcal H_{j-1}\widehat g.
  \]
  % Here the last equality uses the fact that $\mathcal H_{j-1}$ fixes
  % functions of $x^2$.  
  The estimates from part (a), applied after a time
  shift, give
  \begin{equation}\label{eq:cont-effective-source-approx3}
    \| \mathfrak f_j - g_j\|_{L^2} \leq\bar C L^{-1/3},
    \qquad j\geq1.
  \end{equation}
  
  The proof of
  Corollary~\ref{cor:quenchedMix2}, after exchanging $x^1$ and $x^2$, applies
  equally to the odd-to-odd evolution, yielding 
  a sequence $D_j^{\rm o}$ of random variables with uniformly bounded $L^{2q}$ norm such that
  \begin{equation}\label{eq:cont-odd-mixing3}
    \|\Sc^{2j+1,2r+1}h\|_{H^{-s}}
    \leq
    D_j^{\rm o}L^{-\alpha(r-j)}\|h\|_{H^1},
    \qquad j<r \,. 
  \end{equation}
  In a moment, we will apply this to the approximants $g_w$ and $g_r$, using that the common
  final horizontal shear cancels:
  \begin{equation}\label{eq:cont-leading-correlation3}
    \left\langle
      \Sc^{2w,2r}g_w,g_r
    \right\rangle
    =
    \left\langle
      \Sc^{2w-1,2r-1}
      \widehat g,\widehat g
    \right\rangle \, . 
  \end{equation}
  
  With these tools in place, let $m_0$ be an integer, to be specified shortly. 
  For $r-w\leq m_0$, equations
  \eqref{eq:cont-effective-source-approx3}--\eqref{eq:cont-leading-correlation3},
  the $L^2$ isometry, and \eqref{eq:cont-odd-mixing3} give
  \[
    \left|
      \left\langle
        \Sc^{2w,2r}\mathfrak f_w,\mathfrak f_r
      \right\rangle
    \right|
    \leq
    \bar C D_{w-1}^{\rm o}L^{-\alpha(r-w)}
    +\bar C L^{-1/3}.
  \]
  For $r-w>m_0$, we instead apply
  \eqref{eq:detailedBatch3-shifted-mixing} directly to the exact effective
  sources.  By \eqref{eq:cont-effective-source-size3},
  \[
    \left|
      \left\langle
        \Sc^{2w,2r}\mathfrak f_w,\mathfrak f_r
      \right\rangle
    \right|
    \leq
    \bar C D_wL^{4-\alpha(r-w)}.
  \]
  Choose $m_0\geq 8/\alpha$ and set
  \[
    \alpha_0:=\min\left\{\frac{\alpha}{2},\frac{1}{3m_0}\right\}.
  \]
  If $r-w\leq m_0$, then
  $L^{-1/3}\leq L^{-\alpha_0(r-w)}$; if $r-w>m_0$, then
  $L^{4-\alpha(r-w)}\leq L^{-\alpha_0(r-w)}$.  Thus, after enlarging
  $1+D_w+D_{w-1}^{\rm o}$ to a random variable
  $\widetilde D_w$ with uniformly bounded $L^{2q}$ norm, both estimates
  give \eqref{eq:cont-effective-correlation3}.
\end{proof}

\subsubsection*{Proof sketch of Theorem \ref{thm:contBatchDetailed3}}

% {\color{red}  actually QUIT HERE } 

% {\color{red}  real difference in upper and lower bounds at times $t = 2N$: effective discrete scalar driving $\mathfrak{f}_j$ is bounded in $H^1$ by $L^2$. cope by increasing $K_0$ } 

% We now explain how Lemma~\ref{lem:cont-effective-sources3} reduces the
% remaining proof to ideas already presented in the proof of Theorem~\ref{thm:detailedBatch3}. 

% % At even times,

With $\alpha_0$ and $\widetilde{D}_j$ as in Lemma \ref{lem:cont-effective-sources3}, set $\mathfrak{c}_* = \alpha_0 / 2$ and
% \eqref{eq:cont-duhamel3} and the composition law give the exact identity
% \begin{equation}\label{eq:cont-discrete-reduction3}
%   \varphi_{2N}
%   = 
%   \sum_{j=1}^{N}\Sc^{2j,2N}\mathfrak f_j.
% \end{equation}
% After decreasing $\alpha_0$ once more if necessary, set
% $\alpha:=\alpha_0=2\mathfrak c_*$, and define
\[
  \widetilde W_N
  :=
  \sum_{j=1}^{N}\widetilde D_j
    L^{-\mathfrak c_*(N-j)} \, , 
  \qquad
  \widetilde M_N
  :=
  \max_{1\leq k\leq N}
  \frac1k\sum_{j=N-k+1}^{N}\widetilde D_j \, . 
\]
The proof of Claim~\ref{cla:boundWNMN3} applies verbatim and gives
\begin{equation}\label{eq:cont-WM-moments3}
  \sup_{N\geq1}
  \left(
    \|\widetilde W_N\|_{L^{2q}(\Omega)}
    +
    \|\widetilde M_N\|_{L^{2q}(\Omega)}
  \right)
  \leq \bar C.
\end{equation}

% \eqref{eq:contTimeEvens3}

% {\color{red}  QUIT HERE } 

\paragraph{Upper bound at $t=2N$.}
Use the old/recent decomposition from the proof of
Theorem~\ref{thm:detailedBatch3}, with
$K=\lceil K_0\log_LR\rceil$ and $A=(N-K)_+$, applied to
\eqref{eq:contTimeEvens3}.  The old-piece estimate acquires only
the fixed factor $L^2$ from
\eqref{eq:cont-effective-source-size3}; it is absorbed by increasing $K_0$.
For the recent part, \eqref{eq:cont-effective-correlation3} gives
\begin{align}\label{eq:recContribContinuous3}
  \|\varphi^{\rm rec}\|_{L^2}^2
  \leq
  \bar C K
  +
  \bar C L^{-\alpha_0}
    \sum_{j=A+1}^{N}\widetilde D_j
  \leq
  \bar C K(1+\widetilde M_N).
\end{align}
Thus the argument leading to
\eqref{eq:detailedBatch3-upper-recent} proves the desired upper bound with
\[
  \mathfrak U_N^{\rm cont}
  :=
  \bar C
  \left(
    1+\widetilde W_N^2+\widetilde M_N
  \right).
\]
Its uniform $L^q$ bound follows from
\eqref{eq:cont-WM-moments3}.

\paragraph{Lower bound at $t=2N$.}
For the lower bound, take the theorem's constant $C\geq K_0+2$, so $K\leq N$. Retain $A$ and set
$k=\lfloor k_0\log_LR\rfloor$, $B=N-k$, as in the discrete proof.  Define
old, middle, and new pieces of \eqref{eq:contTimeEvens3} by the
same index ranges, and put
\[
  \mathcal G_N^{\rm cont}
  :=
  \{
    \widetilde W_N+\widetilde M_N
    \leq L^{\mathfrak c_*}
  \}.
\]
Markov's inequality and \eqref{eq:cont-WM-moments3} give
\[
  \P\big((\mathcal G_N^{\rm cont})^c\big)
  \leq \bar C L^{-2q\mathfrak c_*} = \bar C L^{-q\alpha_0}.
\]

For $B<j\leq N$, write
$g_j:=\Sc^{2j,2N}\mathfrak f_j$.  The finite-time $H^1$ estimate and
\eqref{eq:cont-effective-source-size3} give
\begin{equation}\label{eq:cont-high-tail3}
  \|\Pi_{\Z^2\setminus B_R}g_j\|_{L^2}
  \leq
  \bar C R^{-1}L^{4k+2}.
\end{equation}
By the lower bound in \eqref{eq:cont-effective-source-size3}, the diagonal
terms $\|\Pi_{B_R}g_j\|_{L^2}^2$ are therefore bounded below by a fixed
positive constant once $k_0$ is sufficiently small.  For $B<w<r\leq N$,
one estimates as in 
\eqref{eq:offDiagonal3}:
\begin{align}
  |\langle g_w,\Pi_{B_R}g_r\rangle|
  &\leq
  |\langle\Sc^{2w,2r}\mathfrak f_w,\mathfrak f_r\rangle|
  +
  \bar C R^{-1}L^{4k+2}
  \nonumber\\
  &\leq
  \bar C\widetilde D_wL^{-\alpha_0(r-w)}
  +
  \bar C R^{-1}L^{4k+2}.
  \label{eq:cont-projected-correlation3}
\end{align}
Summing \eqref{eq:cont-projected-correlation3} on
$\mathcal G_N^{\rm cont}$ gives
\[
  \|\Pi_{B_R}\varphi^{\rm new}\|_{L^2}^2
  \geq
  \bar c k
  -
  \bar C k
    \left(
      L^{-\mathfrak c_*}
      +kR^{-1}L^{4k+2}
    \right).
\]

The middle/new interaction is treated similarly to \eqref{eq:cont-projected-correlation3}, now with
$A<w\leq B<r\leq N$.  
% Notice that the projector cannot simply be dropped:
% One writes the projected pairing as the unprojected correlation minus its
% high-frequency tail, as in
% \eqref{eq:cont-projected-correlation3}.  Geometric summation gives
\[
  |\langle\Pi_{B_R}\varphi^{\rm new},\varphi^{\rm mid}\rangle|
  \leq
  \bar C\log_LR
  \left(
    L^{-\mathfrak c_*}
    +kR^{-1}L^{4k+2}
  \right).
\]
Finally, the old/recent interaction is estimated as in
\eqref{eq:detailedBatch3-old-good}; the
fixed $L^2$ loss in the $H^1$ norm of an effective source is absorbed by
increasing $K_0$. Take $k_0=1/8$ as above. The tail errors are uniform in $R$ since
\[
  kR^{-1}L^{4k+2}\leq k_0 L^2(\log_LR)R^{-1/2}
  \leq k_0\zeta L^{2-\zeta/2}.
\]
Taking $\zeta\geq16$ and then $L_0$ large gives
\[
  \|\Pi_{B_R}\varphi_{2N}\|_{L^2}^2
  \geq \bar c\log_LR
  \qquad\text{on }\mathcal G_N^{\rm cont}.
\]

\paragraph{Passage to $t\in[2N,2N+2)$.}
Write
\[
  \varphi_t
  =
  \Sc^{2N,t}\varphi_{2N}
  +
  \eta_t,
  \qquad
  \eta_t
  :=
  \int_{2N}^{t}\Sc^{r,t}f\,\dif r,
  \qquad
  \|\eta_t\|_{L^2}\leq2F.
\]
For the past contribution $\Sc^{2N,t}\varphi_{2N}$, unprojected
correlations are unchanged. The old estimate gains a factor $\bar C L^{2s}$,
absorbed by increasing $K_0$. For $j>N-k$, the effective source has $H^1$
norm at most $\bar C L^2$ and undergoes at most $2k$ further shears, so
\eqref{eq:cont-high-tail3} with $L^{4k+2}$ remains valid at time $t$.
Thus, uniformly in $t$, the past contribution satisfies the upper bound and,
on $\mathcal G_N^{\rm cont}$, a lower bound $\bar c\log_LR$.  
% The bounded fresh term $\eta_t$ is harmless for the upper and lower bounds, 
% and increase $\zeta$ so that the resulting $O_F(1)$ loss is absorbed by
% $\log_LR$.  
The fresh term preserves the upper bound since $\|\eta_t\|_2\leq2F$
and $\log_LR\geq\zeta\geq1$. For the lower bound, increase $\zeta$ until
$\sqrt{\bar c\zeta}\geq4F$. The triangle inequality then gives
\[
  \|\Pi_{B_R}\varphi_t\|_2
  \geq\sqrt{\bar c\log_LR}-2F
  \geq\tfrac12\sqrt{\bar c\log_LR}
  \qquad\text{on }\mathcal G_N^{\rm cont}.
\] 

This completes the proof sketch.

\subsection{Statistically stationary driven passive scalars; Theorem \ref{thmA:batch}}\label{subsec:wrapup3}

We turn now to implications for statistically stationary passive scalars, starting with a more detailed description of stationarity. In this section we will work with both Settings \ref{set:discDet}, \ref{set:contDet} at the same time. 
Throughout, we consider the driven passive scalar equation 
\eqref{eq:PSA3}, allowing now for a general initial condition $\varphi_0\in H^{-s}_0$.

\paragraph{Even-time chain as a random dynamical system} 
We begin by recording a way of viewing the time-2 advancement $\varphi_{2n} \mapsto \varphi_{2(n+1)}$ as a \emph{random dynamical system}.  

Recall that $\Omega_0 = \T^2$ is viewed as a space of pairs of phase shifts $\vec{\omega} := (\omega^1, \omega^2)$, equipped with a probability $\P_0$. Given $\vec{\omega} \in \Omega_0$, let us abuse notation and write $\Phi_{\vec{\omega}}^{r,t}, 0 \leq r \leq t \leq 2$ for the Lagrangian flow from time $r$ to time $t$, corresponding to the usual velocity field $u(t, x), 0 \leq t < 2$, with phase shifts given by $\vec{\omega}$. Let $\mathcal{S}^{r,t}_{\vec{\omega}} g = g \circ \left(\Phi^{r,t}_{\vec{\omega}}\right)^{-1}$. 
We consider the time-2 evolution on passive scalars 
\begin{align}
  \mathcal{T}_{\vec{\omega}} g := \mathcal{S}_{\vec{\omega}}^{0,2} g + 
  \begin{cases}    
  f & \text{ in Setting \ref{set:discDet},} \\ 
  \int_0^2 \mathcal{S}^{\tau,2}_{\vec{\omega}} f \dif \tau & \text{ in Setting \ref{set:contDet}.}
  \end{cases}
\end{align}
Observe that, writing $\uo = (\vec{\omega}_i), \vec{\omega}_i = (\omega^1_i, \omega^2_i)$, $i \geq 1$, 
the time-$2n$ solution $\varphi_{2n}$ to \eqref{eq:PSA3} is given as 
\begin{align}
  \varphi_{2n} = \mathcal{T}_{\vec{\omega}_{n}} \circ \dots \circ \mathcal{T}_{\vec{\omega}_1} [\varphi_0] 
\end{align}
% (recalling $\varphi_0 = 0$). 
Since transport preserves spatial means and $f$ has mean zero,
$\mathcal{T}_{\vec{\omega}}$ maps $H^{-s}_0$ into itself. In this sense,
compositions of the (IID) random mappings $\mathcal{T}_{\vec{\omega}_i}$
give rise to the Markov chain $(\varphi_{2n})$ on $H^{-s}_0$.
\begin{definition}
  Let $s \in (0, 1)$ and let $\mu$ be a Borel probability measure on $H^{-s}_0$. We say that $\mu$ is \emph{stationary} if for any bounded Borel function $G : H^{-s}_0 \to \R$, one has that
  \[\int_{H^{-s}_0} \left(\int_{\Omega_0} G(\mathcal{T}_{\vec{\omega}}[\varphi]) \dif \P_0(\vec{\omega})\right)\dif \mu(\varphi) = \int_{H^{-s}_0} G(\varphi) \dif \mu(\varphi) \, . \]
Put another way, $\mu$ is stationary if $\mathcal{T}_{\vec{\omega}}[\varphi]$ has law $\mu$ when $\varphi\sim\mu$ and $\vec{\omega}\sim\P_0$ are independent.
\end{definition}

Turning attention to the left-limiting chain $(\varphi_{2n-})$ in Setting \ref{set:discDet}, define 
\begin{align}
  \hat{\mathcal{T}}_{\vec{\omega}} g := 
    \mathcal{S}^{0,2}_{\vec{\omega}} [g + f] \, . 
    % & \text{ in Setting \ref{set:discDet},} 
\end{align}
Then, the time-$2n$ left-limiting solution $\varphi_{2n-}$ is given as 
\begin{align}
  \varphi_{2n-} = \hat{\mathcal{T}}_{\vec{\omega}_n} \circ \dots \circ \hat{\mathcal{T}}_{\vec{\omega}_1} [\varphi_0 - f] \, , 
\end{align}
having adjusted the time-zero scalar to account for the lack of scalar driving at time 0.
The adjusted initial state $\varphi_0-f$ lies in $H^{-s}_0$, and
$\hat{\mathcal{T}}_{\vec{\omega}}$ also preserves $H^{-s}_0$, so stationarity
for this chain is defined on the same state space.

% {\color{red}  

% conventions for the left-limiting chain: for Setting \ref{set:discDet} we set $\hat{ \mathcal{T}}_{\vec{\omega}} g := \mathcal{S}_{\vec{\omega}}^{0,2} [g + f]$. 

% write $\mu_-$ for corresponding stationary law of chain $(\varphi_{2n-})$ corresponding to left-limiting time-$2n$ process. 

% } 

\paragraph{Stationary states of the even-time chain}

We start by addressing existence and uniqueness of stationary states for the Markov chains $(\varphi_{2n}), (\varphi_{2n-})$. The following is an elaboration of Proposition \ref{prop:statStatesPSA} from Section \ref{sec:intro}. Recall below that we say a sequence of probability measures $(\mu_n)$ on $H^{-s}_0$ converges \emph{weakly} to a probability measure $\mu$ on $H^{-s}_0$ if
\begin{gather*}\int G(\varphi) \dif \mu_n(\varphi) \to \int G(\varphi) \dif \mu(\varphi) \\ \text{for all bounded continuous } G : H^{-s}_0 \to \R \,. \end{gather*}

% {\color{red}  modify to include left-limiting law } 
\begin{proposition}\label{prop:statStates3}
  Let $L$ be sufficiently large and let $s > 0$ be as in Theorem \ref{thm:neg-reg-flexible}. 
  \begin{itemize}
    \item[(a)] There exists a unique stationary measure $\mu$ (resp. $\mu_-$) on $H^{-s}_0$ for the chain $(\varphi_{2n})$ (resp. $(\varphi_{2n-})$).
    \item[(b)] For any fixed initial $\varphi_0 \in H^{-s}_0$, the law of $\varphi_{2n}$ (resp. $\varphi_{2n-}$) converges weakly as $n \to \infty$ to $\mu$ (resp. $\mu_-$).
  \end{itemize}
  
\end{proposition}
We prove Proposition \ref{prop:statStates3} at the end of Section \ref{subsec:wrapup3}. For now we complete the proof of Theorem \ref{thmA:batch}. 

\begin{proof}[Proof of Theorem \ref{thmA:batch}]
  We give the proof in Setting \ref{set:discDet}; the proof in Setting
  \ref{set:contDet} is the same, with Theorem
  \ref{thm:contBatchDetailed3} in place of Theorem
  \ref{thm:detailedBatch3}.  
  
  Apply Theorem \ref{thm:detailedBatch3} with
  moment order $2q$, and let $(\mathfrak{U}_N)\subset L^{2q}(\Omega)\subset L^q(\Omega)$, $\mathcal{G}_N \subset \Omega$ be as constructed there. Denote the resulting radius exponent by $\zeta_q$; initially take
  \[
    L \geq L_0\, , \qquad 
    R\geq L^{\zeta_q} \, , 
    \qquad
    N\geq C\log_LR \, .
  \]

  Let $\nu_N$ be the law on $H^{-s}_0$ of the time-$2N$ solution initialized
  at $\varphi_0=0$.  Proposition \ref{prop:statStates3}(b) gives that $\nu_N$ converges weakly to $\mu$ as $N \to \infty$. 
  For fixed $R$, define
  \[
    Z_R(h):=\|\Pi_{B_R}h\|_{L^2}^2.
  \]
  Since $B_R\subset\mathbb Z^2$ is finite, the map
  $\Pi_{B_R}:H^{-s}_0\to L^2$ is continuous.  Hence
  $h\mapsto Z_R(h)^q$ is continuous, though unbounded, on $H^{-s}_0$.

  We first show that the theorem follows provided that
  % It is straightforward to check that the proof is complete
  \begin{equation}\label{eq:stat-batch-moment-limit}
    \int Z_R^q \dif \nu_N \to \int Z_R^q \dif \mu \,. 
    % \E\|\Pi_{B_R}\varphi_{\rm stat}\|_{L^2}^{2q}
    % =
    % \int_{H^{-s}}Z_R(h)^q\dif\mu(h)
    % =
    % \lim_{N\to\infty}\int_{H^{-s}}Z_R(h)^q\dif\nu_N(h)
    % =
    % \lim_{N\to\infty}\E Z_R(\varphi_{2N})^q.
  \end{equation}
  The desired upper bound on $\int \| \Pi_{B_R} \varphi \|_{L^2}^{2q} \dif \mu(\varphi)$ 
  follows on noting 
  % $q$ gives
  \begin{equation}\label{eq:stat-batch-finite-upper}
    \E Z_R(\varphi_{2N})^q
    \leq
    \E\big[\mathfrak U_N^q\big](\log_LR)^q
    \leq
    C(\log_LR)^q 
  \end{equation}
  uniformly in $N$, while for the lower bound we have 
  \begin{equation}\label{eq:stat-batch-finite-lower}
    \E Z_R(\varphi_{2N})^q
    \geq
    c^q(\log_LR)^q\P(\mathcal G_N)
    \geq
    \frac{c^q}{2}(\log_LR)^q.
  \end{equation}
  provided $\P(\mathcal{G}_N) \geq \frac12$, which is assured for $L$ large. 
  % On $\mathcal G_N$, the corresponding lower bound gives
  % $Z_R(\varphi_{2N})\geq c\log_LR$, while
  % \[
  %   \P\big(\mathcal G_N\big)
  %   \geq 1-CL^{-c}.
  % \]
  % After increasing $L_0$ so that the last probability is at least $1/2$,
  % \begin{equation}\label{eq:stat-batch-finite-lower}
  %   \E Z_R(\varphi_{2N})^q
  %   \geq
  %   c^q(\log_LR)^q\P(\mathcal G_N)
  %   \geq
  %   \frac{c^q}{2}(\log_LR)^q.
  % \end{equation}
  % Combining \eqref{eq:stat-batch-moment-limit}--\eqref{eq:stat-batch-finite-lower}
  % and renaming the constants proves the claim in Setting
  % \ref{set:discDet}.  Replacing $\mathfrak U_N,\mathcal G_N$ by
  % $\mathfrak U_N^{\rm cont},\mathcal G_N^{\rm cont}$ gives the identical
  % argument in Setting \ref{set:contDet}.

  It remains to check \eqref{eq:stat-batch-moment-limit}; for this we use the
  full strength of the $2q$-order moment bound on $\mathfrak U_N$. To start, note that for $M$ fixed we have 
  \[
    \int (Z_R^q\wedge M)\dif\nu_N
    \longrightarrow
    \int (Z_R^q\wedge M)\dif\mu
  \]
  by weak convergence. To justify exchanging $N \to \infty$ with the limit $M \to \infty$, we use that 
  \[
    \sup_{N\geq C\log_LR}
    \int Z_R^q\mathds{1}_{\{Z_R^q>M\}}\dif\nu_N
    \leq
    \frac{C(\log_LR)^{2q}}{M} \, , 
  \]
  which follows from Markov's inequality and the fact that $\mathfrak{U}_N \in L^{2q}$ with $\| \mathfrak{U}_N\|_{L^{2q}(\Omega)}$ bounded uniformly in $N$. Equation \eqref{eq:stat-batch-moment-limit} now follows from a standard argument. 
  % The upper
  % bound in Theorem \ref{thm:detailedBatch3} gives
  % \[
  %   Z_R(\varphi_{2N})
  %   \leq
  %   \mathfrak U_N\log_LR,
  %   \qquad
  %   \sup_N
  %   \|\mathfrak U_N\|_{L^{2q}(\Omega)}
  %   \leq C,
  % \]
  % and hence
  % \begin{equation}\label{eq:stat-batch-uniform-integrability}
  %   \sup_{N\geq C\log_LR}
  %   \int Z_R^{2q}\dif\nu_N
  %   =
  %   \sup_{N\geq C\log_LR}
  %   \E Z_R(\varphi_{2N})^{2q}
  %   \leq
  %   C(\log_LR)^{2q}.
  % \end{equation}
  % Thus the random variables $Z_R(\varphi_{2N})^q$, whose laws are the
  % pushforwards of $\nu_N$ under $Z_R^q$, are uniformly integrable.

  % For completeness, let $M>0$.  Since $Z_R^q\wedge M$ is bounded and
  % continuous on $H^{-s}$, \eqref{eq:stat-batch-weak-convergence} gives
  % \[
  %   \int (Z_R^q\wedge M)\dif\nu_N
  %   \longrightarrow
  %   \int (Z_R^q\wedge M)\dif\mu.
  % \]
  % Moreover, \eqref{eq:stat-batch-uniform-integrability} and Chebyshev's
  % inequality imply
  % \[
  %   \sup_{N\geq C\log_LR}
  %   \int Z_R^q\mathds{1}_{\{Z_R^q>M\}}\dif\nu_N
  %   \leq
  %   \frac{C(\log_LR)^{2q}}{M}.
  % \]
  % By weak convergence and the lower semicontinuity of $Z_R^{2q}$,
  % \[
  %   \int Z_R^{2q}\dif\mu
  %   \leq
  %   \liminf_{N\to\infty}\int Z_R^{2q}\dif\nu_N
  %   \leq C(\log_LR)^{2q},
  % \]
  % so the same tail estimate, up to the value of $C$, holds with $\mu$ in
  % place of $\nu_N$.  We may therefore first let $N\to\infty$ and then
  % $M\to\infty$ in the truncated convergence above, obtaining
  % \eqref{eq:stat-batch-moment-limit}.  This completes the proof in Setting
  % \ref{set:discDet}; as noted at the outset, the proof in Setting
  % \ref{set:contDet} is identical.

  The same argument applies at each fixed phase $r\in(0,2)$, using
  Corollary \ref{cor:offPhase3}(a) and the bounds uniform over
  $t\in[2N,2N+2)$. For the discrete left limits, take $t\uparrow2N+2$
  in those bounds and use Proposition \ref{prop:statStates3}(b).

  We also need a lower-tail estimate. Apply the detailed lower bound with
  moment order $1$, and fix its radius exponent $\zeta\geq1$, independently
  of $q$. Write $\ell=\log_LR$. For suitable $c_0,\rho>0$, the good events give
  \[
    \E[c_0\ell-Z_R(\varphi_{2N+r})]_+
    \leq C\ell L^{-\rho}.
  \]
  The integrand is bounded and continuous on $H^{-s}_0$, so weak convergence
  passes this estimate to $\mu_r$. The same endpoint argument gives it for
  $\mu_-$. Markov's inequality, with $c_0$ reduced by a factor of two, yields
  \begin{equation}\label{eq:stationary-batch-lower-tail3}
    \P\{Z_R(\varphi_{\rm stat,r})<c_0\log_LR\}
    \leq C L^{-\rho},\qquad R\geq L^\zeta,
  \end{equation}
  also with $\mu_r$ replaced by $\mu_-$. This proves the asserted lower tail
  and the lower moment bound for every $q\geq1$ on this common radius range.
  To extend the upper moment bound to the same range when $\zeta_q>\zeta$,
  use, for $L^\zeta\leq R\leq L^{\zeta_q}$,
  \[
    \E Z_R^q\leq \E Z_{L^{\zeta_q}}^q
    \leq C_q\zeta_q^q
    \leq C_q(\zeta_q/\zeta)^q(\log_LR)^q.
  \]
  Thus the introductory exponent $\zeta$ may be chosen independently of $q$.
\end{proof}

\subsubsection*{Batchelor's law for the off-phase chain}

Due to the 2-periodic structure of the Pierrehumbert flow $u$ we work with, the time-$t$ scalars $\varphi_t$ do not constitute a homogeneous Markov process -- this was the reason we have worked thus far with the time-$2n$ Markov chain. Since this may be of interest, we state here versions of Batchelor's law for the ``off-phase'' chains
\[\varphi_{2n + r} \, , \qquad r \in (0,2) \,. \]

First, we set down the natural time-stationary probabilistic law for $(\varphi_{2n + r})$. In Setting \ref{set:discDet}, for $r \in (0,2)$ we define the law $\mu_r$ to be the law in $H^{-s}_0$ of
\[\mathcal{S}^{0,r}_{\vec{\omega}} \varphi_{\rm stat} \, , \]
while in Setting \ref{set:contDet} we set $\mu_r$ to be the law in $H^{-s}_0$ of
\[\mathcal{S}^{0,r}_{\vec{\omega}} \varphi_{\rm stat} + \int_0^r \mathcal{S}_{\vec \omega}^{\tau, r} f \dif \tau \,. \]
where $\varphi_{\rm stat}\sim\mu$ and $\vec{\omega}\sim\P_0$ are sampled independently.

% The following is an immediate consequence of Proposition \ref{prop:statStates3} and Theorem \ref{thmA:batch}. 
\begin{corollary}\label{cor:offPhase3}
  Assume the setting and notation of Theorem \ref{thmA:batch}. Let $r \in (0,2)$ and $\varphi_0\in H^{-s}_0$.
  \begin{itemize}
    \item[(a)] The marginal laws $\operatorname{Law}(\varphi_{2n+r})$ converge weakly to $\mu_r$ as $n\to\infty$, as probability measures on $H^{-s}_0$.
    \item[(b)] If $\varphi_{{\rm stat}, r}$ is distributed like $\mu_r$, then 
        \[c (\log_L R)^q \leq \E \| \Pi_{B_R} \varphi_{{\rm stat},r} \|_{L^2}^{2q} \leq C (\log_L R)^q \,. \]
  \end{itemize}  
\end{corollary}

\begin{proof}[Proof sketch of Corollary \ref{cor:offPhase3}]
  In either setting, write $\mathcal T_{\vec\omega,r}\varphi$ for the
  affine evolution in the definition of $\mu_r$. For $G\in C_b(H^{-s}_0)$, let
  \[
    P_rG(\varphi):=\int_{\Omega_0}
      G(\mathcal T_{\vec\omega,r}\varphi)\,\dif\P_0(\vec\omega).
  \]
  Each affine map is continuous on $H^{-s}_0$, and dominated convergence
  therefore gives $P_rG\in C_b(H^{-s}_0)$.

  Let $\nu_{N, r}$ be the law of $\varphi_{2N+r}, N \geq 1$. It is now immediate from Proposition \ref{prop:statStates3}(b) that 
  \[\int G d \nu_{N, r} = \int P_r G d \nu_N \to \int P_r G d \mu = \int G d \mu_r \, , \]
  hence $\nu_{N, r} \to \mu_r$ weakly as $N \to \infty$. The proof of Batchelor's law as in part (b) follows from similar arguments to that in the $r = 0$ case.

  % $P_rG$ be the expectation of $G$ after applying the partial-cycle affine map defining $\mu_r$ (including the Duhamel term in Setting \ref{set:contDet}); the needed Feller property holds because the finite-time propagators act continuously on $H^{-s}$, and therefore, writing $\nu_N$ and $\nu_{N,r}$ for the laws of $\varphi_{2N}$ and $\varphi_{2N+r}$, independence of the next phase block and Proposition \ref{prop:statStates3}(b) give $\int G\dif\nu_{N,r}=\int P_rG\dif\nu_N\to\int P_rG\dif\mu=\int G\dif\mu_r$, proving (a).
  % The bounds in Theorems \ref{thm:detailedBatch3} and \ref{thm:contBatchDetailed3} are uniform for $t\in[2N,2N+2)$ and, when applied with moment order $2q$, give both the required finite-time upper and lower bounds for $Z_R(\varphi_{2N+r})^q$, where $Z_R(h)=\|\Pi_{B_R}h\|_{L^2}^2$, and uniform integrability of these variables; part (a) therefore upgrades to convergence of their expectations, yielding (b).
\end{proof}

\begin{remark}
{We caution that for $ r \in [0,2) \setminus \{ 0, 1\}$, the sequences $(\varphi_{2n +r})$ \emph{need not be} Markovian since, for $n \geq 0$, $\varphi_{2(n+1)+r}$ and $\varphi_{2n + r}$ share dependence on the phase shifts $\vec{\omega}_{n + 1}$ applied along $t \in [2n, 2n+2]$. 
  We can, however, recover the Markov property for the extended chain $((\vec{\omega}_{n+1}, \varphi_{2n+r}))$; augmenting by the IID sequence $(\vec{\omega}_{n+1})$ provides the necessary data to evolve $\varphi_{2n+r}$ all the way to $\varphi_{2(n + 1) + r}$. 
  For the augmented chain, a version of the above arguments implies unique existence of a stationary measure $\tilde{\mu}_r$ for the augmented chain on $\Omega_0 \times H^{-s}_0$, such that
  \begin{align}
    \tilde{\mu}_r(\Omega_0 \times K) & = \mu_r(K) \qquad \text{ for all Borel } \qquad K \subset H^{-s}_0 \, , \text{ and} \\
    \tilde{\mu}_r(E \times H^{-s}_0) & = \P_0(E) \qquad \text{ for all measurable } \qquad E \subset \Omega_0 \,.
  \end{align}
  Here, $\mu_r$ is as in Corollary \ref{cor:offPhase3}. 
  
  We note that for $r = 1$, however, the odd-time chain $(\varphi_{2n + 1})$ does satisfy the Markov property, since here we can cleanly separate the phase shifts: $\varphi_{2(n+1) + 1}$ depends only on $\omega^2_{n + 1}, \omega^1_{n + 2}$ and $\varphi_{2n+1}$, while $\varphi_{2n+1}$ is independent of $\omega^2_{n + 1}, \omega^1_{n + 2}$. 
  } 
  % \[\varphi_{2(n+1)+r} = \mathcal{S}^{2(n+1), 2(n+1) + r} \bigg( \mathcal{S}^{2n+r, 2(n+1)} \varphi_{2n + r}\bigg) \]

  % when $r \in (0,1)$, both $\varphi_{2n +r}$ and $\varphi_{2(n+1)+r}$ depend on $\omega^1_{n + 1}$, while when $r \in (1,2)$ these both depend on $\omega^2_{n + 1}$. For this reason we stop short of referring to $\mu_r$ as a `stationary measure' for $(\varphi_{2n+r})$. 
  % Nevertheless, convergence and Batchelor's law as in Corollary \ref{cor:offPhase3}. 
\end{remark}

    % can get a markov chain for $(\varphi_{2n+1})_{n \geq 0}$. technical point: sequence $(\varphi_{2n +r}), r \in [0,2) \setminus \{ 0, 1\}$ are not actually Markovian. 

\subsubsection*{Proof of Proposition \ref{prop:statStates3}}

% {\color{red}  modify to include left-limiting law.  } 
% \begin{proof}[Proof of Proposition \ref{prop:statStates3}]
We address below the case of the right-limiting chain $(\varphi_{2n})$. 
Statements involving the left-limiting chain $\varphi_{2n-}$ and its stationary measure $\mu_-$ in Setting \ref{set:discDet} follow mutatis mutandis.

  It will be convenient to pass to the two-sided phase space
  \[
    (\hat\Omega,\hat\Fc,\hat\P)
    :=(\Omega_0,\Fc_0,\P_0)^{\Z},
  \]
  and write \(\vec\omega_j=(\omega_j^1,\omega_j^2)\) for the phase block on
  the interval \([2(j-1),2j)\).  Let $\hat \E$ denote expectation w.r.t. $\hat \P$. The velocity field, Lagrangian flow, and
  propagators then extend in the evident way to all \(t\in\R\).  Set
  \[
    b_j:=
    \begin{cases}
      f,
        &\text{in Setting \ref{set:discDet}},\\[1mm]
      \displaystyle
      \int_{2(j-1)}^{2j}\Sc_{\hat\uo}^{\tau,2j}f\,\dif\tau,
        &\text{in Setting \ref{set:contDet}}.
    \end{cases}
  \]
  Each $b_j$ has zero spatial mean, and
  \begin{equation}\label{eq:stationary-increment-bound3}
  \sup_{j\in\Z}\|b_j\|_{H^{-s}}\leq C_L<\infty
  \qquad\text{almost surely}
\end{equation}
  in either setting. Here and below, $C_L$ denotes a finite constant that may depend on $L$ and change from line to line.

  % Thus the evolution over the \(j\)-th phase block is
  % \[
  %   h\longmapsto
  %   \Sc_{\hat\uo}^{2(j-1),2j}h+b_j.
  % \]
  % The same deterministic finite-time composition estimate used in
  % \eqref{eq:deterministFracTime2} gives
  % \begin{equation}\label{eq:stationary-increment-bound3}
  %   \sup_{j\in\Z}\|b_j\|_{H^{-s}}\leq C_L
  %   \qquad\text{almost surely}
  % \end{equation}
  % in either setting.

  To construct a stationary state, define
  \begin{align}\label{eq:stationary-series3}
    \varphi_{\rm stat}
    & :=\sum_{j=-\infty}^{0}\Sc_{\hat\uo}^{2j,0}b_j = 
    \begin{cases}
      \sum_{j=-\infty}^{0}\Sc_{\hat\uo}^{2j,0}f \, ,  & \text{ Setting \ref{set:discDet}} \\ 
    \int_{-\infty}^{0}\Sc_{\hat\uo}^{\tau,0}f\,\dif\tau \, ,
        & \text{ Setting \ref{set:contDet}}
    \end{cases}
  \end{align}
  Ultimately, we will set the stationary measure $\mu$ to be the law of $\varphi_{\rm stat}$. 

  We start by checking that in both settings, the series defining $\varphi_{\rm stat}$ converges $\hat \P$-almost-surely in $H^{-s}_0$.
  For \(j<0\), the random field \(b_j\) depends only on
  \(\vec\omega_j\), whereas \(\Sc_{\hat\uo}^{2j,0}\) depends only on the
  later phase blocks $\vec{\omega}_{j+1}, \dots, \vec{\omega}_0$.  Conditioning on \(b_j\), applying Theorem
  \ref{thm:neg-reg-flexible} at time \(-2j\), and then using
  \eqref{eq:stationary-increment-bound3}, we obtain
  \[
    \left(
      \hat\E\|\Sc_{\hat\uo}^{2j,0}b_j\|_{H^{-s}}^2
    \right)^{1/2}
    \leq C_L L^{2j\mathfrak c}.
  \]
  The right-hand side is summable over \(j\leq 0\).  Hence the series in
  \eqref{eq:stationary-series3} converges in
  \(L^2(\hat\Omega;H^{-s}_0)\); moreover, the sum of the expected norms of
  its terms is finite, so it converges absolutely in \(H^{-s}_0\) almost
  surely. Here the limit remains in $H^{-s}_0$ because every summand has
  zero mean and $H^{-s}_0$ is closed in $H^{-s}$.

  % In Setting \ref{set:discDet} this is precisely
  % \[
  %   \varphi_{\rm stat}
  %   =\sum_{j=-\infty}^{0}\Sc_{\hat\uo}^{2j,0}f,
  % \]
  % while in Setting \ref{set:contDet}, splitting the time integral into
  % two-unit blocks gives
  % \[
  %   \varphi_{\rm stat}
  %   =\int_{-\infty}^{0}\Sc_{\hat\uo}^{\tau,0}f\,\dif\tau.
  % \]

  Let \(\hat\sigma : \hat \Omega \to \hat \Omega\) be the left shift of the two-sided phase sequence. It is straightforward to check from \eqref{eq:stationary-series3} that
  % The cocycle property and \eqref{eq:stationary-series3} give
  \[
    \mathcal T_{\vec\omega_1}
      [\varphi_{\rm stat}(\hat\uo)]
    =\varphi_{\rm stat}(\hat\sigma\hat\uo).
  \]
  Since \(\varphi_{\rm stat}\) depends only on the phase blocks with
  indices at most zero, it is independent of the fresh block
  \(\vec\omega_1\); since \(\hat\P\) is invariant under \(\hat\sigma\),
  the law \(\mu\) of \(\varphi_{\rm stat}\) is stationary.  This proves
  existence in part (a).

  It remains to prove uniqueness and attraction. The argument we present is as in {\cite{bedrossian2024negative}, Theorem 2.6.}  Equip \(H^{-s}_0\) with
  the bounded metric
  \[
    d_s(g,h):=1\wedge\|g-h\|_{H^{-s}}.
  \]
  Given \(g,h\in H^{-s}_0\) and a fixed phase shift sequence $\uo = (\vec{\omega}_1, \vec{\omega}_2, \dots)$, let
  \[g_n = \mathcal{T}_{\vec{\omega}_n} \circ \dots \circ \mathcal{T}_{\vec{\omega}_1} g \, , \qquad h_n = \mathcal{T}_{\vec{\omega}_n} \circ \dots \circ \mathcal{T}_{\vec{\omega}_1} h\]
  and observe that $g_n-h_n=\mathcal{S}^{0,2n}_{\uo}(g-h)$.
  Since $g-h\in H^{-s}_0$, Theorem \ref{thm:neg-reg-flexible} applies and gives
  \begin{equation}\label{eq:stationary-coupling-decay3}
    \E d_s(g_n,h_n)
    \leq
    1\wedge
    \left(
      C L^{\mathfrak d-2n\mathfrak c}
      \|g-h\|_{H^{-s}}
    \right)
    \longrightarrow0.
  \end{equation}

  If \(\mu_1\) and \(\mu_2\) are stationary laws on $H^{-s}_0$, choose
  \(g\sim\mu_1\) and \(h\sim\mu_2\), independently of one another and
  of the future phase blocks $(\vec{\omega}_j)_{j \geq 1}$. 
  Note in our setup that $g_n \sim \mu_1$ and $h_n \sim \mu_2$ for all $n$. 
  Integrating
  \eqref{eq:stationary-coupling-decay3} against
  \(\mu_1(\dif g)\mu_2(\dif h)\) and applying dominated convergence gives
  \[
    \E d_s(g_n,h_n)\longrightarrow0.
  \]
  It follows that for every \(G : H^{-s}_0 \to \R\) that is one-Lipschitz with respect to \(d_s\),
  \[
    \left|\int G\,\dif\mu_1-\int G\,\dif\mu_2\right|
    \leq \E d_s(g_n,h_n)\longrightarrow0.
  \]
  Bounded Lipschitz functions determine Borel probability measures on
  \(H^{-s}_0\), and hence \(\mu_1=\mu_2=\mu\).
  An identical argument, with $g \sim \mu$ and $h \in H^{-s}_0$ arbitrary, allows to conclude part (b).

\section{Sparsity}\label{sec:sparsity}
    %!TEX root = main.tex

We now turn our attention to sparsity in the cumulative Batchelor's law (Theorem \ref{thmA:sparsity}, restated below). Section \ref{subsec:discSparsity4} addresses  discrete-time scalar sources (Setting \ref{set:generalDiscDet} from Section \ref{sec:batchLaw}) while Section \ref{subsec:contSparsity4} addresses  continuous-time scalar source (Setting \ref{set:contDet}). Proofs of theorems from Section \ref{sec:intro} are deferred to the end of Section \ref{subsec:contSparsity4}. The finite-time estimates below use $\varphi_0=0$, as in Section \ref{sec:batchLaw}. 

\subsection{Discrete time case}\label{subsec:discSparsity4}

Recall that for $\theta > 0$, \[\Cc_\theta = \left\{(k,j) \in \Z^2 : |j| < \theta |k|\right\} \]
is the set of roughly-horizontal wavenumbers. 
Our task in this section is to prove the following statement on cone sparsity for the even-time left-limiting process $(\varphi_{2n-})$. 

% {\color{red}  bug: what to do if $L \leq R$? } 
\begin{theorem}\label{thm:batchelor-cone-quenched}
  Assume Setting \ref{set:generalDiscDet}. Fix $q \geq 1, \vartheta \in [0,1)$, and set 
  \begin{align}\label{eq:paramAssumpt4} \eta = \frac{1 - \vartheta}{8q} \, , \qquad \Theta = L^{\vartheta} \,.  \end{align}

  Then, there exist constants $C, L_0$, depending on $q, \vartheta, \eta, p, F, \psi$, such that for all $L \geq L_0$ the following holds. 
  
  There is a family of random variables $\mathfrak{C}_{N, R, \vartheta} : \Omega \to [1,\infty), N \geq1, R\geq2$, such that 
  \[\| \mathfrak{C}_{N, R, \vartheta} \|_{L^q(\Omega)} \leq C \, , \]
  and for any $R \geq 2, N \geq 1$, one has the estimate 
    \begin{align}\label{eq:lefLimEst4}\| \Pi_{B_R \cap \Cc_\Theta} \varphi_{2N-} \|_{L^2}^2 & \leq \mathfrak{C}_{N, R, \vartheta} L^{- \eta} \log_L R\, . 
    \end{align}
    % \item[(b)] For all $t \in [2N, 2N+2)$, one has 
    % \begin{align}\label{eq:allTimeEst4}
    % \| \Pi_{B_R \cap \Cc_\Theta} \varphi_{t} \|_{L^2}^2 & \leq \mathfrak{C}_{N, \vartheta} L^{- \eta} \log_L R + C \,. 
    % \end{align}
  % \end{itemize}
%   \[ \|\Pi_{\Cc_\Theta\cap B_R}\varphi_{2N}\|_{L^2}^2
%     \leq
%     \mathfrak{C}_N \left(1+L^{-\eta}\log_LR\right).
% \]
\end{theorem}
% Theorem \ref{thmA:sparsity} in Setting \ref{set:discDet} follows immediately on taking $q$-th moments of the left- and right-hand sides of \eqref{eq:lefLimEst4}, \eqref{eq:allTimeEst4} respectively. 

% Theorem \ref{thmA:sparsity} and Corollary \ref{corA:sparsityDisc} will follow relatively quickly from the above more general statement -- see the end of Section \ref{subsec:discSparsity4}. 

\begin{proof}[Proof of Theorem~\ref{thm:batchelor-cone-quenched}]
% The estimate in (b) is immediate from that in (a) on noting that $\varphi_t, t \in [2N, 2N+2)$ incorporates the instantaneous driving term $f_{N}$ applied at time $2N$, giving a contribution of size $\lesssim F^2$, with $F$ as in Setting \ref{set:generalDiscDet}. 

Let $s,\mathfrak c_*,\alpha=2\mathfrak c_*$ and $D_j,W_N,M_N$ be as
in the proof of Theorem \ref{thm:detailedBatch3}. In particular,
\eqref{eq:detailedBatch3-shifted-mixing} holds and $W_N,M_N$ have uniformly
bounded $L^{2q}(\Omega)$ norms. Write
\[
  a:=q\eta=\frac{1-\vartheta}{8},\qquad \ell:=\log_LR,\qquad b:=2+2s,
\]
and set
\[
  K:=\left\lceil K_0\ell+\frac{b+\eta}{\mathfrak c_*}\right\rceil,
  \qquad A:=(N-K)_+,\qquad \mathfrak c_*K_0\geq s.
\]
Thus $K\leq C(1+\ell)$. The case $N=1$ is immediate since
$\varphi_{2-}=0$; assume $N\geq2$. Decompose
\[
  \varphi^{\rm old}:=\sum_{j=1}^A\mathcal S^{2j,2N}f_j,
  \qquad \varphi^{\rm rec}:=\sum_{j=A+1}^{N-1}\mathcal S^{2j,2N}f_j.
\]
Empty sums are zero. The old-term estimate from
\eqref{eq:detailedBatch3-upper-old} gives
\begin{equation}\label{eq:sparsity-old4}
  \|\Pi_{\Cc_\Theta\cap B_R}\varphi^{\rm old}\|_2
  \leq C W_N R^s L^{b-\mathfrak c_*K}
  \leq C W_N L^{-\eta}.
\end{equation}
Here the harmless factor $L^b$ leaves room for the source and partial-cycle
losses needed below. In particular, the $2q$-th moment of this norm is at most
$CL^{-2a}$.

For the recent contribution, use the two-shear factorization to write
\[
  \varphi^{\rm rec}=\mathcal L_{\omega_N^2,L}g_N,
  \qquad
  g_N=\mathcal L_{\omega_N^1,L}
       \sum_{j=A+1}^{N-1}\mathcal S^{2j,2(N-1)}f_j.
\]
The recent-source estimate \eqref{eq:detailedBatch3-upper-recent} gives
\[
  \|g_N\|_2^2
  =\left\|\sum_{j=A+1}^{N-1}\mathcal S^{2j,2(N-1)}f_j\right\|_2^2
  \leq CK(1+L^{-\alpha}M_{N-1}),
  \qquad \E\|g_N\|_2^{2q}\leq CK^q.
\]
Let $\widehat{\mathcal F}$ be generated by the phases through $\omega_N^1$,
excluding $\omega_N^2$, and put
$Y:=\|\Pi_{\Cc_\Theta}\mathcal L_{\omega_N^2,L}g_N\|_2^2$.
Since $g_N$ is $\widehat{\mathcal F}$-measurable, $Y\leq\|g_N\|_2^2$,
and Lemma \ref{lem:cone-decay} with $s=0$ gives
\begin{equation}\label{eq:conditional-cone4}
  \E[Y^q\mid\widehat{\mathcal F}]
  \leq \|g_N\|_2^{2q-2}\E[Y\mid\widehat{\mathcal F}]
  \leq C L^{-2a}\|g_N\|_2^{2q}.
\end{equation}
When $g_N=0$, both sides are zero. Combining the two contributions yields
\begin{equation}\label{eq:sparsity-strong-moment4}
  \E\|\Pi_{\Cc_\Theta\cap B_R}\varphi_{2N-}\|_2^{2q}
  \leq C L^{-2a}(1+\ell)^q.
\end{equation}
To obtain the claimed normalization even for $2\leq R<L$, observe that
$\ell\geq\log2/\log L$ and hence
\begin{equation}\label{eq:sparsity-small-radius4}
  L^{-a}\left(\frac{1+\ell}{\ell}\right)^q
  \leq L^{-a}\left(1+\frac{\log L}{\log2}\right)^q\leq C.
\end{equation}
The random variables
\[
  \mathfrak C_{N,R,\vartheta}
  :=1+\frac{L^\eta}{\ell}
      \|\Pi_{\Cc_\Theta\cap B_R}\varphi_{2N-}\|_2^2
\]
therefore have uniformly bounded $L^q$ norms, since $q\eta=a$,
and satisfy \eqref{eq:lefLimEst4}. 

\end{proof}

\subsection{Sparsity, continuous time}\label{subsec:contSparsity4}

We now turn attention to the case of continuous-time forcing (Setting \ref{set:contDet}). The following amounts to the full statement of Theorem \ref{thmA:contSparsity1} stated in Section \ref{sec:intro}. Along the way, we introduce an idea used in the proof of Theorem \ref{thmA:sparsity}, to be addressed at the end of Section \ref{subsec:contSparsity4}.

% After the following theorem, we will conclude by proving the sparsity theorems of Section \ref{sec:intro}.

\begin{theorem}[Constant forcing, continuous-time sparsity estimate]
\label{thm:cont-sparsity}

Assume Setting \ref{set:contDet}. Fix $q \geq 1, \vartheta \in [0,1)$ and let $\eta, \Theta$ be as in \eqref{eq:paramAssumpt4}. Fix the additional constant $\varepsilon_0 \in (0,1)$. 
Then, there exist constants $C, L_0$, depending on $q, \vartheta, \eta, p, F, \psi, \varepsilon_0$, such that for all $L \geq L_0$ the following holds. 

There is a family of random variables $\mathfrak{C}^{\rm cont}_{N, R, \vartheta, r} : \Omega \to [1,\infty), N \geq 1, R\geq2, r \in [1 + \varepsilon_0, 2]$ satisfying the uniform estimates
\[\| \mathfrak{C}^{\rm cont}_{N, R, \vartheta, r}\|_{L^q(\Omega)} \leq C  \]
such that, for any $R \geq 2, N \geq 1$, one has 
\begin{align}\label{eq:coneEstTimeNCont4}
  \| \Pi_{\mathcal{C}_\Theta \cap B_R} \varphi_{2N + r}\|_{L^2}^2 \leq \mathfrak{C}^{\rm cont}_{N, R, \vartheta, r} L^{- \eta} \log_L R
\end{align}

  % Under \eqref{eq:cont-force-assumptions}, and under the exponent choices of
  % Subsection~\ref{subsec:cont-prelim}, let
  % \[
  %   0<\eta<\frac1{4(2q+1)}
  % \]
  % and let $1\leq\Theta\leq L^{1-4(2q+1)\eta}$.  There are constants
  % \[
  %   C=C(\psi,p,F,\varepsilon,\delta,s,\sigma,q,a,\eta)>0,
  %   \qquad L_0,\zeta>0,
  % \]
  % such that for all $L\geq L_0$, all $R\geq L^\zeta$, and all integers
  % $N\geq C\log_LR$, there is a nonnegative random variable
  % $\mathfrak C_{N,R,\Theta}^{\rm cont}(\uo)$ with
  % \[
  %   \|\mathfrak C_{N,R,\Theta}^{\rm cont}\|_{L^q(\Omega)}\leq C
  % \]
  % such that, almost surely,
  % \begin{equation}\label{eq:cont-sparsity-quenched}
  %   \|\Pi_{\Cc_\Theta\cap B_R}\varphi_{2N}\|_{L^2}^2
  %   \leq
  %   C\left(1+\mathfrak C_{N,R,\Theta}^{\rm cont}(\uo)
  %     L^{-\eta}\log_LR\right).
  % \end{equation}
\end{theorem}

As in Remark \ref{rmk:phaseConfinement1}(a), the corresponding estimate
for $r\in[\varepsilon_0,1]$ follows on reversing the axes. The passage to
stationary laws is given at the end of this section.

% {\color{red} 
%  similar estimate for phases $r \in [\varepsilon_0, 1]$ reversing the axes. no similar estimate in $[0,\varepsilon_0)$ or $[1, 1+ \varepsilon_0]$. 

% may scale $\varepsilon_0$ with a small fractional power of $L$ at the expense of $\eta$; details omitted. 
% } 

The old contribution is handled as above; the recent contribution requires separating the forcing injected during the final horizontal shear. 
\begin{proof}[Proof of Theorem~\ref{thm:cont-sparsity}]
Let $\mathfrak c_*:=\alpha_0/2$, with $\alpha_0$ as in
Lemma \ref{lem:cont-effective-sources3}. Use $a,\ell,b,K,A$ as in the discrete
proof, with this value of $\mathfrak c_*$. Write $t=2N+r$, $t_1=2N+1$,
and $r_0=r-1\in[\varepsilon_0,1]$. Decompose
\[
  \varphi^{\rm old}:=\int_0^{2A}\mathcal S^{\tau,t}f\,\dif\tau,
  \qquad \varphi^{\rm rec}:=\varphi_t-\varphi^{\rm old}.
\]
The effective sources have $H^1$ norm at most $CL^2$, and the final
partial cycle has $H^{-s}$ operator norm at most $CL^{2s}$.
Since $\widetilde D_j\geq D_j$ and $\alpha\geq\alpha_0$, the old-term bound is
\[
  \|\Pi_{\Cc_\Theta\cap B_R}\varphi^{\rm old}\|_2
  \leq C\widetilde W_N R^sL^{2+2s-\mathfrak c_*K}
  \leq C\widetilde W_N L^{-\eta}.
\]
Its $2q$-th moment is at most $CL^{-2a}$ by
\eqref{eq:cont-WM-moments3}.

For the recent contribution, define
\[
  h_{N,R}:=\int_{2A}^{t_1}\mathcal S^{\tau,t_1}f\,\dif\tau,
  \qquad
  \eta_t^{\rm h}:=\int_{t_1}^t\mathcal S^{\tau,t}f\,\dif\tau.
\]
Then
\[
  \varphi^{\rm rec}=\mathcal S^{t_1,t}h_{N,R}+\eta_t^{\rm h}.
\]
The input $h_{N,R}$ is independent of $\omega_{N+1}^2$.
By \eqref{eq:recContribContinuous3}, the $L^2$ isometry on $[2N,t_1]$,
and the bound $\|\int_{2N}^{t_1}\mathcal S^{\tau,t_1}f\,\dif\tau\|_2\leq F$,
\[
  \|h_{N,R}\|_2^2\leq C\bigl(1+K(1+\widetilde M_N)\bigr),
  \qquad \E\|h_{N,R}\|_2^{2q}\leq C(1+K)^q.
\]
To apply the transfer-operator notation precisely, put
$Jg(x^1,x^2):=g(x^2,x^1)$ and
$H_{\omega,\lambda}:=\mathcal L_{\omega,\lambda}J$.
Thus $H_{\omega,\lambda}$ is the physical horizontal-shear operator, and
$\mathcal S^{t_1,t}h_{N,R}=H_{\omega_{N+1}^2,r_0L}h_{N,R}$.
Apply \eqref{eq:conditional-cone4} to $Jh_{N,R}$ and amplitude $r_0L$
(taking $L_0\geq\varepsilon_0^{-1}$). This gives
\[
  \E\|\Pi_{\Cc_\Theta}\mathcal S^{t_1,t}h_{N,R}\|_2^{2q}
  \leq C_{\varepsilon_0}L^{-2a}(1+K)^q.
\]
For the fresh forcing, write
\[
  \eta_t^{\rm h}=\int_0^{r_0}H_{\omega_{N+1}^2,\rho L}f\,\dif\rho.
\]
Lemma \ref{lem:cone-decay} and the $L^2$ isometry imply, for $0<\rho\leq1$,
\[
  \E\|\Pi_{\Cc_\Theta}H_{\omega,\rho L}f\|_2^{2q}
  \leq C\min\{1,L^{-2a}\rho^{-1/4}\}\|f\|_2^{2q}.
\]
For $\rho L<1$, use only the isometry; the same bound holds since
$\vartheta\geq0$. Jensen's inequality in $\rho$ now gives
\[
  \E\|\Pi_{\Cc_\Theta}\eta_t^{\rm h}\|_2^{2q}
  \leq r_0^{2q-1}\int_0^{r_0}
       \E\|\Pi_{\Cc_\Theta}H_{\omega_{N+1}^2,\rho L}f\|_2^{2q}\,\dif\rho
  \leq C L^{-2a},
\]
since $\rho^{-1/4}$ is integrable. No independence between forcing times
is needed. Altogether,
\begin{equation}\label{eq:continuous-sparsity-strong4}
  \E\|\Pi_{\Cc_\Theta\cap B_R}\varphi_{2N+r}\|_2^{2q}
  \leq C L^{-2a}(1+\ell)^q.
\end{equation}
Use \eqref{eq:sparsity-small-radius4} and set
\[
  \mathfrak C^{\rm cont}_{N,R,\vartheta,r}
  :=1+\frac{L^\eta}{\ell}
       \|\Pi_{\Cc_\Theta\cap B_R}\varphi_{2N+r}\|_2^2
\]
to conclude the proof.

\end{proof}

% \begin{corollary}[constant forcing, averaged continuous-time sparsity]
% \label{cor:cont-sparsity-averaged}
%   Under the hypotheses of Theorem~\ref{thm:cont-sparsity},
%   \begin{equation}\label{eq:cont-sparsity}
%     \E\|\Pi_{\Cc_\Theta\cap B_R}\varphi_{2N}\|_{L^2}^2
%     \leq
%     C\left(1+L^{-\eta}\log_LR\right).
%   \end{equation}
% \end{corollary}

% \begin{proof}
%   This follows by taking expectations in
%   \eqref{eq:cont-sparsity-quenched} and using
%   $\|\mathfrak C_{N,R,\Theta}^{\rm cont}\|_{L^q(\Omega)}\leq C$.
% \end{proof}

\begin{proof}[Proof of Theorems \ref{thmA:sparsity} and \ref{thmA:contSparsity1}]
First consider zero initial data. For discrete forcing and
$r\in[1+\varepsilon_0,2)$, use the cutoff above at time $2N$, with the
recent sum now running through $j=N$. Transport this sum to $t_1=2N+1$;
its $L^2$ norm has $2q$-th moment at most $CK^q$, and it is independent of
$\omega_{N+1}^2$. The old contribution gains at most $CL^{2s}$, already
allowed for in \eqref{eq:sparsity-old4}. The conditional estimate
\eqref{eq:conditional-cone4} therefore proves
\eqref{eq:sparsity-strong-moment4} also for $\varphi_{2N+r}$.

For fixed $R$, the function
\[
  h\longmapsto \min\{M,\|\Pi_{\Cc_\Theta\cap B_R}h\|_2^{2q}\}
\]
is bounded and continuous on $H^{-s}_0$. Apply Proposition
\ref{prop:statStates3}(b) and Corollary \ref{cor:offPhase3}(a) to the
zero-initial-data process, then let $M\to\infty$. This passes
\eqref{eq:sparsity-strong-moment4} and
\eqref{eq:continuous-sparsity-strong4} to the respective stationary laws.
For the discrete left limits, the initial chain state is $-f$, as in
\eqref{eq:PSA3} and the formula for $\hat{\mathcal T}$ above; for continuous
forcing at $r=2$, the limiting law is $\mu_2=\mu_0=\mu$.
Equation \eqref{eq:sparsity-small-radius4} gives the stated estimates.
Finally, stationarity and uniqueness give the coupling
$\varphi_{\rm stat,0}=\varphi_{\rm stat,-}+f$, which proves
Theorem \ref{thmA:sparsity}(b) by the triangle inequality.
The versions with the flipped cone follow by exchanging the coordinates.
\end{proof}

\begin{proof}[Proof of Corollaries \ref{corA:sparsityDisc} and \ref{corA:contSparsityRatio1}]
Put $\ell=\log_LR\geq\zeta$, taking $\zeta\geq1$. Set
$a=(1-\vartheta)/16$ in Setting \ref{set:discDet} and
$a=(1-\vartheta)/8$ in Setting \ref{set:contDet}.
For discrete forcing, apply Theorem \ref{thmA:sparsity}(a) with $q=1$.
For continuous forcing, use the stationary version of
\eqref{eq:continuous-sparsity-strong4} just proved, with $q=1$ and
$1+\ell\leq2\ell$.
In either case,
\[
  \E X\leq CL^{-2a}\ell,
  \qquad X:=\|\Pi_{\Cc_\Theta\cap B_R}\varphi_{\rm stat,+}\|_2^2,
\]
for each good phase law in either setting. Thus
$\P\{X>L^{-a}\ell\}\leq CL^{-a}$.
The lower-tail estimate \eqref{eq:batchelorLowerTailIntro} gives
$\|\Pi_{B_R}\varphi_{\rm stat,+}\|_2^2\geq c_0\ell$
except on an event of probability at most $CL^{-\rho}$, for suitable
$c_0,\rho>0$.
Combining these events proves the tight ratio bounds.
For the right-limiting discrete law, use the coupling above to obtain
\[
  \|\Pi_{\Cc_\Theta\cap B_R}\varphi_{\rm stat,0}\|_2^2
  \leq2\|\Pi_{\Cc_\Theta\cap B_R}\varphi_{\rm stat,-}\|_2^2+2\|f\|_2^2,
\]
and apply the lower-tail estimate to the denominator with law $\mu_0$.
This gives the additional $C/\ell$ in Corollary \ref{corA:sparsityDisc}(b).
Choosing $0<c<\min\{a,\rho\}$ and enlarging $C=c^{-1}$ and $L_0$
absorbs the constants in the failure probability, yielding $1-L^{-c}$.
\end{proof}

% \begin{proof}[Proof of Theorem \ref{thmA:contSparsity1}]

% \end{proof}

% \section{Batchelor's Law and Sparsity in Continuous Time}\label{sec:contBatchLaw}
%     \input{contTimeBatch.tex}

\bibliographystyle{abbrv}
\bibliography{bibliography}

\end{document}